\documentclass[10pt,reqno]{amsart}
\usepackage{amsmath}
\usepackage{amsthm}
\usepackage{amssymb}
\usepackage[a4paper, hmarginratio=1:1]{geometry}
\usepackage{graphicx}
\usepackage{amssymb}
\usepackage{epstopdf}
\usepackage{microtype}
\usepackage[colorlinks=true, pdfborder={0 0 0}]{hyperref}

\allowdisplaybreaks[4]

\title[Hyper-ideal circle patterns]{Hyper-ideal circle patterns with cone singularities}

\author[Nikolay Dimitrov]{}

\subjclass[2010]{}


\newtheorem{thm}{Theorem}

\newtheorem{lem}{Lemma}[section]
\newtheorem{prop}{Proposition}[section]
\newtheorem{Def}{Definition}[section]
\newtheorem{cor}{Corollary}[section]

\newtheorem*{problem}{Circle Pattern Problem}
\newtheorem{claim}{Claim}

\newcommand{\reals}{\mathbb{R}}

\newcommand{\CC}{\mathbb{C}}
\newcommand{\PP}{\mathbb{F}^2}
\newcommand{\Femb}{\mathbb{F}_{01}}
\newcommand{\Euclideanplane}{\mathbb{E}^2}
\newcommand{\hyperbolicspace}{\mathbb{H}^3}
\newcommand{\hyperbolicplane}{\mathbb{H}^2}
\newcommand{\cellcomplex}{\mathcal{C}}
\newcommand{\Hor}{\mathbb{E}_1}
\newcommand{\Hplane}{\mathbb{H}_0}
\newcommand{\CP}{\mathbf{HCP}}
\newcommand{\CPT}{\widetilde{\mathbf{HCP}}_{\Triang}}
\newcommand{\CPC}{\mathbf{HCP}_{\cellcomplex}}
\newcommand{\Triang}{\mathcal{T}}
\newcommand{\polytope}{\mathcal{P}_{S,\cellcomplex}}
\newcommand{\naturals}{\mathbb{N}}

\newcommand{\ADelta}{\mathcal{A}_{\Delta}}
\newcommand{\Alocaldelta}{\mathcal{A}}

\newcommand{\Vol}{\mathbf{V}}

\newcommand{\ER}{\mathcal{ER}}
\newcommand{\TED}{\mathcal{TE}_{\Delta}}
\newcommand{\ERD}{\mathcal{ER}_{\Delta}}
\newcommand{\TE}{\mathcal{TE}}
\newcommand{\UDelta}{\mathbf{U}_{\Delta}}
\newcommand{\Uglobal}{\mathbf{U}}

\newcommand{\GB}{\mathcal{GB}}
\newcommand{\PhiT}{\Phi_{\Triang}}
\newcommand{\PhiC}{\Phi_{\cellcomplex}}
\newcommand{\tPhiT}{\tilde{\Phi}_{\Triang}}

\newcommand{\angl}{\measuredangle}
\newcommand{\e}{\varepsilon}

\begin{document}

\centerline{\scshape Nikolay Dimitrov\footnote{Technische Universit\"at Berlin, 
Insitut f\"ur Mathematik, MA 8-4, Stra{\ss}e des 17. Juni 136,
10623 Berlin, Germany, \email{dimitrov@math.tu-berlin.de}}}

\let\thefootnote\relax\footnote{Research supported by the
DFG Collaborative Research Center TR 109 ``Discretization in
Geometry and Dynamics"}

\begin{abstract}
The main objective of this study is to understand how geometric
hyper-ideal circle patterns can be constructed from given
combinatorial angle data. We design a hybrid method consisting of
a topological/deformation approach augmented with a variational
principle. In this way, together with the question of
characterization of hyper-ideal patterns in terms of angle data,
we address their constructability via convex optimization. We
presents a new proof of the main results from Jean-Marc
Schlenker's work on hyper-ideal circle patterns by developing an
approach that is potentially more suitable for applications.
\end{abstract}

\maketitle

\section{Introduction} \label{Sec_Intro}

The current article focuses on the existence, uniqueness and
construction of hyper-ideal circle patterns from a given angle
data. In addition to that, it includes an explicit
characterization of all angle data which can be geometrically
realized as a hyper-ideal circle pattern.

There are a lot of papers related to circle patterns. Possibly one
of the prototypical results in this area of research is Andreev's
characterization of compact convex polyhedra with non-obtuse
dihedral angles in hyperbolic space \cite{And1}. It utilized (in
the proper context) the so called Alexandrov's topological /
deformation method \cite{Alex}. The paper was followed by a
generalization which included polyhedra with ideal vertices
\cite{And2}. As it was emphasized by Thurston \cite{ThuNotes},
circle patterns on the sphere are inherently linked to ideal
polyhedra in hyperbolic three-space. He used this fact to extend
Andreev's results to circle patterns on surfaces of non-positive
Euler characteristic \cite{ThuNotes}. Rivin, in his article
\cite{Riv96}, extended Andreev's theorem to the case of ideal
tetrahedra without any restriction to non-obtuse dihedral angles
and thus characterized all Delaunay circle patterns on the sphere.
As an alternative to the topological / deformation method, works
like Colin de Verdi\`ere's \cite{CdV}, Rivin's \cite{Riv94} and
\cite{Riv03}, Leibon's \cite{Leibon}, and Bobenko and Springborn's
\cite{BobSpr1} have developed variational methods for
characterization of circle patterns.

Hyper-ideal circle patterns are generalizations of the standard
(ideal) circle patterns discussed in the preceding paragraph.
Their characterization on the sphere was done by Bao and Bonahon
\cite{BaoBon} (in the context of hyper-ideal polyhedra in the
hyperbolic three-space). Schlenker gave another proof in
\cite{Schl}. With respect to the current article, there are two
papers that are most relevant to our study. These are Schlenkers's
work \cite{Sch1} and Springborn's \cite{S}. The former uses
Alexandrov's deformation approach, while the latter utilizes a
variational method. On the one hand, Schelnker characterizes the
angle data explicitly, in terms of linear inequalities and
equalities, but his proof of existence and uniqueness is not
constructive in an obvious way and thus is not suitable for actual
applications. Springborn on the other hand provides a constructive
method for establishing the existence and uniqueness of
hyper-ideal patterns, but his characterization of the angle data
is implicit (in terms of coherent angle systems) which again
restricts its applicability. Moreover, he addresses only the case
of Euclidean cone-metrics and does not include their hyperbolic
counterparts.

We would like to think of the current paper as a hybrid between a
topological / deformation method and a variational approach. More
precisely, we provide a new proof of Schlenker's results
\cite{Sch1} by applying our version of the topological /
deformation technique and in the process we develop a variational
method for explicit construction of hyper-ideal patterns, in the
spirit of \cite{S}. Thus, our goal is not so much to reprove
Schlenker's results, but rather to introduce a new approach to the
proof which repairs the shortcomings of \cite{Sch1} and \cite{S},
while bringing the two together. We have developed a different
description of the objects involved in this study, which we
believe is more explicit, natural and clear. This, in its own
turn, leads to a different functional than the one used in
\cite{S} and discussed in \cite{Sch1} (in fact, its Legendre
dual). Moreover, our functional is locally strictly convex on an
open subdomain of a certain vector space and can easily be
extended by linearity to a convex functional on the whole space
eliminating any restrictions. Consequently, the optimization
problem that arises is fairly straightforward and
application-friendly. It could be used for the design of numerical
computer algorithms that construct hyper-ideal patterns from given
angle data. Furthermore, we have slightly extended Schlenker's
results to incorporate hyper-ideal patterns with touching circles.
In particular, as a special case, our proof covers circle packings
on compact surfaces with cone metrics. We have tried to make the
article fairly self-contained, including mostly constructions from
``scratch" and avoiding complicated theorems like the
hyperbolization of Haken orbifolds used in \cite{Sch1}. We have
also added some details and corrected an inaccuracy present in
\cite{Sch1} (see the remark after situation 2.2 in the proof of
lemma \ref{Lem_necessary_conditions_general}). Finally, the
motivation for the current article comes from its potential to
provide tools for the construction of a discrete analog of the
classical uniformization theorem for higher genus Riemann
surfaces. We plan to show this in a subsequent paper.

\section{Definitions and notations} \label{Sec_def_and_notations}

We set up the stage for our explorations by fixing some
terminology and notations. For the rest of this article, it is
assumed that $S$ is a closed topological surface. Furthermore, we
denote by $d$ a metric of constant Gaussian curvature on $S$ with
finitely many cone singularities $\text{sing}(d)$. The metric $d$
is called \emph{a flat cone-metric} whenever (i) any point from
$S\setminus\text{sing}(d)$ has a neighborhood isometric to an open
subset of the Euclidean plane $\Euclideanplane$, and (ii) every
point from $\text{sing}(d)$ has a neighborhood isometric to a
neighborhood of the tip of a Euclidean cone. Analogously, the
metric $d$ is called \emph{a hyperbolic cone-metric} whenever (i)
any point from $S\setminus\text{sing}(d)$ has a neighborhood
isometric to an open subset of the hyperbolic plane
$\hyperbolicplane$, and (ii) every point from $\text{sing}(d)$ has
a neighborhood isometric to a neighborhood of the tip of a
hyperbolic cone. We will use $\mathbb{F}^2$ as a notation for both
$\Euclideanplane$ and $\hyperbolicplane$ and for the rest of the
article $d$ will be either a hyperbolic or a Euclidean cone-metric
on $S$.

\begin{figure}[ht]
\centering
\includegraphics[width=14cm]{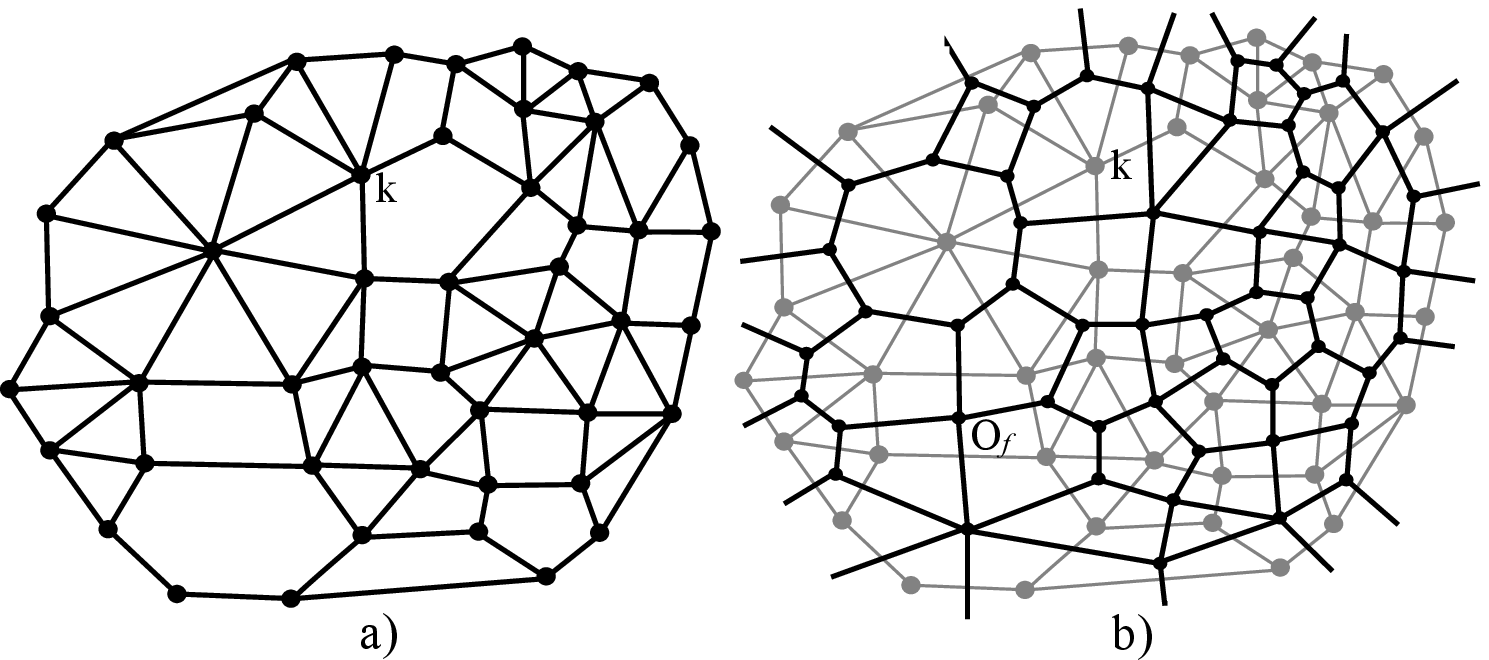}
\caption{a) The cell complex $\cellcomplex=(V,E,F)$ and b) its
dual $\cellcomplex^*=(V^*,E^*,F^*)$} \label{Fig1}
\end{figure}

For the rest of this article, $V$ will be a finite set of points
on $S$. Furthermore, by $\cellcomplex=(V,E,F)$ we will denote a
topological cell complex of $S$, where $V$ are the vertices, $E$
are the edges and $F$ are the faces of $\cellcomplex$ (see figure
\ref{Fig1}a). All three sets are assumed to be finite.
Furthermore, whenever a cone-metric $d$ is introduced on $S$, the
condition $\text{sing}(d) \subseteq V$ always holds. In order to
simplify notations, we will also assume that all cell complexes
involved in this study have the following regularity properties.

\medskip
\noindent $\bullet$ Any pair of edges from a cell complex either
(i) coincide, (ii) have exactly one vertex in common or (iii) are
disjoint with no vertices in common.

\medskip
\noindent $\bullet$ Any pair of faces either (i) coincide, (ii)
have exactly one vertex in common, (iii) have exactly one edge in
common, or (iv) are disjoint with no vertices or edges in common.

\medskip

\noindent This restriction is not essential and all results that
follow will also apply to more general cell complexes. However,
with this assumption in mind, the notations and the exposition
become much lighter. Indeed, let $i,j \in V$ be two vertices that
are endpoints of the same edge. Then, by assumption, $i$ and $j$
should be different and the notation $ij\in E$ uniquely determines
the edge, because there cannot be another edge with both $i$ and
$j$ as endpoints. Similarly, if $i_1,...,i_n$ are all the vertices
of a two-cell $f \in F$, then they are all different and the cell
is uniquely determined by the notation $f=i_1...i_n \in F$.

\begin{Def} A geodesic cell complex on $(S,d)$ is a cell complex
$\cellcomplex_d=(V,E_d,F_d)$ whose edges, with endpoints removed,
are geodesic arcs embedded in $S \setminus V$. Thus, each face
from $F_d$ is isometric to a compact geodesic polygon in $\PP$.
\end{Def}

In other words, we can think of a geodesic cell-complex
$\cellcomplex_d$ on a geometric surface $(S,d)$ as a two
dimensional manifold, obtained by gluing together geodesic
polygons along their edges. The edges that we identify should have
the same length and the identification should be an isometry.
Notice the difference between a topological cell complex
$\cellcomplex$ and a geodesic cell complex $\cellcomplex_d$. While
$\cellcomplex$ is just a purely topological (and hence
combinatorial) object, the geodesic one $\cellcomplex_d$ consists
of polygons with geodesic edges and thus provides the underlying
surface $S$ with a cone-metric $d$.

Assume three circles $c_i, c_j$ and $c_k$ with centers $i, j$ and
$k$ respectively, lie in the geometric plane $\PP$. Moreover, let
the circles' interiors be disjoint. Then, there exists a unique
forth circle $c_{\Delta}$ orthogonal to $c_i, c_j$ and $c_k$.
Furthermore, draw the geodesic triangle $\Delta=ijk$, spanned by
the centers $i, j$ and $k$. Then $\Delta$, together with the
circles $c_i, c_j, c_k$ and $c_{\Delta}$, is called a
\emph{decorated triangle} (see figure \ref{Fig2}a). The circles
$c_i, c_j$ and $c_k$ are called the \emph{vertex circles} of
$\Delta$, while $c_{\Delta}$ is called the \emph{face circle} of
$\Delta$. We point out here that in this article it is allowed for
one, two or all three vertex circles to degenerate to points. Even
in this more general set up, everything said above still applies.

\emph{Remark.} There is a slight subtlety in the case of
$\hyperbolicplane$. Although the vertex circles are always circles
in the usual, natural sense, the face circle may fit a more
general definition. For more details, see section
\ref{Sec_basic_geometry}.

Now, assume two non-overlapping decorated triangles, like
$\Delta_1=jis$ and $\Delta_2=uis$ from figure \ref{Fig2}a share a
common edge $is$. As usual, denote by $c_i, c_j, c_s$ and $c_u$
the vertex circles (some of which may be shrunk to points), and by
$c_{\Delta_1}$ and $c_{\Delta_2}$ the corresponding face circles
of the triangles. Although, in general, the two face circles
$c_{\Delta_1}$ and $c_{\Delta_2}$ are different, sometimes it may
happen that they coincide, i.e. $c_{\Delta_1}=c_{\Delta_2}=c_q$.
In that case all four vertex circles $c_i, c_j, c_s$ and $c_u$ are
orthogonal to $c_q$. Thus, we can erase the edge $is$ and obtain a
decorated geodesic quadrilateral $q=ijsu$ with vertex circles
$c_i, c_j, c_s$ and $c_u$, and a face circle $c_q$. Observe, that
in this case the quadrilateral is convex. If we continue this way,
we can obtain various decorated polygons, like for instance the
decorated pentagon $f'=ijvsu$ from figure \ref{Fig2}a.

\begin{figure}[ht]
\centering
\includegraphics[width=14cm]{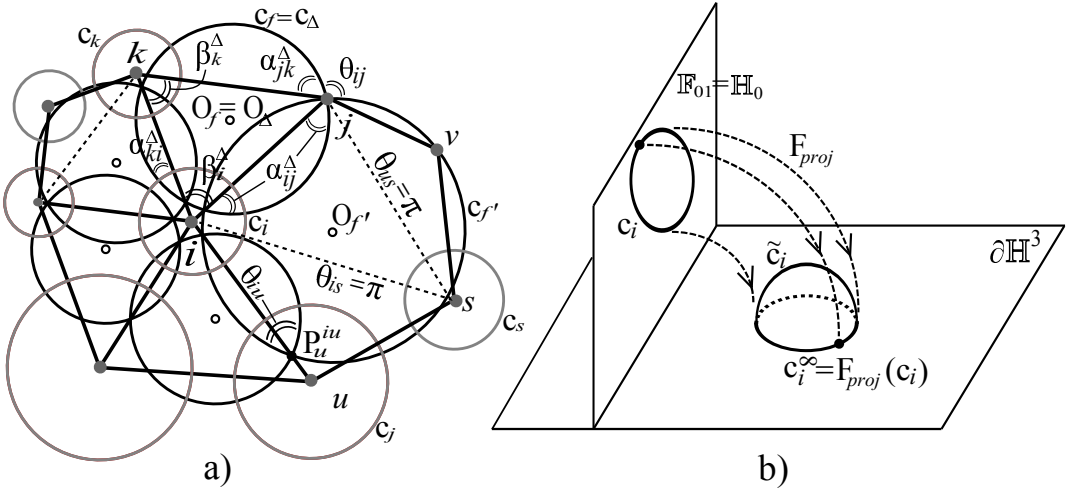}
\caption{a) A hyper-ideal circle pattern of decorated triangles,
polygons and labels; b) The projection $F_{proj}$ from the
hyperbolic plane $\Hplane \subset \hyperbolicspace$ to the ideal
boundary $\partial\hyperbolicspace$.} \label{Fig2}
\end{figure}

\begin{Def}\label{Def_decorated_polygon}
A decorated polygon is a convex geodesic polygon $p$ in $\PP$,
with vertices in labelled in cyclic order $i_1, i_2, ..., i_n$,
together with:

\medskip

\noindent $\bullet$ a set of circles $c_{i_1}, c_{i_2}, ...,
c_{i_n}$ with disjoint interiors such that each $c_{i_s}$ is
centered at vertex $i_s$ for $s=1,..,n$. Some or all of the
circles are allowed to be points, i.e. circles of radius zero;

\medskip

\noindent $\bullet$ another circle $c_p$ orthogonal to
$c_{i_1},...,c_{i_n}$.

\medskip

\noindent The circles $c_{i_1},...,c_{i_n}$ are called
\emph{vertex circles} and the additional orthogonal circle $c_p$
is called the \emph{face circle} of the decorated polygon $p$. 
\end{Def}

\noindent \emph{Remark:} Observe, that the vertex circles are
assumed to have disjoint interiors. That means that all vertex
circles could be either disjoint or some of them could be tangent
to one another.

Whenever two faces of a cell complex share a common edge, we will
say that the two faces are \emph{adjacent to each other}.
Furthermore, assume two decorated polygons $p_1$ and $p_2$ share a
common geodesic edge $ij$, where $i$ and $j$ are the endpoints of
$ij$, which also means that they are common vertices for both
$p_1$ and $p_2$. Then, the decorated polygons $p_1$ and $p_2$ are
called \emph{compatibly adjacent} whenever the vertex circles
$c_i^1, c_j^1$ of $p_1$ and $c_i^2, c_j^2$ of $p_2$ coincide
respectively, that is $c_i^1\equiv c_i^2$ and $c_j^1\equiv c_j^2$.
Furthermore, whenever two decorated polygons are compatibly
adjacent, we will say that their face circles are \emph{adjacent
to each other}. A situation like that is depicted on figure
\ref{Fig2}a for the edge $ij$ and the two faces $f\equiv\Delta$
and $f'$ with face circles $c_f$ and $c_{f'}$.

\begin{Def} \label{Def_local_Del_property}
Let $p_1$ and $p_2$ be two decorated polygons in $\PP$ that are
compatibly adjacent to each other. Let $ij$ be their common
geodesic edge. Furthermore, let $c_{p_1}$ and $c_{p_2}$ be the
face circles of $p_1$ and $p_2$ respectively.

\medskip

\noindent $\bullet$ We say that the edge $ij$ satisfies the
\emph{local Delaunay property} whenever each vertex circle of the
decorated polygon $p_2$ is either (i) disjoint from the interior
of the face circle $c_{p_1}$ of $p_1$, or (ii) if it is not, the
intersection angle between the vertex circle in question and the
face circle $c_{p_1}$ is less than $\pi /2$. See for instance edge
$ij$ on figure \ref{Fig2}a.

\medskip

\noindent $\bullet$ For the edge $ij$, which satisfies the local
Delaunay property, $\theta_{ij} \in [0,\pi)$ denotes the
intersection angle between the two adjacent face circles $c_{p_1}$
and $c_{p_2}$, measured between the circular arcs that bound the
region of common intersection. (See for example angles
$\theta_{ij}, \theta_{iu}, \theta_{is}$ and $\theta_{us}$ from
figure \ref{Fig2}a.)
\end{Def}

It is not difficult to see that the definition of a local Delaunay
property is symmetric in the sense that if the condition of
definition \ref{Def_local_Del_property} holds for the face circle
$c_{p_1}$ and the vertex circles of $p_2$, then it also holds for
the face circle $c_{p_2}$ and the vertex circles of $p_1$.

\begin{Def} \label{Def_hyperideal_circle_pattern}
A \emph{hyper-ideal circle pattern} on a given surface $S$ (figure
\ref{Fig2}a) is a hyperbolic or Euclidean cone-metric $d$ on $S$
together with a geodesic cell complex $\cellcomplex_d=(V,E_d,F_d)$
whose faces are decorated geodesic polygons such that any two
adjacent faces are compatibly-adjacent and each geodesic edge of
$\cellcomplex_d$ has the local Delaunay property. Whenever $d$ is
flat on $S\setminus V$, we call the circle pattern Euclidean, and
whenever $d$ is hyperbolic on $S\setminus V$, we call the pattern
hyperbolic.
\end{Def}

Intuitively speaking, a hyper-ideal circle pattern on a surface
$S$ is a surface homeomorphic to $S$, obtained by gluing together
decorated geodesic polygons along pairs of corresponding edges.
The edges that are being identified should have the same length,
the identification should be an isometry and the vertices that get
identified should have vertex-circles with same radii.


Observe that a hyper-ideal circle pattern on $S$ consists of (i) a
cone-metric $d$ on $S$, (ii) a set of vertices $V \supseteq
\text{sing}(d)$, (iii) an assignment of vertex radii $r$ on $V$,
and (iv) a geodesic cell complex $\cellcomplex_d$ together with
(v) a collection of vertex circles and (vi) a collection of face
circles. However, the geometric data $(S, d, V, r)$ is enough to
further identify uniquely the geodesic cell complex
$\cellcomplex_d$ and the collections of vertex and face circles.
This is done via the weighted Delaunay cell decomposition
construction. More precisely, given (i) a geometric surface
$(S,d)$, (ii) a finite set of points $V \supset \text{sing}(d)$ on
$S$ and (iii) an assignment of disjoint vertex circle radii $r : V
\to [0,\infty)$, one can uniquely generate (obtain) the
corresponding $r-$weighted Delaunay cell complex $\cellcomplex_d$,
where each edge satisfies the local Delaunay property. In the
process, the families of vertex and face circles naturally appear
as part of the construction \cite{BobIzm, S, Sch1}. Alternatively,
one can obtain the $r-$weighted Delaunay cell decomposition as the
geodesic dual to the $r$-weighted Voronoi diagram, also known as
the weighted power diagram with weights $r$ \cite{BobIzm}. A
Voronoi cell in the case when $d$ is Euclidean is defined as
$W_{d,r}(i)=\big\{ x \in S \, | \, d(x,i)^2-r_i^2 \leq
d(x,j)^2-r_j^2 \,\, \text{for all} \,\, j \in V \, \big\}$. A
Voronoi cell in the case when $d$ is hyperbolic is defined as
$W_{d,r}(i)=\big\{ x \in S \, | \, \cosh{(r_j)}\cosh{d(x,i)} \leq
\cosh{(r_i)}\cosh{d(x,j)} \,\, \text{for all} \,\, j \in V \,
\big\}$.

\section{The circle pattern problem and the main result} \label{Sec_problem_and_main_result}

Let us fix an arbitrary hyper-ideal circle pattern on $S$ and let
this pattern be determined by the data $(S, d, V, r)$. Figure
\ref{Fig2}a depicts (a portion of) a hyper-ideal circle pattern.
For each vertex $i \in V$ one can define $\Theta_i
> 0$ to be the cone angle of the cone-metric $d$ at vertex $i$.
Furthermore, since $S$ is a closed surface, each edge $ij \in E_d$
is the common edge of exactly two faces from the $r$-weighted
Delaunay cell-complex $\cellcomplex_d = (V, E_d, F_d)$. Call these
faces $f$ and $f' \in F_d$, one on each side of the edge.
Consequently, one can associate to each edge $ij$ the pair of
adjacent face circles $c_f$ and $c_{f'}$. As a result of this, one
can assign to $ij$ the intersection angle $\theta_{ij} \in
[0,\pi)$ between $c_f$ and $c_{f'}$, as explained in definition
\ref{Def_local_Del_property} and shown on figure \ref{Fig2}a.

Observe that given any hyper-ideal circle pattern on $S$, like the
one from the preceding paragraph,
one can always extract from it the combinatorial data
$(\cellcomplex,\theta,\Theta)$, where $\cellcomplex$ is the
$r-$weighted Delaunay cell decomposition $\cellcomplex_d$ viewed
as a purely topological complex, $\Theta : V \to (0, \infty)$ is
the assignment of cone angles at the vertices of the complex and
$\theta : E_d \to [0,\pi)$ is the assignment of intersection
angles between adjacent face circles of the pattern. In this case,
we will say that the given hyper-ideal circle pattern
\emph{realizes the (combinatorial angle) data}
$(\cellcomplex,\theta,\Theta)$.

The central scope of the current article is to answer the question
whether the procedure described in the previous paragraph can be
reversed. Compare with \cite{Sch1}. 

\begin{problem} \label{Circle_pattern_problem}

Assume the combinatorial data $(\cellcomplex, \theta, \Theta)$ is
provided, where

\begin{itemize}
\item $\cellcomplex = (V, E, F)$ is a topological cell complex on
a surface $S$;

\item $\theta : E \to [0,\pi) \,\,\,\, \text{ and } \,\,\,\,
\Theta : V \to (0,\infty)$.
\end{itemize}

\noindent Find a hyperbolic or flat cone metric $d$ on $S$,
together with a hyper-ideal circle pattern on it that realizes the
data $(\cellcomplex, \theta, \Theta)$.

\end{problem}

In this article, we provide a solution to the circle pattern
problem in the following form (see also \cite{Sch1}).

\begin{thm} \label{Thm_main}

Let $S$ be a closed surface with a topological cell complex
$\cellcomplex = (V, E, F)$ on it. There exist two convex polytopes
 $\polytope^{h}$ and $\polytope^{e}$, depending on the combinatorics of
 $\cellcomplex$ and containing points of type
 $(\theta, \Theta) \in \reals^E\times\reals^V$, for which the following
 statements hold:

\medskip

\noindent {\bf E.} The combinatorial data $(\cellcomplex, \theta,
\Theta)$ is realized by a Euclidean hyper-ideal circle pattern on
$S$ if and only if $(\theta, \Theta) \in \polytope^e$.
Furthermore, this pattern is unique up to scaling and isometry
between hyperbolic cone-metrics on $S$, isotopic to identity.

\medskip

\noindent {\bf H.} The combinatorial data $(\cellcomplex, \theta,
\Theta)$ is realized by a hyperbolic hyper-ideal circle pattern on
$S$ if and only if $(\theta, \Theta) \in \polytope^h$.
Furthermore, this pattern is unique up to isometry between
hyperbolic cone-metrics on $S$, isotopic to identity.



\medskip

\noindent In both cases, whenever the hyper-ideal circle pattern
exists, it can be reconstructed from the unique critical point of
a strictly convex functional defined on a suitably chosen open
subset of $\reals^N$ for some $N \in \naturals$.

\end{thm}

\noindent \emph{Remark:} The two polytopes $\polytope^{h}$ and
$\polytope^{e}$ are called \emph{angle data polytopes}. Their
explicit definition is given in the next section. For both of them
we will use the common notation $\polytope$.



\section{Description of the angle data polytopes} \label{Sec_angle_data_polytope}

In this section we give an explicit description of the two
polytopes $\polytope^{h}$ and $\polytope^{e}$ from theorem
\ref{Thm_main}.

Assume a cell complex $\cellcomplex = (V, E, F)$ is fixed on the
surface $S$ (see figure \ref{Fig1}a). Denote by $\cellcomplex^* =
(V^*, E^*, F^*)$ the cell complex dual to $\cellcomplex$, where
$V^*$ are the dual vertices, $E^*$ are the dual edges and $F^*$
are the dual faces (see figure \ref{Fig1}b). The dual vertices are
in bijective correspondence with the faces of $\cellcomplex$. To
simplify things, we can assume that each face $f \in F$ contains
exactly one vertex $O_f \in V^*$ in its interior. The dual edges
are obtained as follows: if $f$ and $f' \in F$ are two adjacent
faces of $\cellcomplex$ and $ij \in E$ is their common edge, then
there exists a dual edge $ij^* = O_fO_{f'} \in E^*$ which connects
the dual vertices $O_f$ and $O_{f'} \in V^*$. Just like with the
dual vertices, the dual faces are in bijective correspondence with
the vertices of $\cellcomplex$ and again we can assume that the
former contain the latter in their interiors. On figure
\ref{Fig1}b the elements of the original complex $\cellcomplex$
are drawn in grey, while the elements of the dual complex
$\cellcomplex^*$ are in black.

\begin{figure}[ht]
\centering
\includegraphics[width=14cm]{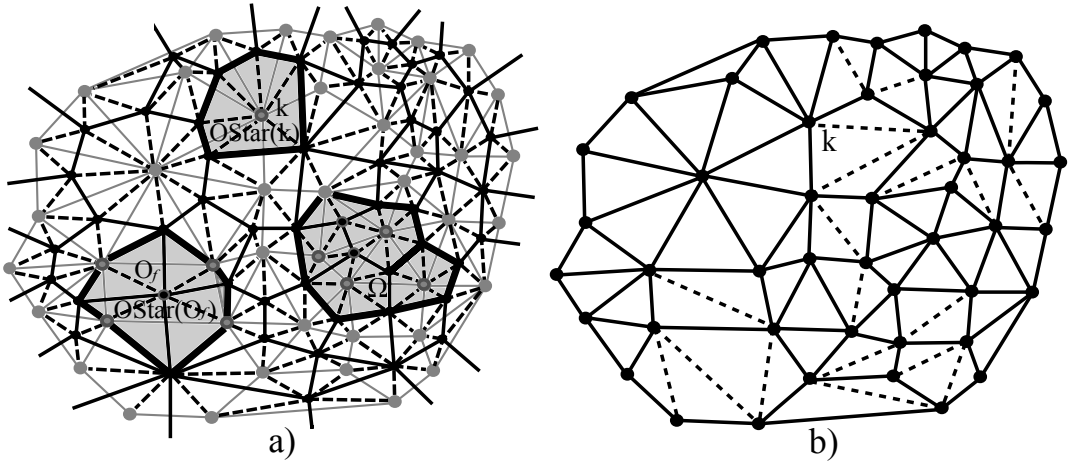}
\caption{a) The triangulation
$\hat{\Triang}=(\hat{V},\hat{E},\hat{F})$ together with two
examples of open stars and one admissible domain $\Omega$; b) The
subtriangulation $\Triang=(V,E_T,F_T)$ of $\cellcomplex$ whose
dashed edges are the auxiliary edges from $E_{\pi}$.} \label{Fig3}
\end{figure}

Next, define the subdivision $\hat{\Triang}=(\hat{V}, \hat{E},
\hat{F})$ of $\cellcomplex^*$, depicted on figure \ref{Fig3}a,
where

\medskip


\noindent $\bullet$ $\hat{V} = V \cup V^*$, i.e. the vertices of
$\hat{\Triang}$ consist of all vertices of $\cellcomplex$ and all
dual vertices. These are all black and grey vertices from figures
\ref{Fig1}b and \ref{Fig3}a;

\medskip

\noindent $\bullet$ $\hat{E} = E^* \cup \big\{\, iO_f \,\, | \,\,
O_f \in V^* \,\, \text{and $i$ is a vertex of $f$ } \big\}$, i.e.
the edges of $\hat{\Triang}$ consist of all dual edges and all
edges, obtained by connecting a dual vertex $O_f \in f$ to all the
vertices of the face $f \in F$ it belongs to. The latter type of
edges will be called \emph{corner edges}. The dual edges can be
seen on both figures \ref{Fig1}b and \ref{Fig3}a painted solid
black, while the corner edges are the black dashed edges from
figure \ref{Fig3}a.

\medskip

\noindent $\bullet$ $\hat{F} = \bigl\{ iO_fO_{f'} \, | \, \text{
$ij \in E$ common edge for $f$ and $f'$ from $F$ } \bigr\},$ i.e.
the faces of $\hat{\Triang}$ are the topological triangles
obtained by looking at the connected components of the complement
of the topological graph $(\hat{V}, \hat{E})$ on $S$. On figure
\ref{Fig3}a these are the triangles with one solid black and two
dashed black edges. They also have two black (dual) vertices and
one grey vertex.



\medskip

The next important notion to be defined is, what we call in this
paper, the \emph{open star} of a vertex from $\hat{\Triang}$.
\begin{Def} \label{Def_open_star}
Let $\hat{v} \in \hat{V}$ be an arbitrary vertex of
$\hat{\Triang}.$ Then its \emph{open star}
$\text{\emph{OStar}}(\hat{v})$ is defined as the open interior of
the union of all closed triangles from $\hat{\Triang}$ which
contain $\hat{v}$.
\end{Def}
In particular, whenever $\hat{v} = k \in V$ is a vertex of
$\cellcomplex$, then its open star is simply the open interior of
the face from $\cellcomplex^*$ 
dual to $k$. An example denoted by $\text{OStar}(k)$ and colored
in grey is shown on figure \ref{Fig3}a. Then, as one can see, the
boundary of $\text{OStar}(k)$ consists entirely of dual edges from
$E^*$. If we denote by $E_k$ the set of all edges of
$\cellcomplex$ which have vertex $k$ as an endpoint, then
$\partial \, \text{OStar}(k) = \cup \bigl\{ ik^* \in E^* \,\, |
\,\, ik \in E_k \, \bigr\}$. If $\hat{v} = O_f \in V^*$ is a
vertex from the dual complex $\cellcomplex^*$, then the boundary
of its open star consists entirely of corner edges from
$\hat{\Triang}$ (see the grey region $\text{OStar}(O_f)$ on figure
\ref{Fig3}a).

Before we continue, let us go back to the original cell complex
$\cellcomplex = (V, E, F)$. We are going to partition the set of
its vertices and edges depending on where we want the circle
pattern realizations of $\cellcomplex$ to have vertex circles of
radius zero and edges with tangent vertex circles centered at
their endpoints. Let

\medskip


\noindent $\bullet$  $V = V_0 \cup V_1$, where $V_0 \cap V_1 =
\varnothing$;

\medskip

\noindent $\bullet$ $E = E_0 \cup E_1$, where $E_0 \cap E_1 =
\varnothing$ and for any $ij \in E_0$ both $i$ and $j$ belong to
$V_1$.

\medskip

Following the terminology of \cite{Sch1}, one can define what
Schlenker calls an admissible domain. Our definition however is
more restrictive than his in the sense that we select a much
smaller collection of admissible domains than the ones described
in \cite{Sch1}. Thus we have decreased the number of conditions on
the angle data that appear in theorem
\ref{Thm_description_of_polytopes} below.

\begin{Def} \label{Def_admissible_domain}
An open connected subdomain $\Omega$ of the surfaces $S$ is called
an \emph{admissible domain of $(S,\cellcomplex)$ } whenever the
following conditions hold:

\medskip

\noindent {\bf 1.} There exists a subset $\hat{V}_0 \subseteq
\hat{V}$, such that $\Omega = \cup \big\{
\text{\emph{OStar}}(\hat{v}) \,\, | \,\, \hat{v} \in \hat{V}_0
\big\};$

\medskip

\noindent {\bf 2.} $\Omega \neq \varnothing$ and $\Omega \neq S$
and $\Omega \cap V \neq \varnothing$;









\end{Def}

A special example of an admissible domain is the open star of a
vertex of $\cellcomplex$. The open star of a dual vertex however
is not an admissible domain because it is disjoint from $V$. An
example of an admissible domain can be seen on figure \ref{Fig3}a,
denoted by the symbol $\Omega$ and shaded in grey. On this picture
$\Omega$ is simply connected but in general it doesn't have to be.

The boundary of an admissible domain $\Omega$ is a disjoint union
of immersed in $S$ topologically polygonal curves, consisting
entirely of edges from the triangulation $\hat{\Triang}$. In other
words, the boundary of $\Omega$ consists of dual edges and/or
corner edges from $\hat{E}$, but all of its connected components
are interpreted as immersed closed paths in the one-skeleton of
$\hat\Triang$, so that some of the edges could be traced (counted)
twice (see figure \ref{Fig3}a). That happens exactly when an edge
of $\hat{\Triang}$ is disjoint from $\Omega$, but the interiors of
the two topological triangles from $\hat{\Triang}$, lying on both
sides of the edge, are contained in $\Omega$. We denote this
immersed version of the boundary of $\Omega$ by $\partial\Omega$.

\begin{thm} \label{Thm_description_of_polytopes}

In the setting of theorem \ref{Thm_main}, the polytopes
$\polytope^{h}$ and $\polytope^{e}$ are defined as follows
(compare to \cite{Sch1}):

\medskip

\noindent {\bf E.} \emph{Euclidean case}. A point
 $(\theta, \Theta) \in \reals^{E_1} \times \reals^{V_1}$ belongs
 to $\polytope^{e}$ exactly when:

\medskip

\noindent E1) For any $ij \in E_0$ let $\theta_{ij} = 0$ while
$\theta_{ij} \in (0,\pi)$ for $ij \in E_1$;

\medskip

\noindent E2) For any $k \in V_0$ let $\Theta_k = \sum_{ik \in
E_k} (\pi - \theta_{ik})$. This is equivalent to $\Theta_k =
\sum_{ik^* \subset
\partial\Omega} (\pi - \theta_{ik})$ for $\Omega = \text{\emph{OStar}}(k)$.
Also $\Theta_k > 0$ for all $k \in V_1$;
\medskip

\noindent E3) $\sum_{k \in V} (2 \pi - \Theta_k) = 2 \pi \chi(S)$;

\medskip

\noindent E4) For any admissible domain $\Omega$ of $(S,
\cellcomplex)$, such that $\Omega \neq \text{\emph{OStar}}(k)$ for
some $k \in V_0$,

\begin{equation} \label{Eqn_first_condition}
\sum_{ij^* \subset \partial\Omega} (\pi - \theta_{ij}) + \sum_{k
\in \Omega \cap V} (2 \pi - \Theta_k) > 2 \pi \chi(\Omega) - \pi
|\partial\Omega \cap V|.
\end{equation}

\medskip

\noindent {\bf H.} \emph{Hyperbolic case}. A point
 $(\theta, \Theta) \in \reals^{E_1} \times \reals^{V_1}$ belongs
 to $\polytope^{h}$ exactly when:

\medskip

\noindent H1) For any $ij \in E_0$ let $\theta_{ij} = 0$ while
$\theta_{ij} \in (0,\pi)$ for $ij \in E_1$;

\medskip

\noindent H2) For any $k \in V_0$ let $\Theta_k = \sum_{ik \in
E_k} (\pi - \theta_{ik})$. Put in another way, $\Theta_k =
\sum_{ik^* \subset
\partial\Omega} (\pi - \theta_{ik})$ for $\Omega = \text{\emph{OStar}}(k)$.
Also $\Theta_k > 0$ for all $k\in V_1$;

\medskip

\noindent H3) $\sum_{k \in V} (2 \pi - \Theta_k) > 2 \pi \chi(S)$;

\medskip

\noindent H4) For any admissible domain $\Omega$ of $(S,
\cellcomplex)$, such that $\Omega \neq \text{\emph{OStar}}(k)$ for
some $k \in V_0$,

\begin{equation*} 
\sum_{ij^* \subset \partial\Omega} (\pi - \theta_{ij}) + \sum_{k
\in \Omega \cap V} (2 \pi - \Theta_k) > 2 \pi \chi(\Omega) - \pi
|\partial\Omega \cap V|.
\end{equation*}

Here $\chi(S)$ and $\chi(\Omega)$ are the Euler characteristics of
$S$ and $\Omega$ respectively.

\end{thm}
Condition (\ref{Eqn_first_condition}) can be also written as

$$\sum_{ij^* \subset \partial\Omega \cap E_1} (\pi - \theta_{ij}) + \sum_{k
\in \Omega \cap V} (2 \pi - \Theta_k) > 2 \pi \chi(\Omega) - \pi
|\partial\Omega \cap V| - \pi |\partial \Omega \cap E_0|.$$

We are going to assume that the vector space
$\reals^{E_1}\times\reals^{V_1}$ is an affine subspace of
$\reals^{E}\times\reals^{V}$ by assuming that any $(\theta,
\Theta) \in \reals^{E_1}\times\reals^{V_1}$ is extended to a point
in $\reals^{E}\times\reals^{V}$ by letting $\theta_{ij}=0$ for all
$ij \in E \setminus E_1 = E_0$, and $\Theta_k = \sum_{ik \in E_k}
(\pi - \theta_{ik})$ for all $k \in V \setminus V_1=V_0$. Thus,
one can naturally assume that the polytopes $\polytope^{h}$ and
$\polytope^{e}$ are in the larger space
$\reals^{E}\times\reals^{V}$.

To optimize the conditions from theorem
\ref{Thm_description_of_polytopes} a bit more, one can define the
so called \emph{strict admissible domain}.

\begin{Def} \label{Def_strict_admissible_domain}
An open connected subdomain $\Omega$ of the surfaces $S$ is called
a \emph{strict admissible domain of $(S,\cellcomplex)$ } whenever
$\Omega$ is admissible and $\partial\Omega \cap V_0 =
\varnothing$.
\end{Def}

As it turns out, the angle data polytopes can be described via
strict admissible domains instead of admissible domains. In fact,
the admissible domains which are not strict do not add more
restrictions to the angle data, i.e. they produce redundant
conditions.

\begin{cor} \label{Cor_polytopes_general_admissible_domains}
The statements of theorem \ref{Thm_description_of_polytopes} still
hold even if the expression ``admissible domain" in points E4 and
H4 of theorem \ref{Thm_description_of_polytopes} is replaced by
the expression ``strict admissible domain".
\end{cor}

\section{Some basic geometric facts} \label{Sec_basic_geometry}

In what follows, we state some basic facts from Euclidean and
hyperbolic geometry, which will be useful in our investigation.

The primary models of the hyperbolic plane $\hyperbolicplane$,
used in this article, are the two standard conformal models - the
upper half-plane and the Poincar\'e disc. The term ``conformal"
means that in both of these models, the measure of angle with
respect to the hyperbolic metric equals the measure of angle with
respect to the underlying Euclidean metric.
Although the notion of a circle in Euclidean geometry is
well-known, circles in the hyperbolic plane require some
attention. Let a \emph{regular circle} in $\hyperbolicplane$ be
defined as a curve in $\hyperbolicplane$ consisting of all points
which are equidistant from a given point, called the center of the
circle. In addition to that, let a \emph{hyper-circle} in
$\hyperbolicplane$ be defined as a curve in $\hyperbolicplane$,
equidistant from a given geodesic, called the \emph{central
geodesic}, and lying on one side of that geodesic. Then the word
\emph{circle} in $\hyperbolicplane$ (or alternatively
\emph{hyperbolic circle}) is the common term we use for regular
circles, horocycles (see \cite{ThuNotes, ThuBook, Bus} or
\cite{BenPetr}) and hyper-circles. Furthermore, regular circles
and horocycles have well defined interiors, i.e. they have discs.
Consequently, a regular circle is the boundary curve of a
\emph{regular disc} and a horocycle is the boundary curve of a
\emph{horodisc}. Analogously, a \emph{hyper-disc} is a connected
subdomain of $\hyperbolicplane$, whose boundary curve is a
hyper-circle with a central geodesic contained in the hyper-disc.
Consequently, the word \emph{disc} in $\hyperbolicplane$ (or
alternatively \emph{hyperbolic disc}) is the common term for
regular discs, horodiscs and hyper-discs. Thus, \emph{inside a
circle} means inside the disc of the circle in question.

A very useful property of the Poincar\'e disc and the upper-half
plane is that hyperbolic circles are exactly the intersections of
the model with ordinary Euclidean circles (circles in the
underlying Euclidean geometry). In this line of thoughts, a
regular circle in $\hyperbolicplane$ is in fact a Euclidean circle
fully contained in $\hyperbolicplane$. The only peculiarity here
is that, generically speaking, the hyperbolic center of a regular
circle in $\hyperbolicplane$ is different from its Euclidean
center. Furthermore, a hyper-circle in $\hyperbolicplane$ is the
circular arc obtained from the intersection of a Euclidean circle
with $\hyperbolicplane$, where the Euclidean circle intersects
$\partial\hyperbolicplane$ at exactly two points. The geodesic
between these two ideal points is the central geodesic of the
hyper-circle. In particular, hyperbolic geodesics are a special
type of hyper-circles, orthogonal to $\partial\hyperbolicplane$,
i.e. we can think that their ``hyperbolic radius" is equal to
zero. Finally, a horocycle is in general a Euclidean circle inside
$\hyperbolicplane$ tangent to $\partial\hyperbolicplane$.

Since circle patterns are traditionally linked to polyhedral
objects in the hyperbolic three-space $\hyperbolicspace$, a fact
which we will exploit a lot in this article, we will strongly rely
on the upper half-space model 
of $\hyperbolicspace$. Just like in the case of $\hyperbolicplane$
the latter is a conformal model and shares analogous properties
with the upper-half plane.

It is worth mentioning that the disc can be transformed into the
upper-half plane by a planar M\"obius transformation. Here is one
way to do this. For notational simplicity, identify the plane with
$\CC$. Let $\hyperbolicplane_{up}$ be the upper-half plane
$\{\text{Im}(z)
> 0\}$ and $\hyperbolicplane_d$ be the unit disc $\{|z| < 1\}$.
Furthermore, draw circle $c_0$ centered at $-i = -\sqrt{-1}$ and
passing through the points $-1$ and $1$. Then it is immediate to
see that the inversion in $c_0$ maps $\hyperbolicplane_{d}$ to
$\hyperbolicplane_{up}$. Consequently, one can easily carry
constructions from the unit disc to the upper-half plane and vice
versa. From now on by $l_{\hyperbolicplane}(MN)$ and
$l_{\hyperbolicspace}(MN)$ we denote the hyperbolic length of a
geodesic segment $MN$, and by $|MN|$ we denote the length of the
straight-line segment $MN$ with respect to the background
Euclidean geometry. Also, we would denote by $l_{\PP}(MN)$ the
distance between $M$ and $N$ in the plane $\PP$. For more details
on two and three dimensional hyperbolic geometry, one can consult
for example \cite{ThuNotes, ThuBook, Bus} or \cite{BenPetr}.

\emph{Remark.} In this section, we have chosen to present proofs
based on compass and straightedge constructions with the
presumption that these may turn out useful for certain
applications, such as computer realizations for instance.


\begin{prop} \label{Prop_radical_axis}
Let circles $c_f$ and $c_{f'}$ in $\PP$ intersect in exactly two
points $P_i^{ij}$ and $P_{j}^{ij}$. Let $Ra_{ij}$ be the geodesic
passing through both points $P_i^{ij}$ and $P_{j}^{ij}$.
Furthermore, denote by $\tilde{R}a_{ij}$ the geodesic $Ra_{ij}$
with the closed geodesic segment $P_i^{ij}P_{j}^{ij}$ removed from
it. Then

\smallskip

\noindent {\bf 1.} Any point form $\tilde{R}a_{ij}$ is the center
of exactly one circle orthogonal to both $c_{f}$ and $c_{f'}$;

\smallskip

\noindent {\bf 2.} If the point $O$ from $\tilde{R}a_{ij}$, is the
center of a circle $c$ orthogonal to $c_{f}$, then $c$ is also
orthogonal to $c_{f'}$;

\smallskip

\noindent {\bf 3.} If a circle $c$ is orthogonal to both circles
$c_{f}$ and $c_{f'}$, then its center $O$ lies on
$\tilde{R}a_{ij}$. In the case of $\PP=\hyperbolicplane$, the
circle $c$ is assumed to be regular.


\end{prop}


\begin{proof}

\noindent \emph{Euclidean case}. When $\PP$ is the Euclidean
plane, the statement of this proposition follows from the
properties of the radical axis of a pair of intersecting circles.

\smallskip
\noindent \emph{Hyperbolic case}. The hyperbolic case follows from
the Euclidean case, combined with some basic properties of the
Poincar\'e disc model. Denote by $O_{\hyperbolicplane}$ the
Euclidean center of the Poincar\'e disc $\hyperbolicplane$.
Whenever the hyperbolic center $O$ of a circle $c$ in
$\hyperbolicplane$ coincides with $O_{\hyperbolicplane}$ then
$O\equiv O_{\hyperbolicplane}$ is also the Euclidean center of
$c$. Thus, in order to complete the proof of the current
proposition, it is enough to move the point $O$, via a hyperbolic
isometry of $\hyperbolicplane$, to $O_{\hyperbolicplane}$. Then
one can apply the Euclidean case. After that, one can use the fact
that the straight line $Ra_{ij}$ passes through
$O_{\hyperbolicplane}$ and consequently its intersection with
$\hyperbolicplane$ is a hyperbolic geodesic.
\end{proof}

\begin{cor} \label{Cor_two_adjacent_triangles}
Let $f$ and $f'$ be two compatibly adjacent decorated polygons in
$\PP$, sharing a common edge $ij$. Let their corresponding face
circles be $c_f$ and $c_{f'}$. Then the two points of intersection
$P^{ij}_i$ and $P^{ij}_j$ of $c_{f}$ and $c_{f'}$ lie on the
common edge $ij$.
\end{cor}

\noindent This last corollary allows us to define the intersection
angle between the face circles of two adjacent decorated polygons.

\begin{Def} \label{Def_angle_between_2_face_circles}
Let two compatibly adjacent decorated polygons $f$ and $f'$ share
a common edge $ij$ and have corresponding face circles $c_f$ and
$c_{f'}$. In accordance with corollary
\ref{Cor_two_adjacent_triangles}, let $P^{ij}_i$ and $P^{ij}_j$ be
the two intersection points of the circles $c_{f}, c_{f'}$ and the
edge $ij$. Point $P^{ij}_i \in ij$ is the closer one to vertex
$i$, while $P^{ij}_j \in ij$ is the closer on to vertex $j$. The
points $P^{ij}_i$ and $P^{ij}_j \in c_f$ split the circle $c_f$
into two circular arcs. Denote by $c_f(ij) \subset c_f$ the arc
whose interior is disjoint from the edges of $f$. In the same way,
define the arc $c_{f'}(ij) \subset c_{f'}$. Then, the
\emph{intersection angle $\theta_{ij} \in [0, 2\pi)$ between $c_f$
and $c_{f'}$} is defined to be the angle between the arcs
$c_f(ij)$ and $c_{f'}(ij)$ measured inside the bounded region the
two arcs enclose (see figure \ref{Fig2}a).
\end{Def}

\noindent A straightforward consequence of definition
\ref{Def_angle_between_2_face_circles} is the following statement

\begin{prop} \label{Prop_special_angles_bw_face_circles}
In the setting of definition 
\ref{Def_angle_between_2_face_circles}, the angle $\theta_{ij} =
0$ if and only if the face circles $c_f$ and $c_{f'}$ are tangent
to the edge $ij$ and to each other at the point $P^{ij}_i \equiv
P^{ij}_j$. Furthermore, $\theta_{ij}=\pi$ if and only if the two
face circles coincide, i.e. $c_f \equiv c_{f'}$. (See for example
edges $is$ and $js$ on figure \ref{Fig2}a.)
\end{prop}

A \emph{ray} in $\PP$ is a geodesic half-line. In other words,
this is a Euclidean half-line, in the case of $\Euclideanplane$,
and a hyperbolic half-geodesic, in the case of $\hyperbolicplane$.
We denote a ray by $A\overrightarrow{r}$, where $A$ is the ray's
point of origin. If a ray starts from a point $A$ and passes
through a point $B$, then it can be also denoted by
$\overrightarrow{AB}$. Let $A\overrightarrow{r_1}$ and
$A\overrightarrow{r_2}$ be two rays with a common origin $A$.
Denote by $\alpha_0=\measuredangle r_1Ar_2$ the angle between
them, fixed so that $\alpha_0 \in (0,\pi)$. Then $Dom_{12}$ is
chosen to be the closed convex domain
 bounded by the two rays (infinite sector).
 Observe that the angle $\alpha_0$ is measured
inside $Dom_{12}$. Let $c$ be a circle which intersects both
$A\overrightarrow{r_1}$ and $A\overrightarrow{r_2}$. The angle
between $c$ and $A\overrightarrow{r_j}$ is the angle $\alpha_j$
measured inside $Dom_{12}$ and outside $c$, or equivalently,
measured inside $c$ and outside $Dom_{12}$. Finally, a
\emph{homothety (uniform scaling)} is a special similarity
transformation of the Euclidean plane $\Euclideanplane$ which
fixes a single point, called the center of the homothety, and maps
each line through that point onto itself. In fact, the group of
similarities of $\Euclideanplane$ is generated by all Euclidean
isometries and homotheties. The latter preserve angles but not
lengths. However, they preserve ratios of lengths.

\begin{lem} \label{Lem_two_rays_and_one_circle}

Let $A$ be an arbitrary point in $\PP$, and let
$A\overrightarrow{r_1}$ and $A\overrightarrow{r_2}$ be two rays
with common origin $A$. Let $\alpha_0 \in (0,\pi)$ be the angle
between the rays and let $\alpha_1, \alpha_2 \in (0,\pi)$ be such
that $\alpha_0 + \alpha_1 + \alpha_2 < \pi$. Then there is a
unique ruler and compass constructible pair of rays
$A\overrightarrow{t_{1}}$ and $A\overrightarrow{t_{2}}$ with a
common origin $A$ that have the following properties:

\smallskip

\noindent {\bf 1.} A circle is tangent to both rays
$A\overrightarrow{t_1}$ and $A\overrightarrow{t_2}$ if and only if
its intersection angles with $A\overrightarrow{r_1}$ and
$A\overrightarrow{r_2}$ are $\alpha_1$ and $\alpha_2$
respectively. When $\PP = \hyperbolicplane$, the circle is
allowed to be a geodesic. 

\smallskip

\noindent {\bf 2.} If a circle is tangent to
$A\overrightarrow{t_1}$ and has intersection angle $\alpha_1$ with
$A\overrightarrow{r_{1}}$, then it is tangent to
$A\overrightarrow{t_2}$ and its intersection angle with
$A\overrightarrow{r_{2}}$ is necessarily $\alpha_2$.



\end{lem}

\begin{proof}
\noindent \emph{Euclidean case.} The proof of the Euclidean case
will help us prove the hyperbolic version. We are aware of several
ways one could go about the construction of the rays
$A\overrightarrow{t_1}$ and $A\overrightarrow{t_2}$. However, we
present just one of them.

\smallskip

\noindent 1. 
Fix an arbitrary number $R>0$. Construct two isosceles triangles
$\triangle X_1Y_1Z_1$ and $\triangle X_2Y_2Z_2$ such that 
(i) points $X_j$ and $Y_j$ lie on $A\overrightarrow{r_j}$, (ii)
$|X_jZ_j|=|Y_jZ_j|=R$, (iii) $\measuredangle X_jZ_jY_j =
2\alpha_j$ and (iv) either point $Z_j \in Dom_{12}$, when
$\alpha_j \in (0,\pi/2]$,  or $Z_j$ and $Dom_{12}$ are separated
by the line $X_jY_j$, when $\alpha_j \in (\pi/2,\pi)$ (here $j\neq
k = 1,2$). Notice that at most one $\alpha_j$ can be greater or
equal to $\pi/2$. Draw lines $l_1$ and $l_2$ such that $l_1$
passes through $Z_1$ and is parallel to $A\overrightarrow{r_1}$,
while $l_2$ passes through $Z_2$ and is parallel to
$A\overrightarrow{r_2}$. Let $O_{12}$ be the intersection point of
$l_1$ and $l_2$. Draw a circle $c_{12}$ of radius $R$ with center
$O_{12}$. The inequality $\alpha_0 + \alpha_1 + \alpha+2 < \pi$ is
equivalent to the fact that $A$ is outside $c_{12}$ which, in its
own turn, is equivalent to the fact that $c_{12}$ intersects each
ray $A\overrightarrow{r_j}$ at exactly two points, $j=1,2$. Then
one can easily check that the intersection angles of the circle
$c_{12}$ with the rays $A\overrightarrow{r_1}$ and
$A\overrightarrow{r_2}$ are $\alpha_1$ and $\alpha_2$
respectively. Denote by $A\overrightarrow{r_{12}}$ the ray
$\overrightarrow{AO_{12}}$. Furthermore, construct
$A\overrightarrow{t_1}$ and $A\overrightarrow{t_2}$ as the two
rays with a common point of origin $A$ and tangent to circle
$c_{12}$. 
The indices $j=1,2$ are chosen so that the ray
$A\overrightarrow{r_1}$ is the one between rays
$A\overrightarrow{t_1}$ and $A\overrightarrow{r_2}$. Observe that
$A\overrightarrow{r_{12}}$ bisects the angle $\measuredangle
t_1At_2 = \gamma_0 \in (0,\pi)$, formed by the tangent rays
$A\overrightarrow{t_1}$ and $A\overrightarrow{t_2}$.

Now let $c$ be any circle tangent to both $A\overrightarrow{t_1}$
and $A\overrightarrow{t_2}$. Then its center $O$ necessarily lies
on the angle bisector $A\overrightarrow{r_{12}}$. Apply to
$c_{12}$ the unique homothety with center $A$ which sends point
$O_{12}$ to point $O$. Since this homothety maps the three rays
$A\overrightarrow{t_1}, A\overrightarrow{t_2}$ and
$A\overrightarrow{r_{12}}$ to themselves, the image of the circle
$c_{12}$ is again a circle, centered at $O \in
A\overrightarrow{r_{12}}$ and tangent to both
$A\overrightarrow{t_1}$ and $A\overrightarrow{t_2}$. But the
circle $c$ is the unique circle centered at $O$ and tangent to
$A\overrightarrow{t_1}$ (and consequently tangent to
$A\overrightarrow{t_2}$ sa well). Hence $c$ is the image of
$c_{12}$. By observing that the homothety also maps the rays
$A\overrightarrow{t_1}$ and $A\overrightarrow{r_2}$ to themselves,
as well as it preserves angles, we conclude that the intersection
angles of $c$ with $A\overrightarrow{t_1}$ and
$A\overrightarrow{r_2}$ equal the intersection angles of its
preimage $c_{12}$ with $A\overrightarrow{t_1}$ and
$A\overrightarrow{r_2}$, which by construction are $\alpha_1$ and
$\alpha_2$ respectively.

Conversely, let a circle $c$, with a center $O$, intersect both
rays $A\overrightarrow{r_1}$ and $A\overrightarrow{r_2}$ at angles
$\alpha_1$ and $\alpha_2$ respectively. For $j=1,2$ let $Q_j$ be
the farthest from $A$ intersection point of $c$ and
$A\overrightarrow{r_j}$. Similarly, let $Q^{12}_j$ be the farthest
from $A$ intersection point of the ray $A\overrightarrow{r_j}$ and
the circle  $c_{12}$ constructed above. Recall, $O_{12} \in
A\overrightarrow{r_{12}}$ is the center of $c_{12}$. Since
$\measuredangle OQ_jA = \pi/2 - \alpha_j = \measuredangle
O_{12}Q^{12}_jA$, the lines $OQ_j$ and $O_{12}Q^{12}_j$ are
parallel. Let points $A_1$ and $A_2$ be the intersection points of
the line $O_{12}O$ with the lines $Q^{12}_1Q_1$ and $Q^{12}_2Q_2$
respectively. Then $\frac{|A_1O|}{|A_1O_{12}|} =
\frac{|OQ_1|}{|O_{12}Q^{12}_1|} = \frac{|OQ_2|}{|O_{12}Q^{12}_2|}
= \frac{|A_2O|}{|A_2O_{12}|}$, coming from the fact that
$|OQ_1|=|OQ_2|$ and $|O_{12}Q^{12}_1|=|O_{12}Q^{12}_2|$. Therefore
$\frac{|A_1O|}{|OO_{12}|} = \frac{|A_2O|}{|OO_{12}|}$, hence
$|A_1O| = |A_2O|$. The last equality means that $A_1\equiv A_2$
and thus the three lines $Q^{12}_1Q_1, \, Q^{12}_2Q_2$ and
$OO_{12}$ have a common point of intersection. But since, for
$j=1,2$ each line $Q^{12}_jQ_j$ contains the ray
$A\overrightarrow{r_j}$, the two lines $Q^{12}_1Q_1$ and
$Q^{12}_2Q_2$ already meet at the point $A$. Thus, point $A$ also
lies on the line $OO_{12}$, which leads to the conclusion that $O
\in A\overrightarrow{r_{12}}$. Therefore, there exists a unique
homothety with center $A$ that maps $O_{12}$ to $O$. Since the
line $O^{12}Q^{12}_1$ is parallel to the line $OQ_1$, then the
homothety maps $O^{12}Q^{12}_1$ to $OQ_1$, which in its own turn
means that $Q_1$ is the homothetic image of $Q^{12}_1$. Therefore,
the homothety sends the circle $c_{12}$ to a circle centered at
$O$ and passing through $Q_1$. But $c$ is the unique circle with
center $c$ and radius $|OQ_1|$ so $c$ is the image of $c_{12}$.
Hence, $c$ is tangent to both $A\overrightarrow{t_1}$ and
$A\overrightarrow{t_2}$.

\smallskip

\noindent 2. Let us have a circle $c$ with center $O$ tangent to
the ray $A\overrightarrow{t_1}$ and intersecting the ray
$A\overrightarrow{r_{1}}$ at an angle $\alpha_1$. Also, recall the
circle $c_{12}$ with center $O_{12} \in A\overrightarrow{r_{12}}$
used in the construction of $A\overrightarrow{t_1}$ and
$A\overrightarrow{t_2}$. Let $T^{12}_1$ and $T_1$ be the points of
tangency between $A\overrightarrow{t_1}$ and the circles $c_{12}$
and $c$ respectively. Then the lines $O_{12}T^{12}_1$ and $OT_1$
are orthogonal to $A\overrightarrow{t_1}$ and hence are parallel
to each other. Let $Q_1$ be the farthest from $A$ intersection
point of $c$ and $A\overrightarrow{r_1}$. Also recall that
$Q^{12}_1$ is the farthest from $A$ intersection point of
$A\overrightarrow{r_1}$ and $c_{12}$. Since $\measuredangle OQ_1A
= \pi/2 - \alpha_1 = \measuredangle O_{12}Q^{12}_1A$, the lines
$OQ_1$ and $O_{12}Q^{12}_1$ are parallel. Let points $A_r$ and
$A_t$ be the intersection points of the line $O_{12}O$ with the
lines $Q^{12}_1Q_1$ and $T^{12}_1T_1$ respectively. Then
$\frac{|A_rO|}{|A_rO_{12}|} = \frac{|OQ_1|}{|O_{12}Q^{12}_1|} =
\frac{|OT_1|}{|O_{12}T^{12}_1|} = \frac{|A_tO|}{|A_tO_{12}|}$, due
to the equalities $|OQ_1|=|OT_1|$ and
$|O_{12}Q^{12}_1|=|O_{12}T^{12}_1|$. Therefore
$\frac{|A_rO|}{|OO_{12}|} = \frac{|A_tO|}{|OO_{12}|}$, hence
$|A_rO| = |A_tO|$. The last equality means that $A_r\equiv A_t$
and thus the three lines $Q^{12}_1Q_1, \, T^{12}_1T_1$ and
$OO_{12}$ have a common point of intersection. But since, line
$Q^{12}_1Q_1$ contains the ray $A\overrightarrow{r_1}$ and the
line $T^{12}_1T_1$ contains the ray $A\overrightarrow{t_1}$, the
two lines $Q^{12}_1Q_1$ and $T^{12}_1T_1$ already meet at the
point $A$. Thus, point $A$ also lies on the line $OO_{12}$, which
leads to the conclusion that $O \in \overrightarrow{AO_{12}} =
A\overrightarrow{r_{12}}$. Consequently, there exists a unique
homothety with center $A$ that maps $O_{12}$ to $O$. Since the
lines $O^{12}Q^{12}_1$ and $O^{12}T^{12}_1$ are parallel to the
lines $OQ_1$ and $OT_1$ respectively, then the homothety maps
$Q^{12}_1$ to $Q_1$ and $T^{12}_1$ to $T_1$. Therefore, the
homothety sends the circle $c_{12}$ to a circle centered at $O$
and passing through $Q_1$ and $T_1$. Since $c$ is the unique
circle with this property, it is the image of $c_{12}$. Hence, $c$
is also tangent to $A\overrightarrow{t_2}$ and its angle of
intersection with $A\overrightarrow{r_2}$ is $\alpha_2$.

\smallskip

\noindent \emph{Hyperbolic case.} One can directly argue that
point $A$ can be moved to the Euclidean center of the unit circle
$\partial \hyperbolicplane$ by a hyperbolic isometry. Then the
hyperbolic rays $A\overrightarrow{r_1}$ and
$A\overrightarrow{r_2}$ become directed Euclidean segments on a
pair of Euclidean rays. As hyperbolic circles are in fact
intersections of Euclidean circles with $\hyperbolicplane$, the
Euclidean version of the current lemma applies and proves the
hyperbolic case. However, we also present a more direct
construction, which may be helpful in applications.

Let us work in the underlying Euclidean geometry. Denote by
$\kappa_1$ and $\kappa_2$ the circles determined by the hyperbolic
rays $A\overrightarrow{r_1}$ and $A\overrightarrow{r_2}$. Then
$\kappa_1$ and $\kappa_2$ are orthogonal to
$\partial\hyperbolicplane$ and $A \in \kappa_1 \cap \kappa_2$. Let
$A^*$ be the second intersection point of $\kappa_1$ and
$\kappa_2$, lying outside $\hyperbolicplane$. Thus, $A^*$ is the
inverse image of $A$ with respect to $\partial\hyperbolicplane$.
Draw the Euclidean rays $A\overrightarrow{\rho_1}$ and
$A\overrightarrow{\rho_2}$ tangent at the point $A$ to $\kappa_1$
and $\kappa_2$ respectively, where the orientation of
$A\overrightarrow{\rho_1}$ and $A\overrightarrow{\rho_2}$ is
induced by the orientation of $A\overrightarrow{r_1}$ and
$A\overrightarrow{r_2}$. Apply the Euclidean version of the
current lemma and construct the Euclidean rays
$A\overrightarrow{\tau_1}$ and $A\overrightarrow{\tau_2}$ as the
tangents to all Euclidean circles intersecting
$A\overrightarrow{\rho_1}$ and $A\overrightarrow{\rho_2}$ at
angles $\alpha_1$ and $\alpha_2$ respectively.

For each $j=1,2$ draw the circle $\kappa^{\tau}_j$ tangent to
$A\overrightarrow{\tau_j}$ at the point $A$ and orthogonal to
$\partial\hyperbolicplane$. In order to do that, define $O_j$ to
be the intersection point of the orthogonal bisector of segment
$AA^*$ and the line orthogonal to $A\overrightarrow{\tau_j}$ at
point $A$. The circle $\kappa^{\tau}_j$ is defined by its center
$O_j$ and its radius $|O_jA|$. Then, in the hyperbolic plane, the
ray $A\overrightarrow{t_j}$ is in fact the hyperbolic ray starting
from $A$ and lying on the hyperbolic geodesic $\kappa^{\tau}_j
\cap \hyperbolicplane$, in the direction induced by the tangent
Euclidean ray $A\overrightarrow{\tau_j}$.

In order to verify that we have constructed the right objects, we
define a suitable hyperbolic isometry. First, from the point $A^*$
draw the pair of tangents to $\partial\hyperbolicplane$ and then
draw the circle $\kappa$ with center $A^*$ so that it passes
through the touching points of the tangents with
$\partial\hyperbolicplane$. Denote by $I_{\kappa}$ the inversion
with respect to $\kappa$. Then $\partial\hyperbolicplane$ and
$\kappa$ are orthogonal, i.e.
$I_{\kappa}(\partial\hyperbolicplane) = \partial\hyperbolicplane$,
and $I_{\kappa}(A) = O_{\hyperbolicplane}$, where
$O_{\hyperbolicplane}$ is the Euclidean center of
$\partial\hyperbolicplane$. Furthermore, draw the Euclidean line
$\varrho$ through $O_{\hyperbolicplane}$ orthogonal to
$AO_{\hyperbolicplane}$. Let $R_{\varrho}$ be the Euclidean
reflection in $\varrho$. The composition $R_{\varrho} \circ
I_{\kappa} : \Euclideanplane \to \Euclideanplane$ restricts to an
orientation-preserving hyperbolic isometry of $\hyperbolicplane$.
Moreover, let $\varsigma : \Euclideanplane \to \Euclideanplane$ be
the Euclidean translation which maps point $A$ to point
$O_{\hyperbolicplane}$. Then $R_{\varrho} \circ
I_{\kappa}(A\overrightarrow{t_1}) =
\varsigma(A\overrightarrow{\tau_1}) \cap \hyperbolicplane, \,\,\,$
$R_{\varrho} \circ I_{\kappa}(A\overrightarrow{r_1}) =
\varsigma(A\overrightarrow{\rho_1}) \cap \hyperbolicplane, \,\,\,$
$R_{\varrho} \circ I_{\kappa}(A\overrightarrow{r_2}) =
\varsigma(A\overrightarrow{\rho_2}) \cap \hyperbolicplane $ and $
R_{\varrho} \circ I_{\kappa}(A\overrightarrow{t_2}) =
\varsigma(A\overrightarrow{\tau_2}) \cap \hyperbolicplane.$ Let
$\tilde{c}$ be a hyperbolic circle. Then there exists a Euclidean
circle $c$ such that $\tilde{c} = c \cap \hyperbolicplane$. In
particular, $\tilde{c}$ can be a geodesic, i.e. $c$ could be
orthogonal to $\partial\hyperbolicplane$. Assume $\tilde{c}$
satisfies the premises of the hyperbolic version of the current
lemma with respect to the configuration $A\overrightarrow{t_1},
\,\, A\overrightarrow{r_1}, \,\, A\overrightarrow{r_2}$ and
$A\overrightarrow{t_2}$. As $R_{\varrho} \circ I_{\kappa}$ is
conformal, the image $R_{\varrho} \circ I_{\kappa}(c) = c'$
satisfies the same premises with respect the Euclidean rays
$\varsigma(A\overrightarrow{\tau_1}), \,\,
\varsigma(A\overrightarrow{\rho_1}), \,\,
\varsigma(A\overrightarrow{\rho_2})$ and $
\varsigma(A\overrightarrow{\tau_2}).$ Therefore, the conclusions
of the Euclidean version of the current lemma apply to $c'$, and
thus the corresponding conclusions of the hyperbolic version apply
to $c$ and $\tilde{c}$.
\end{proof}

The next statement concerns geodesic triangles in $\PP$.

\begin{prop} \label{Prop_3_angles_triangle}
Let $\alpha, \beta$ and $\gamma \in (0, \pi)$. Then

\smallskip

\noindent {\bf 1.} $\alpha + \beta + \gamma = \pi$ if and only if
there exists a Euclidean triangle, unique up to Euclidean isometry
and scaling, with angles $\alpha, beta$ and $\gamma$;

\smallskip

\noindent {\bf 2.} $\alpha + \beta + \gamma < \pi$ if and only if
there exists a geodesic triangle in $\hyperbolicplane$, unique up
to hyperbolic isometry, with angles $\alpha, beta$ and $\gamma$.
\end{prop}

\begin{proof}
The Euclidean case is a standard elementary result from classical
planar geometry. That is why we focus on the hyperbolic case.

Let $\hyperbolicplane$ be the Poincar\'e disc model with
$\partial\hyperbolicplane$ being the unit circle, which is the
boundary at infinity of $\hyperbolicplane$. Let $A$ be an
arbitrary point in $\hyperbolicplane$. Draw two hyperbolic rays
$A\overrightarrow{r_1}$ and $A\overrightarrow{r_2}$, both starting
from $A$, so that the angle between them equals $\alpha$.
By applying the hyperbolic version of lemma
\ref{Lem_two_rays_and_one_circle} point 1, construct the auxiliary
hyperbolic rays $A\overrightarrow{t_1}$ and
$A\overrightarrow{t_2}$ for the angles $\beta$ and $\gamma$.
Denote by $\infty_1$ and $\infty_2$ the ideal points of
$A\overrightarrow{t_1}$ and $A\overrightarrow{t_2}$ respectively.
Draw the unique Euclidean circle $\kappa_{12}$ that passes through
$\infty_1$ and $\infty_2$, and is orthogonal to
$\partial\hyperbolicplane$. Since $\kappa_{12}$ is tangent to both
$A\overrightarrow{t_1}$ and $A\overrightarrow{t_2}$, again by
point 1 of lemma \ref{Lem_two_rays_and_one_circle}, its angles
with $A\overrightarrow{r_1}$ and $A\overrightarrow{r_2}$ are
$\beta$ and $\gamma$. Therefore if $B \in \hyperbolicplane$ is the
intersection point of $\kappa_{12}$ and $A\overrightarrow{r_1}$,
and $C \in \hyperbolicplane$ is the intersection point of
$\kappa_{12}$ and $A\overrightarrow{r_2}$, then the hyperbolic
triangle $\triangle ABC$ has angles $\alpha, \beta$ and $\gamma$
at the vertices $A, B$ and $C$ respectively. By construction
$\triangle ABC$ is unique up to a hyperbolic isometry.
\end{proof}

Next, we focus on decorated triangles. Let $\Delta=ijk$ be a
decorated triangle in $\PP$, with vertex circles $c_i, c_j, c_k$
(some of which could be points) and a face circle $c_{\Delta}$
(see decorated triangle $\Delta=ijk$ on figure \ref{Fig2}a). From
now on, let $V_{\Delta} = \{i,j,k\}$ be the set of vertices and
$E_{\Delta}=\{ij, jk, ki\}$ be the set of edges of $\Delta$. Let
$u \neq v \neq w \in V_{\Delta}$ be some permutation of the
vertices $i, j$ and $k$. In relation to definition
\ref{Def_angle_between_2_face_circles}, denote by $P_{u}^{uv}$ and
$P_{v}^{uv}$ the two intersection points of the face circle
$c_{\Delta}$ with the edge $uv$ of $\Delta$. Point $P_{u}^{uv}$ is
the closer one to vertex $u$, while $P_{v}^{uv}$ is the closer one
to vertex $v$.


Define $\alpha^{\Delta}_{uv} \in (0,\pi)$ to be the angle at
$P_{u}^{uv}$ (or equivalently at $P_{v}^{uv}$) between the
geodesic segment $uv$ and the face-circle $c_{\Delta}$ measured
inside $c_{\Delta}$ and outside the (undecorated) triangle
$\triangle ijk$ (see figure \ref{Fig2}a). The following statement
follows directly from definition
\ref{Def_angle_between_2_face_circles} and corollary
\ref{Cor_two_adjacent_triangles}.

\begin{prop} \label{Prop_alpha_alpha_equals_theta}
Let $\Delta$ and $\Delta'$ be two compatibly adjacent decorated
triangles in $\PP$, sharing a common edge $ij$. Then $\theta_{ij}
= \alpha^{\Delta}_{ij} + \alpha^{\Delta'}_{ij}$.
\end{prop}

\noindent Furthermore, let $\beta^{\Delta}_w = \measuredangle
uwv$, i.e. the angle of $\Delta$ at the vertex $w \in V_{\Delta}$.
Consequently, we conclude that a decorated triangle $\Delta=ijk$
in $\PP$ determines two groups of three angles each
$$(\alpha^{\Delta},\beta^{\Delta}) = (\alpha^{\Delta}_{ij},
\alpha^{\Delta}_{jk}, \alpha^{\Delta}_{ki}, \beta^{\Delta}_{k},
\beta^{\Delta}_{i}, \beta^{\Delta}_{j})$$ satisfying the
inequalities
\begin{align} 
&0 \leq \alpha^{\Delta}_{uv} < \pi \,\,\,\,\, \text{for} \,\,\, uv
\in \{ij,jk,ki\},  \,\,\,\,\,\,\,\,\,\,\, 0 < \beta^{\Delta}_w <
\pi \,\,\,\,\,\,
\text{for} \,\, w \in \{i,j,k\}\label{Formula_Inequalities_bounded_angles}\\
          &\beta^{\Delta}_k + \alpha^{\Delta}_{jk}
          + \alpha^{\Delta}_{ki} \leq \pi, \,\,\,\,\,\,
\alpha^{\Delta}_{ij} + \beta^{\Delta}_{i}
          + \alpha^{\Delta}_{ki} \leq \pi, \,\,\,\,\,\,
\alpha^{\Delta}_{ij} + \alpha^{\Delta}_{jk} + \beta^{\Delta}_j
\leq \pi, \label{Formula_Inequalities_between_angles_alpha_beta}
\end{align}
as well as the restriction
\begin{align} 
&\beta^{\Delta}_k + \beta^{\Delta}_i + \beta^{\Delta}_j = \pi
\,\,\,\,\, \text{whenever} \,\,\, \PP=\Euclideanplane \,\,\,\,
\text{or}   \label{Formula_angles_euclidean_triangle} \\
&\beta^{\Delta}_k + \beta^{\Delta}_i + \beta^{\Delta}_j < \pi
\,\,\,\,\, \text{whenever} \,\,\, \PP=\hyperbolicplane.
\label{Formula_angles_hyperbolic_triangle}
\end{align}
We call the six angles $(\alpha^{\Delta}, \beta^{\Delta})$ the
\emph{angles of the decorated triangle} $\Delta$. On figure
\ref{Fig2}a they are included in the labels of the triangular face
$\Delta=ijk$. Next, define $\Alocaldelta$ to be the set of all six
real numbers $(\alpha^{\Delta}, \beta^{\Delta})$ which satisfy
conditions (\ref{Formula_Inequalities_bounded_angles}),
(\ref{Formula_Inequalities_between_angles_alpha_beta}) and either
(\ref{Formula_angles_euclidean_triangle}), when
$\PP=\Euclideanplane$, or
(\ref{Formula_angles_hyperbolic_triangle}) when
$\PP=\hyperbolicplane$. Notice that for a decorated triangles it
is possible that some of its vertex circles are collapsed to
points or some pairs of vertex circles are tangent. Then in the
case of collapsed vertex circles the corresponding inequalities
from (\ref{Formula_Inequalities_between_angles_alpha_beta}) become
identities, and in the case of a tangency the corresponding
$\alpha^{\Delta}_{uv}$ becomes $0$. Clearly, the angles of a
decorated triangle belong to the set $\Alocaldelta$. The converse
is also true.

\begin{prop} \label{Prop_angles_to_dec_triangle}
Let $(\alpha^{\Delta}_{ij}, \alpha^{\Delta}_{jk},
\alpha^{\Delta}_{ki}, \beta^{\Delta}_{k}, \beta^{\Delta}_{i},
\beta^{\Delta}_{j}) \in \Alocaldelta$. Then, these six angles
determine a decorated triangle $\Delta=ijk$ in $\PP$. 
Furthermore, if $\PP=\hyperbolicplane$, then $\Delta$ is unique up
to hyperbolic isometry. If $\PP=\Euclideanplane$, then $\Delta$ is
unique up to Euclidean isometry and scaling. Conversely, the six
angles of a decorated triangle belong to $\Alocaldelta.$
\end{prop}

\begin{proof}
As discussed above, the set $\Alocaldelta$ is defined so that the
six angles $(\alpha^{\Delta},\beta^{\Delta})$ of any decorated
triangle satisfy the defining conditions of $\Alocaldelta$. That
is way we focus on the proof of the converse statement.

\smallskip
\noindent\emph{Euclidean case.} The Euclidean case is the simpler
one. Here is a compass and straightedge construction. On the plane
$\Euclideanplane$, draw a triangle $\triangle IJK$ with interior
angles $\beta^{\Delta}_i, \beta^{\Delta}_j$ and $\beta^{\Delta}_k$
at the vertices $I, J$ and $K$ respectively. Observe that
$\triangle IJK$ is unique up to similarity. Let $c_{\Delta}$ be
its superscribed circle, where $O_{\Delta}$ is its center. Let
$M_k, M_i$ and $M_j$ be the midpoints of edges $IJ, JK$ and $KI$
respectively. Let $N_k$ be the intersection point of the circle
$c_{\Delta}$ with the line $O_{\Delta}M_k$ (which is orthogonal to
$IJ$), so that $N_k$ and the vertex $K$ are on different sides of
the line $IJ$. Analogously, construct the points $N_i$ and $N_j$.
Take two points $P^{ij}_i$ and $P^{ij}_j$ on $c_{\Delta}$ such
that $\measuredangle N_kO_{\Delta}P^{ij}_i= \measuredangle
N_kO_{\Delta}P^{ij}_j = \alpha^{\Delta}_{ij}$. Analogously,
construct the points $P^{jk}_j$ and $P^{jk}_k$ using angle
$\alpha^{\Delta}_{jk}$, as well as the points $P^{ki}_k$ and
$P^{ki}_i$ via angle $\alpha^{\Delta}_{ki}$. Let the line
$P^{ij}_iP^{ij}_j$ intersects the lines $P^{ij}_jP^{ij}_k$ and
$P^{ij}_kP^{ij}_i$ at the points $i$ and $j$ respectively. Let $k$
be the intersection of lines $P^{jk}_jP^{jk}_k$ and
$P^{ki}_kP^{ki}_i$. We obtain the triangle $\triangle ijk$. Draw
the circles $c_i, c_j$ and $c_k$ centered at the vertices $i, j$
and $k$ respectively so that each of them is orthogonal to
$c_{\Delta}$. Thus, we have constructed the desired decorated
triangle. By construction, it is unique up to similarity.

\smallskip

\noindent \emph{Hyperbolic case.} With the help of proposition
\ref{Prop_3_angles_triangle}, construct a hyperbolic triangle
$\triangle ijk$ in $\hyperbolicplane$ with angles
$\beta^{\Delta}_{i}, \beta^{\Delta}_{j}, \beta^{\Delta}_{k}$ at
the vertices $i, j, k$ respectively. $\triangle ijk$ is unique up
to $\hyperbolicplane-$isometry. Then apply lemma
\ref{Lem_two_rays_and_one_circle} point 1 to construct the
auxiliary hyperbolic rays $i\overrightarrow{t_j}$ and
$i\overrightarrow{t_k}$ playing the role of
$A\overrightarrow{t_{1}}$ and $A\overrightarrow{t_2}$ for the pair
of rays $\overrightarrow{ij}$ and $\overrightarrow{ik}$ with
respective angles $\alpha^{\Delta}_{ij}$ and
$\alpha^{\Delta}_{ki}$. Analogously, construct the auxiliary rays
$j\overrightarrow{\tau_k}$ and $j\overrightarrow{\tau_i}$ playing
the role of $A\overrightarrow{t_{1}}$ and $A\overrightarrow{t_2}$
for the pair of rays $\overrightarrow{jk}$ and
$\overrightarrow{ji}$ with respective angles
$\alpha^{\Delta}_{jk}$ and $\alpha^{\Delta}_{ij}$. In the
underlying Euclidean geometry, $i\overrightarrow{t_j}, \,
i\overrightarrow{t_k}$ and $j\overrightarrow{\tau_i}$ are three
directed circular arcs, determining three respective circles
$\kappa^{t}_j, \, \kappa^{t}_k$ and $\kappa^{\tau}_i$ orthogonal
to $\partial\hyperbolicplane$. By the famous Apollonius' problem,
there exists a unique Euclidean circle $\tilde{c}_{\Delta}$
tangent to the three circles $\kappa^{t}_j, \, \kappa^{t}_k$ and
$\kappa^{\tau}_i$, while contained in the domain they cut out
containing $\triangle ijk$. Moreover, $\tilde{c}_{\Delta}$ is
ruler and compass constructible. Let us go back to the geometry of
$\hyperbolicplane$. Then $c_{\Delta} = \tilde{c}_{\Delta} \cap
\hyperbolicplane$ is a hyperbolic circle tangent to the hyperbolic
rays $i\overrightarrow{t_j}, \, i\overrightarrow{t_k}$ and
$j\overrightarrow{\tau_i}$. By point 1 of lemma
\ref{Lem_two_rays_and_one_circle}, the intersection angles of
$c_{\Delta}$ with $\overrightarrow{ij}$ and $\overrightarrow{ik}$
are $\alpha^{\Delta}_{ij}$ and $\alpha^{\Delta}_{ki}$.
Consequently, since $c_{\Delta}$ is tangent to
$j\overrightarrow{\tau_i}$ and its angle of intersection with
$\overrightarrow{ji}$ is $\alpha^{\Delta}_{ij}$, by point 2 of
lemma \ref{Lem_two_rays_and_one_circle} the circle $c_{\Delta}$ is
tangent to $j\overrightarrow{\tau_k}$ and its intersection angle
with $\overrightarrow{jk}$ is necessarily $\alpha^{\Delta}_{jk}.$
Thus, the intersection angles of $c_{\Delta}$ with the geodesic
edges $ij, jk$ and $ki$ of triangle $\triangle ijk$ are
$\alpha^{\Delta}_{ij}, \alpha^{\Delta}_{jk}$ and
$\alpha^{\Delta}_{ki}$ as required. Therefore $c_{\Delta}$ is the
face circle we have been looking for. Now, to finish the
construction, for each vertex $u$ of the triangle $\Delta$ we
simply draw the unique circle $c_u$ centered at $u$ and orthogonal
to $c_{\Delta}$. Since after fixing the hyperbolic triangle
$\triangle ijk$, the face circle $c_{\Delta}$ is constructed in a
unique way, the decorated triangle $\Delta=ijk$ is unique up to a
hyperbolic isometry and has prescribed angles $(\alpha^{\Delta},
\beta^{\Delta}) \in \Alocaldelta$.
\end{proof}

Before we continue, we make the following assumption. Let
$\Delta=ijk$ be a topological triangle with $V_{\Delta} =
\{i,j,k\}$ and $E_{\Delta} = \{ij,jk,ki\}$. By proposition
\ref{Prop_angles_to_dec_triangle}, an assignment of six angles
from $\Alocaldelta$ turns $\Delta$ into a unique decorated
triangle. From now on, we assume that any topological triangle
$\Delta$ comes with a priori prescribed (combinatorial) data, in
the form of a partition of its set of edges $E_{\Delta} =
E^1_{\Delta} \sqcup E^0_{\Delta}$ and a partition of its set of
vertices $V_{\Delta} = V^1_{\Delta} \sqcup V^0_{\Delta}$. These
partitions tell us that whenever we realize $\Delta$
geometrically, we always have to make sure that only the vertices
from $V_{\Delta}^0$ necessarily have vertex circles collapsed to
points, and that exactly the edges from $E^0_{\Delta}$ correspond
to pairs of touching vertex circles. Then, depending on this data,
conditions (2) and (3) defining $\ADelta = \Alocaldelta$ may
include both strict inequalities and equalities but never
non-strict inequalities.






\section{The space of generalized circle patterns} \label{Sec_Generalized_space_of_patterns}


Let $S$ be a fixed compact surface with a cell complex
$\cellcomplex = (V, E, F)$ on it. Let us first subdivide
$\cellcomplex$ until we obtain a topological triangulation of $S$.
Define $\Triang = (V, E_T, F_T)$ by subdividing the faces of
$\cellcomplex$ via diagonals so that no two diagonals intersect
except possibly at one common vertex (see figure \ref{Fig3}b).
More precisely, let $f \in F$ be a face of $\cellcomplex$ with
vertices $i_1,..., i_n$. The subscripts represent the cyclic order
of the vertices, i.e. $f=i_1i_2...i_n$ so that each $i_si_{s+1}$
is an edge of $\cellcomplex$, i.e. $i_si_{s+1} \in E$ for
$i=1,..,n$ with $n+1 = 1$. Then one way to subdivide $f$ is to
introduce the new edges $i_1i_3, \, i_1i_4,\, ..., \, i_1i_{n-1}$
and put them into the set of new edges $E_{\pi}$. Thus $f$ gets
subdivided into the topological triangles $\Delta=i_1i_si_{s+1}
\in F_T$ for $s=2,.., n-1$. Observe that no new vertices are
introduced. As a result of this procedure we obtain the desired
triangulation $\Triang = (V, E_T, F_T)$, where
 $E_T = E \cup E_{\pi} = E_0 \cup E_1 \cup E_{\pi}$. On figure
 \ref{Fig3}b the edges of $\cellcomplex$ are in solid black, while
 the new edges of $\Triang$ (the ones from $E_{\pi}$) are dashed.

In order to understand better the space of hyper-ideal circle
patterns with combinatorics $\cellcomplex$, first we would like to
introduce the more general space of generalized hyper-ideal circle
patterns with combinatorics $\Triang$, whose edges may not
necessarily satisfy the local Delaunay property from definition
\ref{Def_local_Del_property}. Later, the space we are interested
in will turn out to be a submanifold embedded in the generalized
space.

To define a generalized hyper-ideal circle pattern, one can simply
assign appropriate edge-lengths $l : E_T \to (0,\infty)$ to the
edges of $\Triang$ and radii $r : V \to (0,\infty)$ to its
vertices. The assignment of edge-lengths $l$ represents the
underlying cone-metric $d$ on $S$ by associating to each
triangular face of $\Triang$ an actual geometric triangle, unique
up to isometry. Moreover, $l$ represents not only $d$ but, in
fact, the whole class of \emph{marked} cone-metrics isometric to
$d$. We say marked because, by fixing $\Triang$, the isometry
class of $d$ is defined via all isometries between cone-metrics
which preserve the combinatorics of $\Triang$. Hence, these are
the isometries isotopic to identity on $S$. This last observation
explains the expression ``isotopic to identity" in the statements
of theorem \ref{Thm_main}. Furthermore, the assignment of radii
$r$ makes each triangle $\Delta$, geometrized by $l$, into a
decorated triangle by making it possible to draw the three vertex
circles of $\Delta$ and then uniquely determine the orthogonal
face circle.

One way to define the space of generalized hyper-ideal circle
patterns is presented next. An edge-length and radius assignment
$(l,r) \in \reals^{E_1 \cup E_{\pi}} \times \reals^{V_1}$ belongs
to $\ER$ exactly when

\medskip

\begin{itemize}

\item $l_{ij} > 0$ for all $ij \, \in \, E_1 \cup E_{\pi} \,$ and
 $\, r_k > 0$ for all $k \in V_1$;

\medskip

\item Let $r_k = 0$ for all $k \in V_0 \,$ and $\, l_{ij} = r_i +
r_j$ for all $ij \in E_0$. Notice $l_{ij} > 0$ since $i, j \in
V_1$ by assumption;

\medskip

\item $l_{ij} > r_i + r_j$ for all $ij \in E_1 \cup E_{\pi}$;

\medskip

\item $l_{ij} < l_{jk} + l_{ki}, \,\,\,\,  l_{jk} < l_{ki} +
l_{ij}, \,\,\,\, l_{ki} < l_{ij} + l_{jk}$ for all $\Delta=ijk \in
F_T$.

\end{itemize}

\medskip

\noindent As usual, we are going to assume that $\reals^{E_1 \cup
E_{\pi}} \times \reals^{V_1}$ is a vector subspace of
$\reals^{E_T}\times\reals^{V}$ by assuming that any $(l, r) \in
\reals^{E_1 \cup E_{\pi}} \times \reals^{V_1}$ is extended to a
point in $\reals^{E_T}\times\reals^{V}$ by letting $r_k=0$ for all
$k \in V \setminus V_1 = V_0$, and $l_{ij} = r_i + r_j$ for all
$ij \in  E_0$.

By definition, the space $\ER$ is clearly an open convex polytope
of $\reals^{E_1 \cup E_{\pi}} \times \reals^{V_1}$ and hence a
convex polytope of $\reals^{E_T} \times \reals^{V}$ of dimension
\begin{equation} \label{Eqn_LR_dimension}
\dim \ER = \dim \Bigl( \reals^{E_1 \cup E_{\pi}} \times
\reals^{V_1} \Bigr) = |E_1| + |E_{\pi}| + |V_1|.
\end{equation}
In the Euclidean case, there are scaling transformations of a
circle pattern which do not exist in the hyperbolic case. In other
words, one can zoom in and out a circle pattern but that does not
change its essential geometry. In terms of edge-lengths and radii,
such a geometric scaling is equivalent to the $\reals-$action
$(l,r) \mapsto (e^t l, e^t r)$ on $\ER$ for $t \in \reals$. To
account for scaling, in the Euclidean case we will work with the
space
\begin{equation*}
\ER_1 = \Bigl\{ (l,r) \in \ER \,\, \Big| \,\, \sum_{ij \in E_T}
l_{ij} = 1 \Bigr\}, \,\,\,\,\,  \dim \ER_1 = |E_1| + |E_{\pi}| +
|V_1| -1.
\end{equation*}
In the hyperbolic case, we simply set $\ER_1 = \ER$, so that we
can use $\ER_1$ as a common notation for both Euclidean and
hyperbolic patterns. For $\Delta=ijk \in F_T$ let $V_{\Delta}=\{i,
j, k\}$ and $E_{\Delta}=\{ij, jk, ki\}$. Furthermore, let
$E^1_{\Delta} = E_{\Delta} \cap \big(E_1 \cup E_{\pi}\big)$ and
$E^0_{\Delta} = E_{\Delta} \setminus E^1_{\Delta}$. Together with
that, let $V^1_{\Delta} = V_{\Delta} \cap V_1$ and $V^0_{\Delta} =
V_{\Delta} \setminus V^1_{\Delta}$. We also use the notation
$(l,r)_{\Delta} = \bigl(l_{ij}, l_{jk}, l_{ki}, r_k, r_i,
r_j\bigr)$. So the space $\ER$ for just one triangle is denoted by
$\ERD$. More precisely, $\bigl(l_{ij}, l_{jk}, l_{ki}, r_k, r_i,
r_j\bigr) \in \ERD$ exactly when

\medskip

\begin{itemize}

\item $l_{ij} > 0$ for all $ij \, \in \, E_{\Delta}$ and
 $\, r_k > 0$ for all $k \in V_{\Delta}\cap V_1$;

\medskip

\item $r_k = 0$ for all $k \in V_{\Delta} \cap V_0 \,$ and $\,
l_{ij} = r_i + r_j$ for all $ij \in E_{\Delta} \cap E_0$.

\medskip

\item $l_{ij} > r_i + r_j$ for all $ij \in E_{\Delta} \cap (E_1
\cup E_{\pi})$;

\medskip

\item $l_{ij} < l_{jk} + l_{ki}, \,\,\,\,  l_{jk} < l_{ki} +
l_{ij}, \,\,\,\, l_{ki} < l_{ij} + l_{jk}$.

\end{itemize}

\medskip

\noindent Just like before, $\reals$ acts on $\ERD$ by
$(l,r)_{\Delta} \mapsto (e^t l, e^t r)_{\Delta}$ for $t \in
\reals$. To factor out this action, we could similarly consider
the space $$\ER_{1, \Delta} = \Big\{ \, (l,r)_{\Delta} \in \ERD
\,\,  \big{|} \,\,  \sum_{ij \in E_{\Delta}} l_{ij} = 1 \Big\}.$$
\begin{prop} \label{Prop_corr_triangles_ER}
The lengths of the three edges and the radii of the three vertex
circles of a decorated triangle in $\PP$ belong to the set $\ERD$.
Conversely, any six numbers $(l,r)$ from the set $\ERD$ determine
a decorated triangle in $\PP$ uniquely up to isometry.
\end{prop}
\begin{proof}
From the discussion above, the set $\ERD$ is defined so that the
six-tuple of its edge-lengths and vertex circle radii always
belong to $\ERD$. Conversely, let $(l,r) \in \ERD$. Since the
three positive numbers $l_{ij},l_{jk},l_{ki}$ satisfy the three
triangle inequalities, one can draw in $\PP$ a unique up to
isometry triangle with these numbers as edge-lengths. After that
one can use $r_i,r_j,r_k$ to draw the three vertex circles. The
restrictions imposed on $(l,r)$ guarantee that the interiors of
the circles do not intersect (but may touch). Finally, there is a
unique forth circle (the face circle) orthogonal to the three
vertex circles. This construction works even if some radii are
zero. \end{proof} Although quite simple and natural, the
description of circle patterns in terms of edge-lengths and radii
is not suitable for our purposes. It will mostly paly an
intermediate role. In order to motivate the right parametrization
for the space of circle patterns, we use hyperbolic geometry.

\section{Decorated triangles and hyper-ideal tetrahedra} \label{Sec_hyperbolic_tetrahedra}

Our next step is to establish a natural correspondence between
decorated triangles in $\PP$ and hyper-ideal tetrahedra in
$\hyperbolicspace$. 

\smallskip

\noindent {\bf Construction \ref{Sec_hyperbolic_tetrahedra}.1.}
Let $\Hor$ be a fixed horosphere in $\hyperbolicspace$. Then
$\Hor$ with the hyperbolic metric restricted on it is isometric to
the Euclidean plane $\Euclideanplane$. Define the projection
$E_{proj} : \Hor \to
\partial\hyperbolicspace$ by following down to the ideal boundary
the geodesics emanating from the ideal point of contact between
$\Hor$ and $\partial\hyperbolicspace$. More explicitly, if
$\hyperbolicspace$ is the upper-half space, then by applying a
hyperbolic isometry, we can arrange for $\Hor$ to be the plane
$\reals^2 \times\{1\}$, which is parallel to
$\partial\hyperbolicspace = \reals^2\times\{0\}$ and lying inside
$\hyperbolicspace$. Observe that $\Hor$ is at Euclidean distance
$1$ from $\partial\hyperbolicspace$. Then $E_{proj}$ is simply the
usual orthogonal projection onto $\partial\hyperbolicspace =
\reals^2\times\{0\}$, restricted to $\Hor$. Consequently,
$E_{proj}$ can also be interpreted as the Euclidean vertical
translation from $\Hor$ down to the parallel plane
$\partial\hyperbolicspace$ (i.e. just changing the third
coordinate from $1$ to $0$).

\begin{figure}[ht]
\centering
\includegraphics[width=14cm]{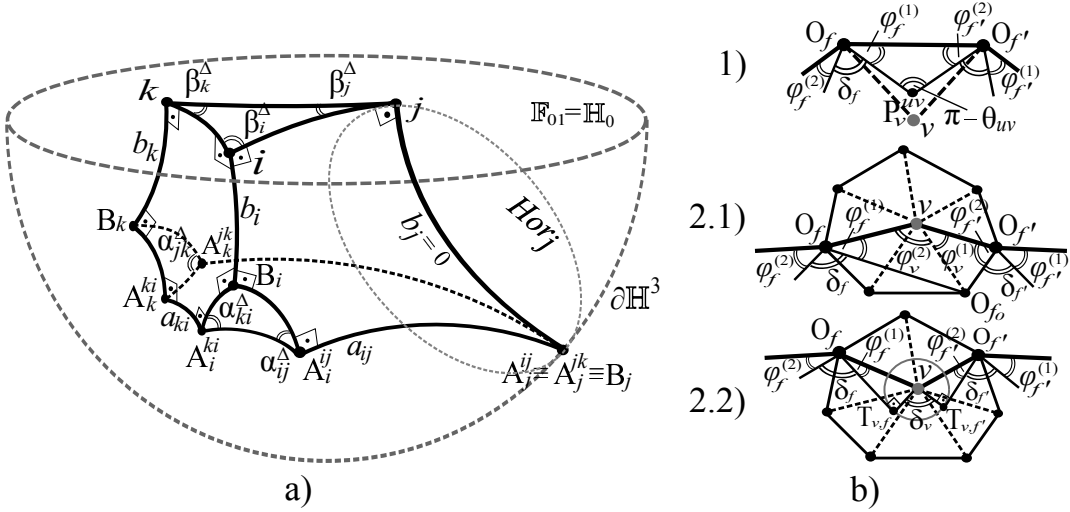}
\caption{a) A hyper-ideal tetrahedron $\tau_{\Delta}$; b) The
three different cases of edges on the boundary of an admissible
domain $\Omega$.} \label{Fig4}
\end{figure}

Similarly, let $\Hplane$ be a hyperbolic plane in
$\hyperbolicspace$. Then naturally, $\Hplane$ is isometric to the
hyperbolic plane $\hyperbolicplane$. Again define a projection
$H_{proj} : \Hplane \to \partial\hyperbolicspace$ by following
down to the ideal boundary the geodesics orthogonal to $\Hplane$
on one side of it. In the upper-half space model of
$\hyperbolicspace$ we can use a hyperbolic isometry to arrange for
$\Hplane$ to be the vertical orthogonal half-plane $\Hplane =
\{(x,y,z) \in \reals^3 \, : \, \, y=0, \,\, z > 0 \}$. Then
$H_{proj}$ is simply the ninety-degree Euclidean rotation around
the $x-$axis which rotates $\Hplane$ to the horizontal half-plane
$\{(x,y,0) \in \reals^3 \, : \,\, y>0\} \, \subset \,
\partial\hyperbolicspace$. The map $H_{proj}$ is depicted on
figure \ref{Fig2}b.

Use the common symbol $\Femb$ to denote both the horosphere $\Hor$
and the hyperbolic plane $\Hplane$. Similarly, use the notation
$F_{proj}$ for both maps $E_{proj}$ and $H_{proj}$. Now, let
$\Delta=ijk$ be a decorated triangle in $\PP$, with vertex circles
$c_i, c_j, c_k$ and a face circle $c_{\Delta}$. Then without loss
of generality we can think that $\Delta$ is in fact in $\Femb$.
Let $\hat{\Delta} = F_{proj}(\Delta) \, \subset\,
\partial\hyperbolicspace$ with vertex circles
$F_{proj}(c_i), \, F_{proj}(c_j), \, F_{proj}(c_k)$ and a face
circle $F_{proj}(c_{\Delta})$. Due to the nature of $F_{proj}$
(compare with figure \ref{Fig2}b), the decorated triangle
$\hat{\Delta}$ is an identical copy of $\Delta$. The edges of
$\hat{\Delta}$ can be completed to either
circles or straight lines. 
Each straight line 
can be extended vertically to a
half-plane 
orthogonal to $\partial\hyperbolicspace$. Thus, it gives rise to a
hyperbolic plane in $\hyperbolicspace$. Analogously, each circle,
whether a vertex circle, the face circle or a circle coming from
an edge of $\hat{\Delta}$, can be extended to a half sphere in
$\hyperbolicspace$, centered at a point on
$\partial\hyperbolicspace$. Each such half sphere is a hyperbolic
plane in $\hyperbolicspace$. To fix notations, for each $u \in
V_{\Delta}$, the hyperbolic plane extending the projected vertex
circle $F_{proj}(c_u)$ is denoted by $\tilde{c}_{u}$ (see figure
\ref{Fig2}b) and the hyperbolic plane extending the projected face
circle $F_{proj}(c_{\Delta})$ is denoted by $\tilde{c}_{\Delta}$.
Furthermore, for each edge $uv \in E_{\Delta}$, the completion of
$F_{proj}(uv)$ to either a straight line or a whole circle
(whichever applies) extends to the hyperbolic plane
$\widetilde{uv}$. As a result of this construction, to each
decorated triangle $\Delta \subset \Femb$ we associate the finite
set of all hyperbolic planes constructed above. In the case of
$\PP = \hyperbolicplane$ we add to that set the plane $\Hplane$.
Subsequently, all these hyperbolic planes bound a convex
hyperbolic polyhedron of finite volume, denoted by $\tau_{\Delta}$
(see figure \ref{Fig4}a).

For each vertex $u \in V_{\Delta}$ let $B_u$ be the intersection
point of the three hyperbolic planes $\widetilde{wu}, \,
\widetilde{uv}$ and $\tilde{c}_u$, where $u \neq v \neq w \in
V_{\Delta}$. Furthermore, let $A^{uv}_u$ be the intersection point
of the three hyperbolic planes $\tilde{c}_{\Delta}, \,
\tilde{c}_u$ and $\widetilde{uv}$, and let $A^{uv}_v$ be the
intersection point of $\tilde{c}_{\Delta}, \, \tilde{c}_v$ and
$\widetilde{uv}$. In the case when $c_{u}$ is collapsed to a
point, the hyperbolic plane $\tilde{c}_u$ degenerates to an ideal
point, and thus the three points $B_u, \, A^{wu}_u$ and $A^{uv}_u$
merge into one ideal point of $\tau_{\Delta}$. Also, if the vertex
circles $c_u$ and $c_v$ touch, then the two points $A^{uv}_u$ and
$A^{uv}_v$ become one ideal point of $\tau_{\Delta}$. When $\Femb
= \Hplane$, in the most general case, the polyhedron
$\tau_{\Delta}$ has the combinatorics of a tetrahedron with all
four vertices truncated so that the four \emph{truncating faces}
are disjoint triangles with no vertices in common. One of these
truncating faces is the decorated triangle $\Delta=ijk \subset
\Hplane$. Whenever a vertex circle of the corresponding decorated
triangle is shrunk to a point, the tetrahedral vertex it
corresponds to is not truncated. It becomes an ideal vertex (see
figure \ref{Fig4}a). Whenever two vertex circles touch, the
corresponding truncating faces share a common ideal vertex.
Furthermore, when $\Femb = \Hor$, in the most general case, the
polyhedron $\tau_{\Delta}$ has the combinatorics of a tetrahedron
with one ideal vertex and the three remaining vertices truncated
so that the three truncating faces are again disjoint triangles
with no vertices in common. The rest of the cases are analogous to
the ones discussed above. Notice that for $\Femb = \Hor$, the
polyhedron $\tau_{\Delta}$ is \emph{decorated} with the horosphere
$\Hor$, which we call a \emph{decorating horosphere} and the
decorated triangle $\Delta=ijk$ lies on it.

Now, assume that for $u \in V_{\Delta}$ the vertex circle $c_u$ is
a point. Then we add to the polyhedron $\tau_{\Delta}$ the unique
horosphere $Hor_u$ tangent to $\partial\hyperbolicspace$ at the
ideal point $B_u\equiv A^{wu}_u \equiv A^{uv}_u$ and at the same
time tangent to $\Femb$ at the point $u$ (as shown on figure
\ref{Fig4}a).
\newline \noindent{\bf End of construction
\ref{Sec_hyperbolic_tetrahedra}.1.}
\begin{lem} \label{Lem_tetrahedron_angles_edge_lengths}
Let $\tau_{\Delta}$ be the polyhedron obtained from the decorated
triangle $\Delta=ijk$ in $\Femb \subset \hyperbolicspace$,
according to construction \ref{Sec_hyperbolic_tetrahedra}.1. Then
for each $u \neq v \neq w \in V_{\Delta}$ the following statements
hold (see figure \ref{Fig4}a):

\smallskip
\noindent {\bf 1.}  The geodesic edge $A^{uv}_uA^{uv}_v$ of the
polyhedron $\tau_{\Delta}$ is orthogonal to both truncating
triangular faces $A^{wu}_uA^{uv}_uB_u \,\subset \, \tilde{c}_{u}$
and $A^{uv}_vA^{vw}_vB_v \,\subset \, \tilde{c}_v$. Its hyperbolic
length is denoted by $a_{uv} =
l_{\hyperbolicspace}(A^{uv}_uA^{uv}_v)$. The interior dihedral
angle of $\tau_{\Delta}$ at the edge $A^{uv}_uA^{uv}_v$ is equal
to the value $\alpha^{\Delta}_{uv} = \angl A^{wu}_uA^{uv}_uB_u =
\angl A^{vw}_vA^{uv}_vB_v $.

\smallskip
\noindent {\bf 2.} The geodesic edge $uB_u$ of
$\tau_{\Delta}$ is orthogonal to the truncating face
$A^{wu}_uA^{uv}_uB_u \,\,\subset \,\, \tilde{c}_{u}$ and to
$\Femb$. Its length is denoted by $l_{\hyperbolicspace}(uB_u) =
b_{u}$. The interior dihedral angle of $\tau_{\Delta}$ at the edge
$uB_u$ is equal to $\beta^{\Delta}_{u} = \angl A^{wu}_uB_uA^{uv}_u
= \angl wuv$.

\smallskip
\noindent {\bf 3.} Let $c_u$ be a point. If $c_v$ has a non-zero
radius, then the dihedral angle at the edge $A^{uv}_uA^{uv}_v$ is
equal to $\alpha^{\Delta}_{uv} = \angl A^{vw}_vA^{uv}_vB_v$ and
the edge itself is perpendicular to $Hor_u$ and $\tilde{c}_v$. The
oriented distance between $Hor_u$ and $\tilde{c}_v$ along the
geodesic $A^{uv}_uA^{uv}_v$ is denoted by $a_{uv}$, where its sign
is positive whenever $Hor_u$ and $\tilde{c}_v$ are disjoint, and
negative if they intersect. If $c_v$ is also a point then the
dihedral angle at the edge $A^{uv}_uA^{uv}_v$ is still
$\alpha^{\Delta}_{uv}$ and the edge is perpendicular to both
$Hor_u$ and $Hor_v$. The oriented distance between $Hor_u$ and
$Hor_v$ along $A^{uv}_uA^{uv}_v$ is $a_{uv}$, with a positive sign
if the two horospheres are disjoint and negative otherwise.

\smallskip
\noindent {\bf 4.} If $c_u$ is a point, then $b_u = 0$. The edge
$uB_u$ is orthogonal to $Hor_u$ and $\Femb$ and its interior
dihedral angle is $\angl wuv = \beta^{\Delta}_u$.

\smallskip
\noindent {\bf 5.} If $c_u$ and $c_v$ touch, then $A^{uv}_u\equiv
A^{uv}_v \, \in \, \partial\hyperbolicspace$ and so $a_{uv}=0$.
Moreover, the dihedral angle at that ideal point is
$\alpha^{\Delta}_{uv}=0$.
\end{lem}
\begin{proof}
The proof is a straightforward consequence of the conformal
properties of the upper half-space model of $\hyperbolicspace$,
combined with construction \ref{Sec_hyperbolic_tetrahedra}.1
above. 
\end{proof}
We fix some terminology. The edges $A^{ij}_iA^{ij}_j, \,\,
A^{jk}_jA^{jk}_k$ and $ A^{ki}_kA^{ki}_i$, as well as the edges
$iB_i, \, jB_j$ and $kB_k$ of the polyhedron $\tau_{\Delta}$ are
called \emph{principal edges}. The rest of the edges are called
\emph{auxiliary edges}. The lengths of the principal edges are
called \emph{principal edge-lengths} of $\tau_{\Delta}$. For short
sometimes we will also call them just \emph{edge-lengths} of
$\tau_{\Delta}.$ The interior dihedral angles at the principal
edges are called \emph{principal dihedral angles} of
$\tau_{\Delta}$. For short, often we will call them simply
\emph{dihedral angles} of $\tau_{\Delta}$. Observe that the
dihedral angles at the auxiliary edges are all equal to $\pi/2$.

\begin{Def} \label{Def_hyperideal_tetrahedron}
A \emph{hyper-ideal tetrahedron} (see \cite{Sch1, S} and figure
\ref{Fig4}a) is a geodesic polyhedron in $\hyperbolicspace$ that
has the combinatorics of a tetrahedron with some (possibly all) of
its vertices truncated by triangular \emph{truncating faces}. Each
truncating face is orthogonal to the faces and the edges it
truncates. Furthermore, a pair of truncating faces either do not
intersect or share only one vertex. Finally, the non-truncated
vertices are all ideal.
\end{Def}

The polyhedron $\tau_{\Delta}$, constructed above, is a
hyper-ideal tetrahedron. This terminology comes from the
interpretation that in the Klein projective model or the Minkowski
space-time model of $\hyperbolicspace$ \cite{ThuBook, BenPetr, S},
$\tau_{\Delta}$ can be represented by an actual tetrahedron with
some vertices lying outside $\hyperbolicspace$ (hence the term
\emph{hyper-ideal vertices}). The dual to each hyper-ideal vertex
is the orthogonal truncating plane.

\medskip

\noindent{\bf Construction \ref{Sec_hyperbolic_tetrahedra}.2.} We
already know how to construct a hyper-ideal tetrahedron from a
decorated triangle. Now we explain how to do the opposite. That
is, we can take a hyper-ideal tetrahedron and associate to it a
decorated triangle. Let $\tau$ be a hyper-ideal tetrahedron
carrying the notations from construction
\ref{Sec_hyperbolic_tetrahedra}.1. Then either $ijk$ is a
truncating triangular face of $\tau$, defining a hyperbolic plane
$\Hplane$, as shown on figure \ref{Fig4}a, or it lies on a
horosphere $\Hor$ centered at an ideal vertex of $\tau$. Either
way, we denote $\Hplane$ and $\Hor$ with the common letter
$\Femb$. If $\Femb = \Hplane$ then $ijk$ is a hyperbolic triangle.
If $\Femb = \Hor$ then the hyperbolic metric restricted on $\Hor$
makes $\Hor$ isometric to the Euclidean plane and the triangle
$ijk$ is then a Euclidean triangle. Furthermore, each triangular
truncating face $A^{wu}_uA^{uv}_uB_u$ determines a hyperbolic
plane $\tilde{c}_u$ whose ideal points form a circle $c^{\infty}_u
\, \subset \,
\partial\hyperbolicspace$ (see
figure \ref{Fig2}b). Similarly, the face
$A^{ki}_iA^{ij}_iA^{ij}_jA^{jk}_jA^{jk}_kA^{ki}_i$ determines a
hyperbolic plane $\tilde{c}_{\Delta}$, whose ideal points form a
circle $c^{\infty}_{\Delta} \, \subset \,
\partial\hyperbolicspace$. 
Then the preimages $c_i = F^{-1}_{proj}(c^{\infty}_i), \, c_j =
F^{-1}_{proj}(c^{\infty}_j)$ and $c_k =
F^{-1}_{proj}(c^{\infty}_k)$ are vertex circles for $\triangle
ijk$. Furthermore, $c_{\Delta} =
F^{-1}_{proj}(c^{\infty}_{\Delta})$ is the corresponding face
circle, orthogonal to the three vertex circles, since the face
plane $\tilde{c}_{\Delta}$ of $\tau$ is by definition orthogonal
to the truncating planes $\tilde{c}_i, \tilde{c}_j, \tilde{c}_k$.
Because the upper half-space model is conformal, the six principal
dihedral angles of $\tau$ equal the angles of the constructed
decorated triangle. Observe that this construction is the converse
of construction \ref{Sec_hyperbolic_tetrahedra}.1.
\newline \noindent{\bf End of construction
\ref{Sec_hyperbolic_tetrahedra}.2.}

\begin{lem} \label{Lem_two_triang_two_tetrahedra}
Let construction \ref{Sec_hyperbolic_tetrahedra}.1 produce two
hyper-ideal tetrahedra $\tau_{\Delta}$ and $\tau_{\Delta}$ in
$\hyperbolicspace$ from two isometric (or similar if applicable)
decorated triangles $\Delta$ and $\Delta'$ in $\PP$. Then
$\tau_{\Delta}$ and $\tau_{\Delta'}$ are isometric. Conversely,
let construction \ref{Sec_hyperbolic_tetrahedra}.2 produce two
decorated triangles $\Delta_{\tau}$ and $\Delta_{\tau'}$ from two
isometric hyper-ideal tetrahedra $\tau$ and $\tau'$ in
$\hyperbolicspace$. Then $\Delta_{\tau}$ and $\Delta_{\tau'}$ are
isometric (or similar if applicable).
\end{lem}
\begin{proof}
Let $\Femb$ and $\Femb'$ be two isometrically embedded copies of
$\PP$ in $\hyperbolicspace$. Both of these are ether two
hyperbolic planes or two horospheres. Without loss of generality
one can think that $\Delta$ lies in $\Femb$ and $\Delta'$ lies in
$\Femb'$. Throughout this proof $F_{proj} : \Femb \to
\partial\hyperbolicspace$ and $F'_{proj} : \Femb' \to
\partial\hyperbolicspace$ are the corresponding maps defined in
construction \ref{Sec_hyperbolic_tetrahedra}.1 and also used in
construction \ref{Sec_hyperbolic_tetrahedra}.2.

First, assume $\Femb$ and $\Femb'$ are horospheres. In this case,
the two triangles are similar. Therefore, there exists an isometry
$g$ of $\hyperbolicspace$ such that $g(\Femb)$ and $\Femb'$ have a
common point at infinity and
$g(F_{proj}(\Delta))=F'_{proj}(\Delta')$. Here, one invokes the
property that hyperbolic isometries of $\hyperbolicspace$
naturally extend to conformal automorphisms of
$\partial\hyperbolicspace$ (see \cite{ThuBook, BenPetr}). Observe
that by construction \ref{Sec_hyperbolic_tetrahedra}.1, $\,
g(F_{proj}(\Delta))$ and $F'_{proj}(\Delta')$ give rise to the two
hyper-ideal tetrahedra $\tau_{g(\Delta)}$ and $\tau_{\Delta'}$, so
$\tau_{g(\Delta)}=\tau_{\Delta'}$. Since construction
\ref{Sec_hyperbolic_tetrahedra}.1 is entirely defined in terms of
the geometry of $\hyperbolicspace$, it commutes with
$\hyperbolicspace-$isometries, i.e. $\tau_{g(\Delta)} =
g(\tau_{\Delta})$. Hence $g(\tau_{\Delta}) =\tau_{\Delta'}$, i.e.
$\tau_{\Delta}$ and $\tau_{\Delta'}$ are isometric.

Next, let us assume that $\Femb$ and $\Femb'$ are hyperbolic
planes. Then $\Delta$ and $\Delta'$ are isometric. Therefore,
there exists and isometry $g \in Isom(\hyperbolicspace)$ such that
$g(\Femb) = \Femb'$, $\, g(\Delta)=\Delta'$ and $g$ maps the side
of $\Femb$ on which $F_{proj}$ is defined to the side of $\Femb'$
on which $F_{proj}'$ is defined. Again, since construction
\ref{Sec_hyperbolic_tetrahedra}.1 is purely geometric, $g \circ
F_{proj} = F'_{proj} \circ g$. Therefore, $g(F_{proj}(\Delta)) =
F'_{proj}(g(\Delta)) = F'_{proj}(\Delta')$. By construction
\ref{Sec_hyperbolic_tetrahedra}.1 $\, \tau_{\Delta'} =
\tau_{g(\Delta)}=g(\tau_{\Delta})$, i.e. $\tau_{\Delta}$ and
$\tau_{\Delta'}$ are isometric.

Conversely, let $\tau$ and $\tau'$ be two hyper-ideal tetrahedra
in $\hyperbolicspace$ and $g \in Isom(\hyperbolicspace)$ be a
hyperbolic isometry, such that $g(\tau)=\tau'$.

If $\PP=\hyperbolicplane$ then $\tau$ and $\tau'$ have a pair of
triangular truncating faces $\Delta$ and $\Delta'$ respectively,
where $\Delta'=g(\Delta)$. The faces $\Delta$ and $\Delta'$ define
the hyperbolic planes $\Femb \supset \Delta$ and $\Femb' \supset
\Delta'$ respectively, thus $\Femb'=g(\Femb)$. Construction
\ref{Sec_hyperbolic_tetrahedra}.2 makes $\Delta$ into a decorated
triangle $\Delta_{\tau}$  and $\Delta'$ into a decorated triangle
$\Delta_{\tau'}$. Since construction
\ref{Sec_hyperbolic_tetrahedra}.2 is purely geometric in nature,
it commutes with $g$. In other words,
$\Delta_{\tau'}=\Delta_{g(\tau)}=g(\Delta_{\tau})$, i.e. the
decorated triangles $\Delta_{\tau'}$ and $\Delta_{\tau}$ are
isometric.

Finally, let $\PP=\Euclideanplane$. Then $\tau$ and $\tau'$ have a
corresponding pair of ideal vertices $\infty$ of $\tau$ and
$\infty' \in
\partial\hyperbolicspace$ of $\tau'$ such that $g(\infty)=\infty'$.
There are two horospheres $\Femb$ and $\Femb'$ of $\tau$
respectively tangent to $\partial\hyperbolicspace$ at these two
points. Construction \ref{Sec_hyperbolic_tetrahedra}.2 produces
two decorated triangles $\Delta_{\tau} \subset \Femb$ and
$\Delta_{\tau'} \subset \Femb'$. Observe, that in this case
$g(\Delta_{\tau})$ might not be equal to $\Delta_{\tau'}$ because
$g(\Femb)$ might not be equal to $\Femb'$. Instead, in general we
have two horospheres $\Femb'$ and $\Femb'' = g(\Femb)$ tangent to
infinity at the same ideal vertex $\infty'$. As pointed out
already, construction \ref{Sec_hyperbolic_tetrahedra}.2 commutes
with $g$, so
$g(\Delta_{\tau})=\Delta_{g(\tau)}''=\Delta_{\tau'}''$, where the
decorated triangle $\Delta_{\tau'}''$ is the result of
construction \ref{Sec_hyperbolic_tetrahedra}.2  on the horosphere
$\Femb''$ with respect to $\tau'$. Now, shifting the horosphere
$\Femb''$ to the horosphere $\Femb'$ slides the decorated triangle
$\Delta_{\tau'}''$ onto $\Femb'$ so that it matches the decorated
triangle $\Delta_{\tau'}$. As the restriction of $g$ onto $\Femb$
is an Euclidean isometry between $\Femb$ and $\Femb''$, and the
shifting of $\Femb''$ onto $\Femb'$ is scaling, the two decorated
triangles $\Delta_{\tau'}$ and $\Delta_{\tau}$ are similar.
\end{proof}

\begin{prop} \label{Prop_corr_triangles_polyhedra_angles}
Each hyper-ideal tetrahedron in $\hyperbolicspace$ is defined
uniquely up to isometry by its six principal dihedral angles,
belonging to the set $\ADelta$.
\end{prop}
\begin{proof}
The principal dihedral angles of a hyper-ideal tetrahedron are
also the angles of its corresponding decorated triangle obtained
by construction \ref{Sec_hyperbolic_tetrahedra}.2, so they belong
to $\ADelta$. Conversely, assume we are given a vector of six
numbers $(\alpha^{\Delta},\beta^{\Delta})$ from $\ADelta$.
Proposition \ref{Prop_angles_to_dec_triangle} allows us to
construct a decorated triangle with
$(\alpha^{\Delta},\beta^{\Delta})$ as angles. Construction
\ref{Sec_hyperbolic_tetrahedra}.1 allows us to extend the
decorated triangle to a hyper-ideal tetrahedron with
$(\alpha^{\Delta},\beta^{\Delta})$ as principal dihedral angles, a
fact established in lemma
\ref{Lem_tetrahedron_angles_edge_lengths}. In order to prove
uniqueness, let $(\alpha^{\Delta},\beta^{\Delta})$ give rise to
two hyper-ideal tetrahedra. Then by construction
\ref{Sec_hyperbolic_tetrahedra}.2 these two tetrahedra correspond
to two decorated triangles with equal angles. By proposition
\ref{Prop_angles_to_dec_triangle}, the two decorated triangles are
isometric (or similar). Therefore, by lemma
\ref{Lem_two_triang_two_tetrahedra}, and the fact that
constructions \ref{Sec_hyperbolic_tetrahedra}.1 and
\ref{Sec_hyperbolic_tetrahedra}.2 are converse to each other, the
two tetrahedra are isometric.\end{proof}

\noindent Consequently, we can geometrize a combinatorial triangle
in two ways. We can either turn it into a decorated triangle or we
can turn it into a hyper-ideal tetrahedron. The association with
hyper-ideal tetrahedra will provide us with the right quantitative
description of the space of generalized hyper-ideal circle
patterns. Given a tetrahedron $\tau_{\Delta}$ arising from a
decorated triangle $\Delta=ijk$ via construction
\ref{Sec_hyperbolic_tetrahedra}.1, we can extract the six
principal edge-lengths $(a,b)_{\Delta}= \big(a_{ij}, a_{jk},
a_{ki}, b_k, b_i, b_j\big) \in \reals^{E_{\Delta}}\times
\reals^{V_{\Delta}}$. Recall that these are the numbers determined
in lemma \ref{Lem_tetrahedron_angles_edge_lengths} (see also
figure \ref{Fig4}a).
Notice that some of them could be zero. 
Given a combinatorial triangle $\Delta=ijk$, define the space of
tetrahedral edge-lengths $\TED$ to be the set of all six numbers
$\big(a_{ij}, a_{jk}, a_{ki}, b_k, b_i, b_j\big)$ which are the
principal edge-lengths of hyper-ideal tetrahedrons with fixed
combinatorics provided by $\Delta$.


Our goal is to describe generalized hyper-ideal circle patterns in
terms of principal edge-lengths of the corresponding tetrahedra.
For the topological triangulation $\Triang=(V, E_T, F_T)$ on the
surface $S$ assign to its edges $a : E_T \to \reals$ and to its
vertices $b : V \to \reals$ so that for each face $\Delta \in F_T$
the six numbers $(a, b)_{\Delta}$, with possible zeroes among
them, are principal edge-lengths of a hyper-ideal tetrahedron $
\tau_{\Delta}$. Thus, one can define the space $\TE$ as the set of
all assignments $(a, b) \in \reals^{E_{\pi} \cup E_1} \times
\reals^{V_1} \, \subset \, \reals^{E_T}\times\reals^{V}$ such that
for any $\Delta \in F_T$ the six numbers $(a, b)_{\Delta}$, some
of which could be fixed to be zero, belong to $\TED$.

\begin{lem} \label{Lem_local_link_edges_radii_to_tetrahedral_edges}
Let $\Delta = ijk$ be a decorated triangle in $\PP \cong \Femb$
and let $\tau_{\Delta}$ be its corresponding hyper-ideal
tetrahedron (see constructions \ref{Sec_hyperbolic_tetrahedra}.1
and \ref{Sec_hyperbolic_tetrahedra}.2, as well as figure
\ref{Fig4}a). Let $(l, r)_{\Delta}=\big(l_{ij}, l_{jk}, l_{ki},
r_k, r_i, r_j\big) \in \ERD$ be the the three edge-lengths and
three vertex radii of $\Delta$, and let $(a,
b)_{\Delta}=\big(a_{ij}, a_{jk}, a_{ki}, b_k, b_i, b_j\big) \in
\TED$ be the six principal edge-lengths of $\tau_{\Delta}$. Then
for $v \in V_{\Delta}$ and $uv \in E_{\Delta}$ the following
formulas hold:

\smallskip
\noindent {\bf 1.} $\PP=\Euclideanplane$.
\begin{align}
r_v &= e^{- b_v} \,\, \text{ if } \,\, v \in V_1 \,\,\,\, \text{
and } \,\,\,\, r_v=b_v=0 \,\, \text{ if } \,\, v
\in V_0 \label{Eqn_Eucl_r_b}\\
l_{uv} &= \sqrt{e^{-2 b_u} + e^{-2 b_v} + 2
e^{-b_u-b_v}\cosh{a_{uv}}} \,\,\,\, \text{ if } \,\, uv \in E_1
\cup E_{\pi} \, \text{ and } \, u,v \in V_1 \label{Eqn_Eucl_l_a_b_b}\\
l_{uv} &= \sqrt{e^{-2 b_v} + e^{a_{uv} - b_v}} \,\,\,\, \text{ if
} \,\, uv \in E_1 \cup E_{\pi} \,\,\text{ and } \,\, u \in
V_0, \,\, v \in V_1\label{Eqn_Eucl_l_a_b}\\
l_{uv} &= e^{a_{uv}/2} \,\, \text{ if } \,\,\,\, uv \in E_1 \cup
E_{\pi} \,\, \text{ and } \,\, u,v \in
V_0 \label{Eqn_Eucl_l_a}\\
l_{uv} &= e^{ - b_u} + e^{ - b_v} \,\,\,\, \text{ if } \,\, uv \in
E_0 \label{Eqn_Eucl_l_b_Eo}
\end{align}

\smallskip
\noindent {\bf 2.} $\PP = \hyperbolicplane$.
\begin{align}
r_v &= \sinh^{-1}{\Big(\frac{1}{\sinh{b_v}}\Big)} \,\, \text{ if }
\,\, v \in V_1 \,\,\,\, \text{ and } \,\,\,\, r_v=b_v=0 \,\,
\text{ if } \,\, v
\in V_0 \label{Eqn_Hyp_r_b}\\
l_{uv} &= \cosh^{-1}{\Big(\frac{\cosh{a_{uv}} + \cosh{b_u}
\cosh{b_v}}{\sinh{b_u}\sinh{b_v}}\Big)} \,\,\,\, \text{ if } \,\,
uv \in E_1
\cup E_{\pi} \, \text{ and } \, u,v \in V_1  \label{Eqn_Hyp_l_a_b_b}\\
l_{uv} &= \cosh^{-1}{\Big(\frac{e^{a_{uv}} +
\cosh{b_v}}{\sinh{b_v}}\Big)}\,\,\,\, \text{ if } \,\, uv \in E_1
\cup E_{\pi} \,\,\text{ and } \,\, u \in
V_0, \,\, v \in V_1  \label{Eqn_Hyp_l_a_b}\\
l_{uv} &= 2 \sinh^{-1}{\big(e^{a_{uv}/2}\big)}
\,\,\,\, \text{ if } \,\, uv \in E_1 \cup E_{\pi} \,\, \text{ and
} \,\, u,v \in
V_0  \label{Eqn_Hyp_l_a}\\
l_{uv} &= \sinh^{-1}{\Big(\frac{1}{\sinh{b_u}}\Big)} +
\sinh^{-1}{\Big(\frac{1}{\sinh{b_v}}\Big)} \,\,\,\, \text{ if }
\,\, uv \in E_0. \label{Eqn_Hyp_l_b_Eo}
\end{align}
\end{lem}

\begin{proof}
Constructions \ref{Sec_hyperbolic_tetrahedra}.1 and
\ref{Sec_hyperbolic_tetrahedra}.2 
reveal that both six-tuples $(l,r)_{\Delta}$ and $(a,b)_{\Delta}$
are naturally assigned to the same polyhedron $\tau_{\Delta}$
(figure \ref{Fig4}a). Then, for any $uv \in E_{\Delta}$, the face
that contains the points $uB_uA^{uv}_uA^{uv}_vB_vv$ can be treated
as a geodesic polygon in $\hyperbolicplane$. Observe that some of
the points that determine the face may actually merge together
into ideal points. First, one can derive formulas
(\ref{Eqn_Eucl_r_b}) by a direct integration in the upper
half-plane. Equality (\ref{Eqn_Hyp_r_b}) comes from a standard
formula from hyperbolic trigonometry (see \cite{Bus}). Then, in
order to derive the rest of the expressions from the list
(\ref{Eqn_Eucl_r_b}) - (\ref{Eqn_Hyp_l_b_Eo}), one could look at
the face $uB_uA^{uv}_uA^{uv}_vB_vv$ and simply apply various
formulas from hyperbolic trigonometry to express the length
$l_{uv}=l_{\hyperbolicplane}(uv)$ as a function of the given
lengths $a_{uv}, b_u$ and $b_v$. For instance,
(\ref{Eqn_Hyp_l_a_b_b}) follows from the hyperbolic law of cosines
for a right-angled hexagon, while (\ref{Eqn_Hyp_l_a_b}) and
(\ref{Eqn_Hyp_l_a}) could be respectively interpreted as the
reduction of that cosine law to right-angled pentagons with one
ideal vertex and to right-angled quadrilaterals with two ideal
vertices. In fact most of the equalities (\ref{Eqn_Eucl_r_b}) -
(\ref{Eqn_Hyp_l_b_Eo}) could be found in the texts \cite{Bus,
BenPetr, ThuBook, S}. Those that might not be easy to come across
in the literature could be derived by combining the hyperbolic
geometry of the upper half-plane model with the underlying
Euclidean geometry.  
\end{proof}

Recall that in the case of $\PP=\Euclideanplane$, decorated
triangles are considered up to Euclidean motions and scaling. In
the polyhedral interpretation, the corresponding hyper-ideal
tetrahedra come decorated with a choice of a horosphere $\Hor$.
The rescaling of the Euclidean decorated triangle corresponds to a
shift of $\Hor$ closer to or further from its ideal vertex. If
there are other ideal vertices, their horoshperes $Hor_u$ are
adjusted accordingly. This rescaling manifests itself as a free
$\reals-$action on both spaces $\TE$ and $\TED$. For $t \in
\reals,$ the actions $(a,b) \mapsto ACT_t(a,b)$ and
$(a,b)_{\Delta} \mapsto ACT^{\Delta}_t(a,b)$ are expressed with
the formulas
\begin{align*}
&b_k \mapsto b_k - t \,\,\,\,  \text{ if } \, k \in V_1\\
&a_{ij} \mapsto a_{ij} + t \,\,\,\, \text{ if } i \in V_0, \, j
\in
V_1\\
&a_{ij} \mapsto a_{ij} + 2t \,\,\,\, \text{ if } i,\, j \in V_0.
\end{align*}
To factor out the $\reals-$action define the cross-sections
\begin{align*}
\TE_0 &= \Big\{(a,b) \in \TE \,\, \big{|} \,\, \sum_{i \text{ or }
j \in V_0} a_{ij} - \sum_{k \in V_1} b_k = 0 \,\, \Big\}\\
\TE_{0,\Delta} &= \Big\{(a,b)_{\Delta} \in \TED \,\, \big{|} \,\,
\sum_{i \text{ or } j \in V^0_{\Delta}} a_{ij} - \sum_{k \in
V_{\Delta}^1} b_k = 0 \,\, \Big\}.
\end{align*}
\noindent In the case $\PP = \hyperbolicplane$, we simply take
$\TE_0=\TE$ and $\TE_{0,\Delta} = \TE_{\Delta},$ i.e. we assume
that the $\reals-$action is trivial. Let $PR : \TE \to \TE_0$ and
$PR_{\Delta} : \TED \to \TE_{0,\Delta}$ be the linear projection
maps along the orbits of $\reals$. When $\PP=\hyperbolicplane$,
the maps $PR$ and $PR_{\Delta}$ are actually the identity, because
$\reals$ acts trivially. When $\PP=\Euclideanplane$, the maps $PR$
and $PR_{\Delta}$ are linear maps with one dimensional kernels.
Recall that whenever $\PP = \Euclideanplane$, $\, \reals$ also
acts on $\ER$ and $\ERD$ by $(l, r) \mapsto (e^tl, e^tr)$ and $(l,
r)_{\Delta} \mapsto (e^tl, e^tr)_{\Delta}$ respectively. In the
case $\PP=\hyperbolicplane$, we can again assume that $\reals$
acts trivially. Either way, denote these actions by $(l,r) \mapsto
E_t(l,r)$ and $(l, r)_{\Delta} \mapsto E^{\Delta}_t(l, r)$.

\begin{lem} \label{Lem_map_between_TE_and_ER} There are real analytic
diffeomorphisms $ \tilde{\Psi} \, : \, \TE \, \to \, \ER$ and
$\tilde{\Psi}_{\Delta}  \, : \, \TED \, \to \, \ERD$, defined by
the formulas in lemma
\ref{Lem_local_link_edges_radii_to_tetrahedral_edges}.
Geometrically speaking, $\tilde{\Psi}_{\Delta}$ maps the principal
edge-lengths of a hyper-ideal tetrahedron with combinatorics
$\Delta$ to the edge-lengths and vertex radii of its corresponding
decorated triangle (see constructions
\ref{Sec_hyperbolic_tetrahedra}.1 and
\ref{Sec_hyperbolic_tetrahedra}.2).
Furthermore, 
$\tilde{\Psi} \circ ACT_t = E_t
\circ \tilde{\Psi}$ and $\tilde{\Psi}_{\Delta} \circ
ACT^{\Delta}_t = E^{\Delta}_t \circ \Psi_{\Delta}$. Consequently,
there is a pair of real analytic diffeomorphisms $ \Psi \, : \,
\TE_0 \, \to \, \ER_1$ and $\Psi_{\Delta} \, : \, \TE_{0,\Delta}
\, \to \, \ER_{1,\Delta}$. Finally, $\TE = \tilde{\Psi}^{-1}(\ER)$
and $\TE_0 = \Psi^{-1}(\ER_1)$ as well as $\TE_{\Delta} =
\tilde{\Psi}^{-1}_{\Delta}(\ER_{\Delta})$ and $\TE_{0,\Delta} =
\Psi^{-1}_{\Delta}(\ER_{1,\Delta})$.
\end{lem}
\begin{proof}
The geometric interpretation follows from lemma
\ref{Lem_local_link_edges_radii_to_tetrahedral_edges}. The
formulas from lemma
\ref{Lem_local_link_edges_radii_to_tetrahedral_edges} are real
analytic expressions, so the maps $\tilde{\Psi}$ and
$\tilde{\Psi}_{\Delta}$ are real analytic. It is straightforward
to invert formulas (\ref{Eqn_Eucl_r_b}) and (\ref{Eqn_Hyp_r_b})
and express $b_v$ as a function of $r_v$. After that, one can
easily invert the rest of the formulas and express $a_{uv}$ in
terms of $r_u, r_v$ and $l_{uv}$ explicitly. Thus, one obtains
well defined real analytic expressions, which define the inverse
maps $\tilde{\Psi}^{-1}$ and $\tilde{\Psi}_{\Delta}^{-1}$.
Furthermore, having in mind how $\reals$ acts on the spaces $\TE$
and $\ER$ (resp. $\TED$ and $\ERD$), it is straight forward to
check the $\reals-$equivariance of the maps $\tilde{\Psi}$ and
$\tilde{\Psi}_{\Delta}$. Finally, one can define $\Psi$ and
$\Psi_{\Delta}$ so that $\Psi^{-1}= PR \circ
\tilde{\Psi}^{-1}|_{\ER_1} \, : \, \ER_1 \, \to \, \TE_0$ and
$\Psi^{-1}_{\Delta}= PR_{\Delta} \circ
\tilde{\Psi}^{-1}_{\Delta}|_{\ER_1} \, : \, \ER_{1,\Delta} \, \to
\, \TE_{0,\Delta}$.
\end{proof}

\begin{prop} \label{Prop_corr_HTet_and_TE}
A hyper-ideal tetrahedron in $\hyperbolicspace$ is defined
uniquely up to isometry by its principal edge-lengths, belonging
to the set $\TE_{0,\Delta}$.
\end{prop}

\begin{proof}
By definition, the principal edge-lengths of a hyper-ideal
tetrahedron belong to $\TED$. If given edge-lengths
$(a,b)_{\Delta} \in \TE_{0,\Delta}$, then take $(l,r)_{\Delta} =
\Psi_{\Delta}(a,b) \in \ER_{1,\Delta}$. Choose an arbitrary $\Femb
\subset \hyperbolicspace$ and by proposition
\ref{Prop_corr_triangles_ER} construct in $\Femb$ a decorated
triangle with edge-lengths and vertex radii $(l,r)_{\Delta}$. Then
apply construction \ref{Sec_hyperbolic_tetrahedra}.1 to obtain a
hyper-ideal tetrahedron. By lemma \ref{Lem_map_between_TE_and_ER}
the tetrahedron has principal edge-lengths $(a,b)_{\Delta} \in
\TE_{0,\Delta}$. In order to prove uniqueness, let $(a,b)_{\Delta}
\in \TE_{0,\Delta}$ give rise to two hyper-ideal tetrahedra. Then
by construction \ref{Sec_hyperbolic_tetrahedra}.2 these two
tetrahedra correspond to two decorated triangles with equal
edge-lengths and vertex radii. By proposition
\ref{Prop_corr_triangles_ER}, the two decorated triangles are
isometric. Therefore, by lemma
\ref{Lem_two_triang_two_tetrahedra}, and the fact that
constructions \ref{Sec_hyperbolic_tetrahedra}.1 and
\ref{Sec_hyperbolic_tetrahedra}.2 are converse to each other, the
two tetrahedra are isometric. \end{proof}


\begin{lem} \label{Lem_diffeo_TEdelta_angles}
For a given combinatorial triangle $\Delta$, there exists an
$\reals-$invariant real analytic map $\tilde{\Phi}_{\Delta} \, :
\, \TE_{\Delta} \, \to \, \ADelta$ such that its restriction
$\tilde{\Phi}_{\Delta}|_{\TE_{0,\Delta}} = \Phi_{\Delta} \, : \,
\TE_{0, \Delta} \, \to \, \ADelta$ is a real analytic
diffeomorphism. The maps $\tilde{\Phi}_{\Delta}$ and
$\Phi_{\Delta}$ associate to the tetrahedral edge-lengths
$(a,b)_{\Delta} \in \TE_{0,\Delta}$ of a hyper-ideal tetrahedron
with combinatorics $\Delta$ its corresponding dihedral angles
$(\alpha^{\Delta}, \beta^{\Delta}) \in \ADelta$. In particular,
the angles $\alpha^{\Delta}_{ij} = \alpha_{ij}^{\Delta}(a,b)$ for
$ij \in E_{\Delta}$ and $\beta^{\Delta}_{k} =
\beta_{k}^{\Delta}(a,b)$ for $k \in V_{\Delta}$ are
$\reals-$invariant and depend analytically on the edge-length
variables $(a,b)_{\Delta} \in \TED$ and can be written explicitly
in terms of compositions of hyperbolic trigonometric formulas.
\end{lem}

\begin{proof}
Let $(a,b)_{\Delta} \in \TE_{0, \Delta}$. By applying proposition
\ref{Prop_corr_HTet_and_TE}, take a hyper-ideal tetrahedron
$\tau_{\Delta}$ whose principal edge-lengths are $(a,b)_{\Delta}$
and record its six angles $(\alpha^{\Delta}, \beta^{\Delta}) \in
\ADelta$. Denote this map by $\Phi_{\Delta}(a,b) =
(\alpha^{\Delta},\beta^{\Delta})$. Proposition
\ref{Prop_corr_HTet_and_TE} guarantees the correctness of the
map's definition, because if we choose another tetrahedron
$\tau'_{\Delta}$ with the same principal edge-lengths, its
principal dihedral angles will also be $(\alpha^{\Delta},
\beta^{\Delta}) \in \ADelta$ since $\tau_{\Delta}$ and
$\tau'_{\Delta}$ have to be isometric. Furthermore, proposition
\ref{Prop_corr_triangles_polyhedra_angles} implies that
$\Phi_{\Delta} \, : \, \TE_{0,\Delta} \, \to \, \ADelta$ has a
well defined inverse $\Phi_{\Delta}^{-1} \, : \, \ADelta \, \to \,
\TE_{0,\Delta}$. The interpretation of $(\alpha^{\Delta},
\beta^{\Delta})$ as the principal dihedral angles of a hyper-ideal
tetrahedron and $(a,b)_{\Delta}$ as the corresponding principal
edge-lengths of the same tetrahedron follows immediately form the
construction of the map $\Phi_{\Delta}$.
The map $\tilde{\Phi}_{\Delta}$ is defined as
$\tilde{\Phi}_{\Delta} = \Phi_{\Delta} \circ PR_{\Delta}.$ Recall
that in the hyperbolic case, $RP_{\Delta}$ is the identity and so
$\tilde{\Phi}_{\Delta} = \Phi_{\Delta}$.

Next, one needs so show that both $\Phi_{\Delta}$ and
$\Phi^{-1}_{\Delta}$ depend analytically on their respective
variables. This simply means that it is enough to make sure that
the angles can be expressed as real analytic function of the
edge-lengths and vice versa. This follows directly from hyperbolic
trigonometry and the fact that the dependence in each direction
can be written down explicitly. To illustrate this fact, we work
out only one case in detail. The rest are analogous. In addition
to the exposition that follows, we encourage the reader to look at
figure \ref{Fig4}a and follow the notations there.

Before we continue, we would like to emphasize that there are
several alternative ways of writing down the analytic dependence
of the dihedral angles on the edge-lengths and vice versa in terms
of compositions of different hyperbolic trigonometric formulas
(such as the hyperbolic law of sines and the various cases of the
hyperbolic laws of cosines). However, we are presenting just one
of them, while all the rest yield the same results, simply written
down as combinations of different formulas.

Let $\tau_{\Delta}$ be a hyper-ideal tetrahedron with principal
edge-lengths $(a,b)_{\Delta}$ and dihedral angles
$(\alpha^{\Delta},\beta^{\Delta})$. Moreover, let $\tau_{\Delta}$
be labelled as in construction \ref{Sec_hyperbolic_tetrahedra}.1
and figure \ref{Fig4}a. Furthermore, assume that $\tau_{\Delta}$
has one ideal vertex $A^{ij}_j\equiv A^{jk}_j \equiv B_j \in
\partial\hyperbolicspace$ and three triangular truncating faces
without ideal vertices (figure \ref{Fig4}a).
Therefore $b_j=0$ and $\alpha^{\Delta}_{ij} + \beta^{\Delta}_j +
\alpha^{\Delta}_{jk} = \pi$. Let $\zeta_{14}(a_{uv},b_v)$ and
$\zeta_{13}(a_{uv},b_u,b_v)$ be the respective right hand sides of
equations (14) and (13) from lemma
\ref{Lem_local_link_edges_radii_to_tetrahedral_edges}. Then,
according to lemma
\ref{Lem_local_link_edges_radii_to_tetrahedral_edges}, the edge
lengths of the triangular face $ijk$ are
$$l_{ij}=\zeta_{14}(a_{ij},b_i), \,\,\,\,
l_{jk}=\zeta_{14}(a_{jk},b_k), \,\,\,\,
l_{ki}=\zeta_{13}(a_{ki},b_k,b_i).$$ For an arbitrary geodesic
triangle with edge-lengths $x_1, x_2, x_3$ and an angle $\gamma_3$
opposite to $x_3$, the hyperbolic law of cosines in terms of edges
(see \cite{Bus}) gives the formula
$$\gamma_3 = \zeta(x_1,x_2,x_3) = \arccos{\Big(\frac{\cosh{x_1}\cosh{x_2} - \cosh{x_3}}
{\sinh{x_1}\sinh{x_2}}\Big)}.$$ Applying this equality to the
triangle $\triangle ijk$, one obtains the expressions
$$\beta^{\Delta}_{i}=\zeta(l_{ki}, l_{ij}, l_{jk}), \,\,\,\,
\beta^{\Delta}_{j}=\zeta(l_{ij}, l_{jk}, l_{ki}), \,\,\,\,
\beta^{\Delta}_{k}=\zeta(l_{jk}, l_{ki}, l_{ij}).$$ Thus, the
angles $\beta^{\Delta}_i, \beta^{\Delta}_j$ and $\beta^{\Delta}_k$
are analytic functions with respect to the variables
$(a,b)_{\Delta}$.

For any permutation $u \neq v \neq w \in V_{\Delta}$ such that
$v\neq j$, define the lengths $\sigma^{vw}_v =
l_{\hyperbolicspace}(A^{vw}_vB_v)$ and $\sigma_v =
l_{\hyperbolicspace}(A^{uv}_vA^{vw}_v).$ Since the face
$iB_iA^{ki}_iA^{ki}_kB_kk$ is a right-angled hexagons, one finds
$\sigma^{ki}_i = \zeta_{13}(b_k,a_{ki},b_i)$. Furthermore, the
face $iB_iA^{ij}_iB_jj$ is a right-angled pentagon with ideal
vertex $B_j\equiv A^{ij}_j\equiv A^{jk}_j$. Then $\sigma^{ij}_i =
\zeta_{14}(-a_{ij},b_i)$. Analogously, by considering the
right-angled pentagon $A^{ki}_iA^{ij}_iB_jA^{jk}_kA^{ki}_k$ with
ideal vertex $B_j$, one comes to the equality $\sigma_i =
\zeta_{14}(a_{jk} - a_{ij}, \, a_{ki})$. By applying the
hyperbolic law of sines (see \cite{Bus}) to the triangular face
$A^{ki}_iA^{ij}_iB_i$ one comes to the expressions
$$\alpha^{\Delta}_{ij} = \arcsin{\left(
\frac{\sinh{\sigma^{ki}_i}}{\sinh{\sigma_i}}\sin{\beta^{\Delta}_i}\right)}
\,\,\,\, \text{ and } \,\,\,\, \alpha^{\Delta}_{ki} =
\arcsin{\left(
\frac{\sinh{\sigma^{ij}_i}}{\sinh{\sigma_i}}\sin{\beta^{\Delta}_i}\right)}$$
which are real analytic functions with respect to
$(a,b)_{\Delta}$. Subsequently, $\alpha^{\Delta}_{jk} = \pi -
\alpha^{\Delta}_{ij} - \beta^{\Delta}_j$ is also a real analytic
function with respect to $(a,b)_{\Delta}.$

Next, we express the edge-lengths in terms of the dihedral angles.
For that we need the law of cosines for a hyperbolic triangle in
terms of its angles \cite{Bus}
$$x_3 = \nu(\gamma_1,\gamma_2,\gamma_3)
= \cosh^{-1}{\Big(\frac{\cos{\gamma_1}\cos{\gamma_2} +
\cos{\gamma_3}} {\sin{\gamma_1}\sin{\gamma_2}}\Big)},$$ where
$\gamma_1, \gamma_2$ and $\gamma_3$ are the angles of the triangle
and $x_3$ is the length of the edge opposite to $\gamma_3$.
Applying this formula to the triangle $\triangle ijk$
$$l_{ij} = \nu(\beta^{\Delta}_i, \beta^{\Delta}_j, \beta^{\Delta}_k),
\,\,\,\, l_{jk} = \nu(\beta^{\Delta}_j, \beta^{\Delta}_k,
\beta^{\Delta}_i), \,\,\,\, l_{ki} = \nu(\beta^{\Delta}_k,
\beta^{\Delta}_i, \beta^{\Delta}_j).$$ By applying the same
argument to the triangles $\triangle A^{ki}_iA^{ij}_iB_i$ and
$\triangle A^{jk}_kA^{ki}_kB_k$ one obtains the respective
formulas \begin{align*}\sigma^{ij}_i &= \nu(\beta^{\Delta}_i,
\alpha^{\Delta}_{ij}, \alpha^{\Delta}_{ki}), \,\,\,\,\,
\sigma^{ki}_i = \nu(\beta^{\Delta}_i, \alpha^{\Delta}_{ki},
\alpha^{\Delta}_{ij}) \,\,\,\, \text{ and }\\
\sigma^{jk}_k &= \nu(\beta^{\Delta}_k, \alpha^{\Delta}_{jk},
\alpha^{\Delta}_{ki}), \,\,\,\,\, \sigma^{ki}_k =
\nu(\beta^{\Delta}_k, \alpha^{\Delta}_{ki},
\alpha^{\Delta}_{jk})\end{align*} Since $iB_i$ and
$A^{ki}_iA^{ki}_k$ are edges of the right-angled hexagon
$iB_iA^{ki}_iA^{ki}_kB_kk$, the equalities $b_i =
\zeta_{13}(\sigma^{ki}_k, l_{ki},\sigma^{ki}_i)$ and $a_{ki} =
\zeta_{13}(l_{ki},\sigma^{ki}_i,\sigma^{ki}_k) $ hold. Same
argument can be used to find $b_k$. However, we present a
different formula, which could be useful in the case of
hyper-ideal tetrahedra of other combinatorial types. For the
right-angled pentagon $kB_kA^{jk}_kB_jj$ with one ideal vertex
$B_j$ the following law of cosines applies \cite{Bus}:
$$b_k = \cosh^{-1}{\left(\frac{\cosh{l_{jk}}\cosh{\sigma^{jk}_k} + 1}
{\sinh{l_{jk}}\sinh{\sigma^{jk}_k}}\right)}.$$ Thus, the
edge-lengths $b_i$ and $b_k$ are real analytic functions with
respect to the dihedral angles $(\alpha^{\Delta},\beta^{\Delta}).$
So is $b_j \equiv 0$. Considering the pentagonal faces
$iB_iA^{ij}_iB_jj$ and $kB_kA^{jk}_kB_jj$, one can invert formula
(14) in order to obtain
$$a_{ij} = \log{\big( \cosh{l_{ij}}\sinh{b_i} - \cosh{b_i} \big)}\,\, \text{ and }
\,\, a_{jk} = \log{\big( \cosh{l_{ij}}\sinh{b_i} - \cosh{b_i}
\big)}.$$
\end{proof}

\section{Volumes of hyper-ideal tetrahedra} \label{Sec_Volumes_of_tetrahedra}

Assume we are given a combinatorial triangle $\Delta=ijk$.  
For $\PP = \hyperbolicplane$ define the affine injective map
$J_{\Delta} \, : \, \reals^{E_{\Delta}^1} \times
\reals^{V_{\Delta}^1} \, \to \, \reals^{E_{\Delta}} \times
\reals^{V_{\Delta}}$ by
\begin{align} \label{Formula_map_J}
\beta^{\Delta}_v &= \tilde{\beta}^{\Delta}_v \,\,\,\,
\text{ for } \,\, v \in V_{\Delta}^1 \nonumber\\
\beta^{\Delta}_v &= \pi - \tilde{\alpha}^{\Delta}_{uv} -
\tilde{\alpha}^{\Delta}_{vw}\,\,\,\, \text{ for } \,\, v \in
V_{\Delta}^0\\
\alpha^{\Delta}_{uv} &= \tilde{\alpha}^{\Delta}_{uv} \,\,\,\,
\text{
for } uv \in E_{\Delta}^1 \nonumber\\
\alpha^{\Delta}_{uv} &= 0 \,\,\,\, \text{ for } uv \in
E_{\Delta}^0, \nonumber
\end{align}
where $u\neq v\neq w \in V_{\Delta}$. For the case $\PP =
\Euclideanplane$, let us first exclude the case of $V^0_{\Delta} =
V_{\Delta}$ . Thus, without loss of generality, we may assume that
$k\in V_{\Delta}^1$ and $V_{k,\Delta}^1 =
V_{\Delta}^1\setminus\{k\}$. The corresponding affine injective
map $J_{\Delta} \, : \, \reals^{E_{\Delta}^1} \times
\reals^{V_{k,\Delta}^1} \, \to \, \reals^{E_{\Delta}} \times
\reals^{V_{\Delta}}$ is defined again by formulas
(\ref{Formula_map_J}) with the additional condition that when
$v=k$ we replace $\beta^{\Delta}_k = \tilde{\beta}^{\Delta}_k$ by
$\beta^{\Delta}_k = \pi -
\tilde{\beta}^{\Delta}_i-\tilde{\beta}^{\Delta}_j$. Now, let us
consider the case $\PP=\Euclideanplane$ with $V_{\Delta} =
V_{\Delta}^0$. Just in this case, let $E_{\Delta}^1 = E_{\Delta}
\setminus \{ki\}$. Then the affine embedding is $J_{\Delta} \, :
\, \reals^{E_{\Delta}^1} \, \to \, \reals^{E_{\Delta}} \times
\reals^{V_{\Delta}}$ is $\alpha^{\Delta}_{ij} =
\tilde{\alpha}^{\Delta}_{ij},\,\, \alpha^{\Delta}_{jk} =
\tilde{\alpha}^{\Delta}_{jk}, \,\, \alpha^{\Delta}_{ki} = \pi -
\tilde{\alpha}^{\Delta}_{ij} - \tilde{\alpha}^{\Delta}_{jk}$ and
$\beta^{\Delta}_w=\alpha^{\Delta}_{uv}$ for all $u\neq v\neq w \in
V_{\Delta}.$

Since $\ADelta\, \subset \, image(J_{\Delta})$, let $\ADelta^1 =
J_{\Delta}^{-1}(\ADelta).$ Then $J_{\Delta} \, : \, \ADelta^1 \,
\to \, \ADelta$ is an affine isomorphism between two convex
polytopes. Therefore, whatever function is convex on one of them,
it is mapped to a convex function on the other. For the points of
$\ADelta^1$ we use the notation $(\tilde{\alpha}^{\Delta},
\tilde{\beta}^{\Delta}) \in \ADelta$ having in mind that in some
cases either all coordinates $\tilde{\alpha}^{\Delta}$ or all
$\tilde{\beta}^{\Delta}$ could be absent.

Now, let us revise the space $\TE_{0,\Delta}$. First, let
$\PP=\Euclideanplane$. If $V_{\Delta}^1 \neq \varnothing$ and $k
\in V_{\Delta}^1$ then define
$$\TE_{1,\Delta} = \big\{ \, (a,b) \in \TED \,\,
| \,\, b_k=0 \, \big\}.$$ If $V_{\Delta}^1 = \varnothing$, then
$$\TE_{1,\Delta} = \big\{ \, (a,b) \in \TED \,\,
| \,\, a_{ki}=0 \, \big\}.$$ Furthermore $\TE_{1,\Delta} =
\TE_{\Delta}$ when $\PP = \hyperbolicplane$. Then, define the
linear projection
$$PR_1 \, : \, \TED \, \to \, \TE_{1,\Delta},$$
by $PR_1(a,b)_{\Delta} = ACT^{\Delta}_{b_k}(a,b),$ i.e. $PR_1$
projects onto $\TE_{1,\Delta}$ along the $\reals-$orbits. Notice
that in the hyperbolic case $PR_1$ is identity. Furthermore,
observe that $PR_1|_{\TE_{0,\Delta}} \, : \, \TE_{0,\Delta} \, \to
\, TE_{1,\Delta}$ is a linear diffeomorphism between the two
spaces. Consequently, one can define the real analytic
diffeomorphism $\Phi_{1,\Delta} \, : \, \TE_{1,\Delta} \, \to \,
\ADelta^1$ by $\Phi_{1,\Delta}^{-1} = PR_{1} \circ
\Phi_{\Delta}^{-1} \circ J_{\Delta}$.

Let for any $(\alpha^{\Delta},\beta^{\Delta}) \in \ADelta$ the
function $\Vol_{\Delta}(\alpha^{\Delta},\beta^{\Delta})$ be the
hyperbolic volume of the unique (up to isometry) hyper-ideal
tetrahedron with principal dihedral angles
$(\alpha^{\Delta},\beta^{\Delta})$ (see proposition
\ref{Prop_corr_triangles_polyhedra_angles}). Then let the volume
$\Vol_{1,\Delta}(\tilde{\alpha}^{\Delta},\tilde{\beta}^{\Delta}) =
\Vol_{\Delta}\big(J_{\Delta}(\tilde{\alpha}^{\Delta},\tilde{\beta}^{\Delta})\big)$
be defined for all
$(\tilde{\alpha}^{\Delta},\tilde{\beta}^{\Delta}) \in \ADelta^1$.

\begin{lem} \label{Lem_volume_concavity}
The functions $\Vol_{1,\Delta} \, : \, \ADelta^{1} \, \to \,
\reals$ and $\Vol_{\Delta} \,  : \, \ADelta \, \to \, \reals $ are
real analytic and strictly concave. Furthermore,
$\Phi_{1,\Delta}^{-1}(\tilde{\alpha}^{\Delta},\tilde{\beta}^{\Delta})
= - 2 \, \big(\nabla
\Vol_{1,\Delta}\big)_{(\tilde{\alpha}^{\Delta},\tilde{\beta}^{\Delta})}$
for $(\tilde{\alpha}^{\Delta},\tilde{\beta}^{\Delta}) \in
\ADelta^1$.
\end{lem}

\begin{proof}
This lemma represents a crucial step in this article. We can
safely say that it is the ``engine" of the proof. In its own turn,
its proof relies on the famous Schl\"afli's formula
\cite{Schlafli, Milnor, Sch1, Sch2}. In this article we use its
version for hyper-ideal tetrahedra with fixed combinatorics
$\Delta=ijk$
\begin{align}
d \Vol_{\Delta} &= - \frac{1}{2}\left( \sum_{uv \in E_{\Delta}}
a_{uv} \, d \alpha^{\Delta}_{uv} \, + \, \sum_{v \in V_{\Delta}}
b_v \, d
\beta^{\Delta}_v \right) \nonumber\\
&= - \frac{1}{2}\left( \sum_{uv \in E_{\Delta}^1} a_{uv} \, d
\alpha^{\Delta}_{uv} \, + \, \sum_{v \in V_{\Delta}^1} b_v \, d
\beta^{\Delta}_v  \right) \label{Formula_Schlafli_1}
\end{align}
with $(\alpha^{\Delta},\beta^{\Delta}) \in \ADelta$. In other
words, the differential $d \Vol_{\Delta}$ is restricted to the
submanifold $\ADelta$, so in general its variables might not be
independent (see formula (\ref{Formula_map_J})). In terms of
independent variables, after pulling back $d \Vol_{\Delta}$ via
map $J_{\Delta}$ given by (\ref{Formula_map_J}), Schl\"afli's
formula becomes
\begin{align} \label{Formula_Schlafli_2}
d \Vol_{1,\Delta} = - \frac{1}{2}\left( \, \sum_{uv \in
E_{\Delta}^1} a_{uv} \, d \tilde{\alpha}^{\Delta}_{uv} \, + \,
\sum_{v \in \tilde{V}_{\Delta}^1} b_v \, d
\tilde{\beta}^{\Delta}_v \, \right),
\end{align}
where $\tilde{V}_{\Delta}^1 = V_{\Delta}^1$ when $\PP =
\hyperbolicplane$ and $\tilde{V}_{\Delta}^1 = V_{k,\Delta}^1$ when
$\PP = \Euclideanplane$. Clearly, it follows from
(\ref{Formula_Schlafli_2}) that $- 2 \, \big(\nabla
\Vol_{1,\Delta}\big)_{(\tilde{\alpha}^{\Delta},\tilde{\beta}^{\Delta})}
=(a,b)_{\Delta} \in \TE_{1,\Delta}$. Since also $(a,b)_{\Delta} =
\Phi_{1,\Delta}^{-1}(\tilde{\alpha}^{\Delta},\tilde{\beta}^{\Delta}),$
it can be concluded that $- 2 \, \big(\nabla
\Vol_{1,\Delta}\big)_{(\tilde{\alpha}^{\Delta},\tilde{\beta}^{\Delta})}
=
\Phi_{1,\Delta}^{-1}(\tilde{\alpha}^{\Delta},\tilde{\beta}^{\Delta})$
for any $(\tilde{\alpha}^{\Delta},\tilde{\beta}^{\Delta}) \in
\ADelta^1$. As $\Phi_{1,\Delta}$ is real analytic, then so is
$\nabla \Vol_{1,\Delta}$, which means that $\Vol_{1,\Delta}$ is
also real analytic.

The proof of the strict concavity of $\Vol_{1,\Delta}$, and
consequently of $\Vol_{\Delta}$, can be found in \cite{Sch2}. One
can also find comments and references in \cite{Sch1}. Another
proof of the current lemma in the case $\PP=\Euclideanplane$ can
be seen in \cite{S}. It works for all hyper-ideal tetrahedra with
at least one ideal vertex. There, one can also find a fairly nice
explicit formula for the volume of such tetrahedra. This formula
however does not work for the most general case of a hyper-ideal
tetrahedron with exactly four hyper-ideal vertices. Nevertheless,
an explicit (and fairly complicated) expression does exist and can
be found in \cite{Weird}.

We do not intend to repeat here the full proof of the strict
concavity of the volume function because we do not want to
overload this anyway lengthy article. However, we would mention
the basic ideas behind the proof, linking it to some of the
constructions we have carried out up to now. According to lemma
\ref{Lem_diffeo_TEdelta_angles} the map $\Phi_{1,\Delta}$ is a
diffeomorphism and since
$\Phi_{1,\Delta}^{-1}(\tilde{\alpha}^{\Delta},\tilde{\beta}^{\Delta})=-
2 \, \big(\nabla
\Vol_{1,\Delta}\big)_{(\tilde{\alpha}^{\Delta},\tilde{\beta}^{\Delta})},$
then for the hessian
$\text{Hess}(\Vol_{1,\Delta})_{(\tilde{\alpha}^{\Delta},\tilde{\beta}^{\Delta})}
= - \frac{1}{2} \big(D
\Phi_{1,\Delta}\big)_{(\tilde{\alpha}^{\Delta},\tilde{\beta}^{\Delta})}.$
Therefore, the rank of
$\text{Hess}(\Vol_{1,\Delta})_{(\tilde{\alpha}^{\Delta},\tilde{\beta}^{\Delta})}$
is the same (maximal) for all
$(\tilde{\alpha}^{\Delta},\tilde{\beta}^{\Delta}) \in \ADelta^1$
because the derivative $D \Phi_{1,\Delta}$ of the diffeomorphism
$\Phi_{1,\Delta}$ has non-zero determinant. Consequently, if we
can show that for one point the hessian of $\Vol_{1,\Delta}$ is
negative definite, then it should be negative definite everywhere.
Indeed, if at one point the hessian is negative definite and at
another point it is not, then somewhere in between the rank should
drop, which is not the case. Now choose
$(\tilde{\alpha}^{\Delta},\tilde{\beta}^{\Delta})$ to be
computationally the most convenient dihedral angles of the most
symmetric hyper-ideal tetrahedron possible with combinatorics
$\Delta$. For instance, if the tetrahedron has four truncating
faces, one can choose all dihedral angles to be $\pi/4$. Or if it
has exactly one vertex at infinity, one can take the betas to be
$\pi/3$ and the alphas to be say $\pi/6$. Finally, by using the
explicit formulas for the principal edge-lengths in terms of the
dihedral angles (see lemma \ref{Lem_diffeo_TEdelta_angles}),
differentiate them and evaluate the derivatives at the chosen
$(\tilde{\alpha}^{\Delta},\tilde{\beta}^{\Delta}) \in \ADelta^1$
to obtain an explicit matrix for $\big(D
\Phi_{1,\Delta}\big)_{(\tilde{\alpha}^{\Delta},\tilde{\beta}^{\Delta})}$.
Finally, one can check that it is negative definite.
\end{proof}

\section{The space of generalized circle patterns revisited} \label{Sec_defining_space_circle_patterns}

In section \ref{Sec_Generalized_space_of_patterns} we started
discussing the space of generalized hyper-ideal circle patterns,
i.e. patterns that do not necessarily satisfy the local Delaunay
property. Recall that we have fixed a closed topological surfaces
$S$ with a cell complex $\cellcomplex = (V,E,F)$ on it (figure
\ref{Fig1}a). Then, $\cellcomplex$ was subdivided into a
triangulation $\Triang=(V,E_T,F_T)$ (figure \ref{Fig3}b). Let us
denote by $\widetilde{\CP}_{\Triang}$ the space of either
hyperbolic or Euclidean generalized hyper-ideal circle patterns on
$S$ with combinatorics $\Triang$, considered up to isometries (and
global scaling in the Euclidean case) which preserve the induced
by $\cellcomplex$ marking on $S$. In section
\ref{Sec_Generalized_space_of_patterns} we saw that the space
$\ER_1$ is a global chart of $\widetilde{\CP}_{\Triang}$. In lemma
\ref{Lem_map_between_TE_and_ER} we constructed a real analytic
diffeomorphism $\Psi \, : \, \TE_0 \, \to \, ER_1$, turning
$\TE_0$ into another global real analytic chart of
$\widetilde{\CP}_{\Triang}$. Thus the two diffeomorphic spaces
$\ER_1$ and $\TE_0$ are two global real analytic charts of
$\widetilde{\CP}_{\Triang}$, so we can simply identify
$\widetilde{\CP}_{\Triang}$ with both of them. Consequently, the
space of generalized hyper-ideal pattern
$\widetilde{\CP}_{\Triang}$ is a real analytic manifold.

These definitions are quite nice and somewhat natural. However,
there is one small subtlety which we are going to address now. It
is the fact that $\widetilde{\CP}_{\Triang}$ is non-empty. Observe
that $\ER_1$ is the interior of a convex polytope, but there is no
guarantee that this interior even exists. However, if one finds at
least one element that belongs to $\widetilde{\CP}_{\Triang}$ then
$\widetilde{\CP}_{\Triang}$ will be an actual manifold of
dimension $|E_1| + |E_{\pi}| + |V_1| - 1$. We start with the
following construction.

\begin{lem} \label{Lem_existence_of_decorated_polygon}
Let $f$ be a topological (combinatorial) $N-$gon with a set of
vertices $V_f=V_f^1 \sqcup V_f^0$ and a set of edges $E_f = E_f^1
\sqcup E_f^0$. Define two real numbers $\check{r}$ and
$\check{\epsilon}$ such that

\smallskip
\begin{itemize}
\item $\check{r} = \sinh^{-1}\Big(\frac{1}{10}\Big)$ and
 $\check{\epsilon} = \sinh^{-1}\Big(\frac{1}{8}\Big)$ in
 the hyperbolic case $\PP=\hyperbolicplane$ and

 \smallskip
\item $\check{r} = 1$ and
 $\check{\epsilon} = \frac{1}{4}$ in the Euclidean case $\PP=\Euclideanplane$.
\end{itemize}

\smallskip
\noindent For $k \in V_f^1$ let $r_k=\check{r}$ and for $k \in
V_f^0$ let $r_k=0$. For $ij \in E_f^1$ let $l_{ij}=2(\check{r} +
\check{\epsilon})$ and for $ij \in E_f^0$ let $l_{ij}=2\check{r}$.
Then there exists a unique up to isometry decorated $N-$gon in
$\PP$ with prescribed edge-lengths $l_{ij}$ and vertex radii
$r_k$.
\end{lem}

\begin{proof}
\emph{Case $N=3$.} It is straightforward to check that in both
geometries, and for any admissible splitting $V_f=V_f^1 \sqcup
V_f^0$ and $E_f = E_f^1 \sqcup E_f^0$, the prescribed edge-lengths
and vertex radii define a  unique up to isometry decorated
triangle. One just needs to verify that the edge-lengths and
vertex radii satisfy the conditions of the set $\ER_f$. Then
proposition \ref{Prop_corr_triangles_ER} completes the proof.

\emph{Case $N \geq 4$ with $E_f^1\neq\varnothing$ when $N=4$.} We
build the decorated polygon by gluing together appropriate
combinations of elementary pieces. To each edge $ij \in E_f$ we
associate the following geometric triangle $\Delta_{ij}=\triangle
ijO_f$ together with circles centered at its  vertices.

\smallskip
\noindent \emph{Type 1.} Let $ij \in E^1_f$ with $i,j \in V_f^1$.
Then $\Delta_{ij} =\triangle ijO_f$ is an isosceles triangle with
both edges $iO_f$ and $jO_f$ having equal lengths $x$ and $ij$
having length $2(\check{r}+\check{\epsilon}).$ Furthermore, there
are two circles of radius $\check{r}$ centered at the vertices $i$
and $j$ respectively, while a unique third circle of radius $y$ is
centered at the vertex $O_f$ and is orthogonal to the other two
circles. Its angle at vertex $O_f$ is denoted $\measuredangle
iO_fj = \omega_1$. The altitude from the vertex $O_f$ down to the
edge $ij$ splits $\Delta_{ij}$ into two identical right-angled
triangles with angle $\omega_1/2$ at $O_f$. Then, by a hyperbolic
trigonometric formula \cite{Bus}
\begin{align*}
\omega_1 &= 2
\arcsin\left(\frac{\sinh{(\check{r}+\check{\epsilon})}}
{\sinh{x}}\right) = 2 \arcsin\left(\frac{1}
{8\sinh{x}}\right)\,\,\,\,
\text{ for } \,\,\, \PP=\hyperbolicplane, 
\\
\omega_1 &= 2 \arcsin\left(\frac{\check{r}+\check{\epsilon}}
{x}\right) = 2 \arcsin\left(\frac{5} {4 x}\right) \,\,\,\, \text{
for } \,\,\, \PP=\Euclideanplane.
\end{align*}
Since the the circle centered at $O_f$ is orthogonal to the two
circles centered at $i$ and $j$, its radius $y$ satisfies the
equations \cite{Bus}
\begin{align}
\cosh{y} &= \frac{\cosh{x}}{\cosh{\check{r}}} \,\,\,\, \text{ and
} \,\,\,\, \cosh{y} = \frac{\sqrt{\sinh^2{x} -
\sinh^2{\check{r}}}}{\cosh{\check{r}}} \,\,\,\, \text{ when }
\,\,\, \PP=\hyperbolicplane \label{Eqn_coshy_hyperbolic}\\
y &= \sqrt{x^2 - \check{r}^2} =  \sqrt{x^2 - 1} \,\,\,\, \text{
when } \,\,\, \PP=\Euclideanplane \label{Eqn_coshy_Euclidean}
\end{align}

\smallskip
\noindent \emph{Type 2.} Let $ij \in E^0_f$, so $i,j \in V_f^1$.
Then $\Delta_{ij} =\triangle ijO_f$ is an isosceles triangle with
both edges $iO_f$ and $jO_f$ having equal lengths $x$ and $ij$
having length $2\check{r}.$ The two circles of radius $\check{r}$
centered at the vertices $i$ and $j$ respectively touch and are
orthogonal to the third circle of radius $y$, centered at the
vertex $O_f$. The radius $y$ is again determined by formula
(\ref{Eqn_coshy_hyperbolic}) or (\ref{Eqn_coshy_Euclidean}). The
angle at vertex $O_f$ is denoted $\measuredangle iO_fj =
\omega_2$. As before, the altitude from the vertex $O_f$ down to
the edge $ij$ splits $\Delta_{ij}$ into two identical right-angled
triangles with angle $\omega_2/2$ at $O_f$. Then
\begin{align*}
\omega_2 &= 2 \arcsin\left(\frac{\sinh{\check{r}}}
{\sinh{x}}\right) = 2 \arcsin\left(\frac{1} {10\sinh{x}}\right)
\,\,\,\, \text{ for }
\,\,\, \PP=\hyperbolicplane, 
\\
\omega_2 &= 2 \arcsin\left(\frac{\check{r}}{x}\right) = 2
\arcsin\left(\frac{1}{x}\right) \,\,\,\, \text{ for } \,\,\,
\PP=\Euclideanplane.
\end{align*}

\smallskip
\noindent \emph{Type 3.} Let $ij \in E^1_f$ and $i,j \in V_f^0$.
Then $\Delta_{ij} =\triangle ijO_f$ is an isosceles triangle with
both edges $iO_f$ and $jO_f$ having equal lengths $y$ and $ij$
having length $2\check{r}.$ The circles centered at the vertices
$i$ and $j$ respectively have radius zero. The circle centered at
$O_f$ has radius $y$ so it passes through the vertices $i$ and
$j$. The radius $y$ is again determined by formula
(\ref{Eqn_coshy_hyperbolic}) or (\ref{Eqn_coshy_Euclidean}). The
angle at vertex $O_f$ is denoted $\measuredangle iO_fj =
\omega_3$. As before, the altitude from the vertex $O_f$ down to
the edge $ij$ splits $\Delta_{ij}$ into two identical right-angled
triangles with angle $\omega_3/2$ at $O_f$. Then
\begin{align*}
\omega_3 &= 2
\arcsin{\left(\frac{\sinh{\check{r}}\cosh{\check{r}}}
{\sqrt{\sinh^2{x} - \sinh^2{\check{r}}}}\right)}\\
&= 2 \arcsin\left(\frac{\sqrt{101}} {10\sqrt{100
\sinh^2{x}-1}}\right)
 \,\,\,\, \text{ for } \,\, \PP=\hyperbolicplane\\
\omega_3 &= 2 \arcsin\left(\frac{1}{\sqrt{x^2-1}}\right) \,\,\,\,
\text{ for } \,\,\, \PP=\Euclideanplane.
\end{align*}

\smallskip
\noindent \emph{Type 4.} Let $ij \in E^1_f$ and $i \in V_f^1, \, j
\in V_f^0$. Then for the geometric triangle $\Delta_{ij}
=\triangle ijO_f$ edge $iO_f$ has length $x$, edge $jO_f$ has
length $y$ and edge $ij$ has length $2\check{r}.$ The circle
centered at vertex $i$ has radius $\check{r}$ and the one centered
at $j$ has radius zero. The circle centered at $O_f$ has radius
$y$ so it is orthogonal to circle centered at $i$ and passes
through $j$. The radius $y$ is as usual determined by formula
(\ref{Eqn_coshy_hyperbolic}) or (\ref{Eqn_coshy_Euclidean}). The
angle at vertex $O_f$ is denoted $\measuredangle iO_fj =
\omega_4$. Then using the law of cosines
\begin{align*}
\omega_4 &= \arccos{\left(\frac{\sinh^2{x} + 1 -
\cosh{2\check{r}}\cosh{\check{r}}}{\sinh^2{x}\sqrt{\sinh^2{x} -
\sinh^2{\check{r}}}}\right)}\\
&= \arccos{\left(\frac{\sinh^2{x} + 1 -
1.02\sqrt{101}}{\sinh^2{x}\sqrt{\sinh^2{x} - 0.01}}\right)}
\,\,\,\, \text{ for } \,\,\, \PP=\hyperbolicplane
\\
\omega_4 &= \arccos\Big(\frac{2 x^2 - 5}{2 x \sqrt{x^2-1}}\Big)
\,\,\,\, \text{ for } \,\,\, \PP=\Euclideanplane.
\end{align*}
Since the values $\check{r}$ and $\check{\epsilon}$ are fixed by
the conditions of the current lemma, all four types of triangles
with circles described above are defined uniquely, up to isometry,
as long as the edge-length $x$ is also fixed. To geometrize $f$ we
assign to each edge $ij\in E_f$ a triangle $\Delta_{ij}$ whose
type is uniquely determined by the combinatorics of $f$ (which
edges have touching circles and which vertices have no circles
assigned). One can glue together two neighboring triangles along
an edge $iO_f$. The circles at $i$ and at $O_f$ agree by
construction. In this way, one can assemble a decorated polygon,
here denoted by $f_x$, with edge-lengths and vertex radii
prescribed by the current lemma. Observe that as long as the
edge-length $x$ is fixed, the decorated polygon $f_x$ is uniquely
defined, up to isometry. However, in general, the decorated
polygon $f_x$ may have a cone point at $O_f$, while we need $f_x$
to be realizable in $\PP$. All $\omega_m(x)$ are real analytic
functions where defined.

Let $N_m$ be the number of triangles $\Delta_{ij}$ of type
$m=1,2,3,4$ needed in order to geometrize $f$. Clearly, $N =
\sum_{m=1}^{4} N_m$. Furthermore, let $\omega_f(x) =
\sum_{m=1}^{4} N_m \omega_m(x)$. Then $\omega_f(x)$ is the cone
angle of $f_x$ at the point $O_f$ and is a real analytic function
with respect to $x$. We are going to show that there is exactly
one value $x^*$ for which $\omega_f(x^*)=2\pi$. The latter fact
immediately implies that $f_{x^*}$ can be realized in $\PP$
uniquely up to isometry.

For the varying parameter $x$, a triangle $\Delta_{ij}$ of type
$m=1,2,3,4$ can be geometrically realized if and only if its angle
$\omega_m(x) \in (0,\pi)$. Based on this condition, one can find
an interval $(\chi_m, \infty) \subset \reals$ such that $x \in
(\chi_m,\infty)$ exactly when $\omega_m(x) \in (0,\pi)$. A
straightforward computation shows that $\frac{d \omega_m}{dx}(x) <
0$ for any $x \in (\chi_m,\infty)$, where $m=1,2,3,4$. In other
words $\omega_m(x)$ is strictly decreasing on $(\chi_m, \infty)$.
Take $\chi_0=\max{\{\chi_m \, : \, m=1,2,3,4\}}$ and consider the
common interval $(\chi_0,\infty)$ on which $\omega_m(x)$ is well
defined. In fact, each $\omega_m(x)$ is continuous on
$[\chi_0,\infty)$. In the hyperbolic case $\chi_0 =
\sinh^{-1}\Big(\frac{\sqrt{201}}{100}\Big)$, and in the Euclidean
case $\chi_0=\sqrt{2}$. Clearly $\frac{d \omega_f}{dx}(x) =
\sum_{m=1}^{4} N_m \frac{d \omega_m}{dx}(x) < 0$. Thus, we can
conclude that $\omega_f(x)$ is a strictly decreasing continuous
function on $[\chi_0,\infty)$. As $\lim_{x \to
\infty}\omega_m(x)=0$ for $m=1,2,3,4$, we immediately conclude
that $\lim_{x \to \infty}\omega_f(x)=0 < 2\pi$. Furthermore, we
would like to estimate $\omega_f(\chi_0)$. A straightforward check
shows that $\omega_2(\chi_0)=\min{\{\omega_m(\chi_0)\, : \,
m=1,2,3,4\}}$. If $N>4$ then
$$\omega_f(\chi_0)=
\sum_{m=1}^{4} N_m \omega_m(\chi_0) \geq N \, \omega_2(\chi_0)
\geq 5 \, \omega_2(\chi_0) > 2\pi.$$ If $N=4$ and $E_f^1\neq
\varnothing$, then
$$\omega_f(\chi_0)=
\sum_{m=1}^{4} N_m \omega_m(\chi_0) \geq 3 \, \omega_2(\chi_0) +
\min_{m=1,3,4} \omega_m(\chi_0)
> 2\pi.$$
Consequently, in all of these cases there exists a unique $x^* \in
(\chi_0,\infty)$ such that $\omega_f(x^*)=2\pi$.

\emph{Case N=4 and $E_f=E_f^0$.} Just like in the previous case,
we need to glue together four triangles $\Delta_{ij}$ of type 2.
In the Euclidean case if we choose $x=\sqrt{2}$, then we obtain
four identical isosceles triangles with angles $\omega_2=\pi/2$.
Hence they fit together to form a unique decorated square. The
situation in the hyperbolic case is absolutely analogous with
$x=\sinh^{-1}{\big(\frac{\sqrt{2}}{8}\big)}$. \end{proof}

\begin{lem} \label{Lem_nonempty_space_of_patterns}
For any cell complex $\cellcomplex=(V,E,F)$ on the closed surface
$S$ there exists a hyper-ideal circle pattern on $S$ with
prescribed combinatorics $\cellcomplex$. In particular, this is
true for the triangulation $\Triang=(V,E_T,F_T)$.
\end{lem}
\begin{proof}
Based on the combinatorics of $\cellcomplex$, we assign a
decorated polygon from lemma
\ref{Lem_existence_of_decorated_polygon} to each face $f \in F.$
By constriction, all decorated polygons fit together into a
hyper-ideal circle pattern because all pairs of corresponding
edges, along which we glue the polygons together, are of equal
length either $2(\check{r}+\check{\epsilon})$ or $2\check{r}$, and
all vertex circles are of radius either $\check{r}$ or zero. It is
straightforward to see that the corresponding circle pattern
satisfies the local Delaunay property.
\end{proof}
Let us apply lemma \ref{Lem_nonempty_space_of_patterns} to the
triangulation $\Triang=(V,E_T,F_T)$. For each $\Delta \in F_T$ the
corresponding decorated triangle with geometry determined by lemma
\ref{Lem_existence_of_decorated_polygon}, has angles
$(\check{\alpha}^{\Delta},\check{\beta}^{\Delta}) \in \ADelta \neq
\varnothing$. We also denote by $(\check{\theta},\check{\Theta})
\in \reals^{E_T} \times \reals^{V}$ the intersection angles
between pairs of adjacent face circles of the constructed circle
pattern together with its cone angles.

\section{A functional on the space of circle patterns} \label{Sec_functional_on_patterns}

For any triangle $\Delta \in F_T$ from the triangulation
$\Triang=(V,E_T,F_T)$ on $S$ we associate the function $\UDelta \,
: \, \TED \, \to \, \reals$ defined as
\begin{align} \label{Formula_Dual_Schlafli_local_no_one}
\UDelta(a,b) &= \sum_{ij \in E_{\Delta}} \big(\alpha^{\Delta}_{ij}
- \check{\alpha}^{\Delta}_{ij}\big) a_{ij} + \sum_{k \in
V_{\Delta}} \big(\beta^{\Delta}_{ij} -
\check{\beta}^{\Delta}_{ij}\big) b_{k} + 2
\Vol_{\Delta}(\alpha^{\Delta}, \beta^{\Delta})
\end{align}
where the first set of angles
$\alpha^{\Delta}_{ij}=\alpha^{\Delta}_{ij}(a,b)$ for $ij \in
E_{\Delta}$ and $\beta^{\Delta}_k=\beta^{\Delta}_k(a,b)$ for $k
\in V_{\Delta}$ are the $\reals-$invariant real analytic functions
from lemma \ref{Lem_diffeo_TEdelta_angles} depending on the
tetrahedral edge-length variables $(a,b)_{\Delta} \in \TED$. The
second set of angles $\check{\alpha}^{\Delta}_{ij}$ for $ij \in
E_{\Delta}$ and $\check{\beta}^{\Delta}_k$ for $k \in V_{\Delta}$
are the constant angles of a decorated triangle with combinatorics
$\Delta$ from lemma \ref{Lem_existence_of_decorated_polygon}. The
function $\UDelta$ is real analytic and one can check that it is
also $\reals-$invariant, i.e. $\UDelta \circ ACT^{\Delta}_t =
\UDelta$ for all $t \in \reals.$ Notice that formula
(\ref{Formula_Dual_Schlafli_local_no_one}) is relevant even for
terms with indices $ij \in E_{\Delta}^0$, when in this case
$a_{ij} = 0$ and $\alpha^{\Delta}_{ij} \equiv 0$, as well as $k
\in V_{\Delta}^0$, when $b_k=0$ is fixed. Hence we can also write
\begin{align} \label{Formula_Dual_Schlafli_local}
\UDelta(a,b) &= \sum_{ij \in E_{\Delta}^1}
\big(\alpha^{\Delta}_{ij} - \check{\alpha}^{\Delta}_{ij}\big)
a_{ij} + \sum_{k \in V_{\Delta}^1} \big(\beta^{\Delta}_{ij} -
\check{\beta}^{\Delta}_{ij}\big) b_{k} + 2
\Vol_{\Delta}(\alpha^{\Delta}, \beta^{\Delta})
\end{align}
In the case $V_{\Delta}^1 = \varnothing$ for $\PP=\Euclideanplane$
the angles satisfy the identities $\alpha^{\Delta}_{ij} =
\beta^{\Delta}_k$ for $i\neq j \neq k \in V_{\Delta}$.

\begin{lem} \label{Lem_derivatives_Ulocal}
Let $\Delta \in F_T$. Then
$$\frac{\partial\UDelta}{\partial a_{ij}} = \alpha^{\Delta}_{ij} -
\check{\alpha}^{\Delta}_{ij} \,\,\,\,\,\, \text{ and }
\,\,\,\,\,\, \frac{\partial\UDelta}{\partial b_{k}} =
\beta^{\Delta}_{k} - \check{\beta}^{\Delta}_{k}$$ whenever $ij \in
E_{\Delta}^1$ and $k \in V_{\Delta}^1$. Furthermore,
$\frac{\partial\UDelta}{\partial a_{ij}} \equiv 0$ for all $ij \in
E_{\Delta}^0$ and $\frac{\partial\UDelta}{\partial b_{k}} \equiv
0$ for $k\in V_{\Delta}^0$. 
\end{lem}

\begin{proof}
By inspecting formula (\ref{Formula_Dual_Schlafli_local}) one sees
that the function $\UDelta$ depends explicitly on the variables
$a_{ij}$ and $b_k$ if and only if $ij \in E_{\Delta}^1$ and $k \in
V_{\Delta}^1$. In all the other cases the partial derivatives are
identically zero. The differential of
(\ref{Formula_Dual_Schlafli_local}) is
\begin{align*} 
d\UDelta = & \, \, d\left(\sum_{ij \in E_{\Delta}^1}
(\alpha^{\Delta}_{ij}-\check{\alpha}^{\Delta}_{ij})\, a_{ij} +
\sum_{k \in V_{\Delta}^1} (\beta^{\Delta}_k -
\check{\beta}^{\Delta}_{k})\, b_k \right) + 2
dV(\alpha^{\Delta},\beta^{\Delta}) = \\
= & \sum_{ij \in E_{\Delta}^1}
(\alpha^{\Delta}_{ij}-\check{\alpha}^{\Delta}_{ij}) \, d a_{ij} +
\sum_{k \in V_{\Delta}^1}
(\beta^{\Delta}_k-\check{\beta}^{\Delta}_{k}) \, d b_k +\\
& \,\,\,\,\,\,\,\,\,\,\,\,\,\, + \sum_{ij \in E_{\Delta}^1}
a_{ij}\, d \alpha^{\Delta}_{ij} + \sum_{k \in V_{\Delta}^1} b_k \,
d \beta^{\Delta}_k + 2 dV
\end{align*}
After applying Schl\"afli's formula (\ref{Formula_Schlafli_1}),
one is left with the total differential
\begin{equation} \label{equation_differential_of_local_U}
d\UDelta = \sum_{ij \in E_{\Delta}^1}
(\alpha^{\Delta}_{ij}-\check{\alpha}^{\Delta}_{ij})\, d a_{ij} +
\sum_{k \in V_{\Delta}^1}
(\beta^{\Delta}_k-\check{\beta}^{\Delta}_{k}) \, d b_k.
\end{equation}
Thus, one can read off the partial derivatives from
(\ref{equation_differential_of_local_U}). They are the
coefficients of the differential $d \UDelta$.
\end{proof}
\begin{lem} \label{Lem_convexity_of_local_U}
Let $\Uglobal_{1,\Delta}$ be the restriction of the function
$\UDelta$
$$\Uglobal_{1,\Delta} = \UDelta |_{\TE_{1,\Delta}}
\, : \, \TE_{1,\Delta} \, \to \, \reals.$$ Then
$\Uglobal_{1,\Delta}$ is a locally strictly convex function on
$\TE_{1,\Delta}$. Thus, the function $\UDelta \, : \, \TED \, \to
\, \reals$ is \emph{locally strictly convex on $\TED$
transversally to the orbits of the $\reals-$action on $\TED$}. In
other words, any point of $\TED$ has a convex neighborhood $B
\subset \TED$ such that for any two points $(a^1,b^1)_{\Delta}$
and $(a^2,b^2)_{\Delta} \in \TED$ with the property that
$(a^2,b^2)_{\Delta} \neq ACT^{\Delta}_t(a^1,b^1)$ for all $t \in
\reals$ and for any $\lambda \in [0,1]$
$$\UDelta\Big((1-\lambda) (a^1,b^1)_{\Delta} + \lambda (a^2,b^2)_{\Delta}\Big)
< (1-\lambda) \UDelta(a^1,b^1) + \lambda \UDelta(a^2,b^2).$$
\end{lem}

\begin{proof}
By lemma \ref{Lem_derivatives_Ulocal} and the construction of the
diffeomorphism $\Phi_{1,\Delta} \, : \, \TE_{1,\Delta} \, \to \,
\ADelta^1$ presented in section \ref{Sec_Volumes_of_tetrahedra},
one sees that for any $(a,b)_{\Delta} \in \TE_{1,\Delta}$ the
identity $\Phi_{1,\Delta}(a,b) =
\big(\nabla\Uglobal_{1,\Delta}\big)_{(a,b)} +
(\check{\tilde{\alpha}}^{\Delta},\check{\tilde{\beta}}^{\Delta})$
holds for the fixed set of angles
$(\check{\tilde{\alpha}}^{\Delta},\check{\tilde{\beta}}^{\Delta})
\in \ADelta^1$, where
$J_{\Delta}(\check{\tilde{\alpha}}^{\Delta},\check{\tilde{\beta}}^{\Delta})
=(\check{\alpha}^{\Delta},\check{\beta}^{\Delta}) \in \ADelta$ are
the angles of the decorated triangles constructed in lemma
\ref{Lem_existence_of_decorated_polygon}. For the definition of
$(\check{\alpha}^{\Delta},\check{\beta}^{\Delta})$ see section
\ref{Sec_defining_space_circle_patterns} and for the definition of
the map $J_{\Delta}$ see section \ref{Sec_Volumes_of_tetrahedra}.
By taking the derivative of $\Phi_{1,\Delta}$, one obtains the
hessian of $\Uglobal_{1,\Delta}$
$$\Big(D\Phi_{1,\Delta}\Big)_{(a,b)_{\Delta}}=\text{Hess}(\Uglobal_{1,\Delta})_{(a,b)_{\Delta}}.$$
By lemma \ref{Lem_volume_concavity}
$$\Big(D \Phi_{1,\Delta}^{-1}\Big)_{\Phi_{1,\Delta}(a,b)}=
-2 \text{Hess}(\Vol_{1,\Delta})_{\Phi_{1,\Delta}(a,b)}.$$ Since
$\Big(D \Phi_{1,\Delta}^{-1}\Big)_{\Phi_{1,\Delta}(a,b)}=
\Big(D\Phi_{1,\Delta}\Big)_{(a,b)_{\Delta}}^{-1}$, the two
hessians are related by
$$\text{Hess}(\Uglobal_{1,\Delta})_{(a,b)_{\Delta}} =-\frac{1}{2}
\Big(\text{Hess}(\Vol_{1,\Delta})_{\Phi_{1,\Delta}(a,b)}\Big)^{-1}.$$
By lemma \ref{Lem_volume_concavity}, the hessian
$\text{Hess}(\Vol_{1,\Delta})_{\Phi_{1,\Delta}(a,b)}$ is negative
definite, so its inverse matrix multiplied by $-1/2$ is positive
definite. Thus $\text{Hess}(\Uglobal_{1,\Delta})_{(a,b)_{\Delta}}$
is positive definite for each $(a,b)_{\Delta} \in \TE_{1,\Delta}$.
Therefore, the function $\Uglobal_{1,\Delta}$ is strictly locally
convex in the domain $\TE_{1,\Delta}$.

Now, pick an arbitrary point in $\TED$ and choose any convex
neighborhood $B \subset \TED$ around this point. Take two
arbitrary points $(a^1,b^1)_{\Delta}$ and $(a^2,b^2)_{\Delta} \in
B$ not on the same $\reals-$orbit. Then
$PR_1(a^1,b^1)_{\Delta}\neq PR_1(a^2,b^2)_{\Delta}$. Choose any
$\lambda \in [0,1]$. Then by the identity
$\UDelta=\Uglobal_{1,\Delta} \circ PR_1$, by the linearity of
$PR_1$ and by the strict local convexity of $\Uglobal_{1,\Delta}$,
it follows that
\begin{align*}
\UDelta\Big((1-\lambda)(a^1,b^1)_{\Delta} + &\lambda
(a^2,b^2)_{\Delta}\Big) =
\Uglobal_{1,\Delta}\Big(PR_1\big((1-\lambda)(a^1,b^1)_{\Delta} +
\lambda
(a^2,b^2)_{\Delta}\big)\Big)\\
=& \, \Uglobal_{1,\Delta}\Big((1-\lambda)PR_1(a^1,b^1)_{\Delta} +
\lambda
PR_1(a^2,b^2)_{\Delta}\Big) \\
<& (1-\lambda) \Uglobal_{1,\Delta}\Big(PR_1(a^1,b^1)_{\Delta}\Big)
+
\lambda \Uglobal_{1,\Delta}\Big(PR_1(a^2,b^2)_{\Delta}\Big) \\
=& (1-\lambda) \UDelta(a^1,b^1) + \lambda \UDelta(a^2,b^2).
\end{align*}
\end{proof}
\noindent Our next step is to define a global functional for the
triangulation $\Triang=(V,E_T,F_T)$. For any $(a,b) \in \TE$ let
\begin{equation} \label{Eqn_Global_Functional}
\Uglobal_{\Triang}(a,b) = \sum_{\Delta \in F_T} \UDelta(a,b).
\end{equation}
As we will see later, this functional will have a profound impact
on our study. But before that, let us write down a modification of
(\ref{Eqn_Global_Functional}). First, in the case
$\PP=\Euclideanplane$ let us define the affine subspace
$$\GB = \Big\{ (\theta,\Theta) \in \reals^{E_1}\times
\reals^{V_1} \,\, \big{|} \,\, \sum_{k \in V} (2 \pi - \Theta_k) =
2\pi \chi(S) \Big\}.$$ We are interested in this space because the
cone angles $\Theta$ at the vertices of a circle pattern should
always satisfy the global Gauss-Bonnet theorem. In the case
$\PP=\hyperbolicplane$ we simply let $\GB = \reals^{E_1} \times
\reals^{V_1}$. For the specifics of the notations, see the
paragraph right after theorem \ref{Thm_description_of_polytopes}.
Let $(\theta,\Theta) \in \GB$. Define
\begin{equation} \label{Eqn_Global_Functional_Theta}
\Uglobal_{\theta,\Theta}(a,b) = \sum_{\Delta \in F_T} \UDelta(a,b)
- \sum_{ij \in E_1\cup E_{\pi}} (\theta_{ij} -
\check{\theta}_{ij}) a_{ij} -
 \sum_{k \in V_1} (\Theta_{k} - \check{\Theta}_{k}) b_{k}.
\end{equation}
In this formula $(\check{\theta},\check{\Theta}) \in \GB$
represents the fixed angle data extracted from the special circle
pattern constructed in lemma \ref{Lem_nonempty_space_of_patterns}
with the help of lemma \ref{Lem_existence_of_decorated_polygon}
(see section \ref{Sec_defining_space_circle_patterns}). By
proposition \ref{Prop_alpha_alpha_equals_theta}, for any two
triangles $\Delta$ and $\Delta'$ of $\Triang$ that share a common
edge $ij \in E$ we have $\check{\theta}_{ij} =
\check{\alpha}^{\Delta}_{ij} + \check{\alpha}^{\Delta'}_{ij}$.
Furthermore, $\check{\Theta}_k=\sum_{\Delta \in F_k}
\check{\beta}^{\Delta}$. Here, $F_k$ is the set of all faces that
share the same vertex $k \in V$. Similarly, we will use the
notation $E_k$ for the set of all edges adjacent to the same
vertex $k$. With all these properties in mind, one can verify that
both functions $\Uglobal_{\Triang}$ and $\Uglobal_{\theta,\Theta}$
are $\reals-$invarinat, i.e. $\Uglobal_{\Triang} \circ ACT_t =
\Uglobal_{\Triang}$ and $\Uglobal_{\theta,\Theta} \circ ACT_t =
\Uglobal_{\theta,\Theta}$ for all $t \in \reals$.

 \emph{Remark.} We remind the reader that in the case
$\PP=\hyperbolicplane$ the action $ACT_t$ is trivial and thus some
terminology and notations are redundant. For instance the
functions $\Uglobal_{\Triang}$ and $\Uglobal_{\theta,\Theta}$ are
locally strictly convex on $\TE$ and there is no need to introduce
the notion of convexity transverse to the action $ACT_t$.
Nevertheless we keep these notations and terminology since they
allows us to treat both the Euclidean and the hyperbolic case
simultaneously.

\begin{lem} \label{Lem_convexity_of_Uglobal_and_Utheta}
The functions $\Uglobal_{\Triang} \, : \, \TE \, \to \, \reals$
and $\Uglobal_{\theta,\Theta} \, : \, \TE \, \to \, \reals$ are
convex locally strictly convex in $\TE$, transversally to the
orbits of the $\reals-$action on $\TE$. In other words, any point
of $\TE$ has a convex neighborhood $B \subset \TE$ such that for
any two points $(a^1,b^1)$ and $(a^2,b^2) \in \TE$ with the
property that $(a^2,b^2) \neq ACT_t(a^1,b^1)$ for all $t \in
\reals$ and for any $\lambda \in [0,1]$
\begin{align*}
&\Uglobal_{\Triang}\Big((1-\lambda) (a^1,b^1) + \lambda
(a^2,b^2)\Big) < (1-\lambda) \Uglobal_{\Triang}(a^1,b^1) + \lambda
\Uglobal_{\Triang}(a^2,b^2) \,\,\,\, \text{ and}\\
&\Uglobal_{\theta,\Theta}\Big((1-\lambda) (a^1,b^1) + \lambda
(a^2,b^2)\Big) < (1-\lambda) \Uglobal_{\theta,\Theta}(a^1,b^1) +
\lambda \Uglobal_{\theta,\Theta}(a^2,b^2).
\end{align*}
\end{lem}

\begin{proof}
For $\Delta \in F_T$ each function $\UDelta$ has variables that
belong to $\reals^{E_{\Delta}^1}\times\reals^{V_{\Delta}^1}$, so
naturally it can be regarded as a function with variables
belonging to the bigger space $\reals^{E}\times\reals^{V}$
depending explicitly only on the variables related to $\Delta$. 

Let $(a^0,b^0) \in \TE$ be an arbitrary point and let
$B\subset\TE$ be an open convex set, containing $(a^0,b^0)$. Take
any two points $(a^1,b^1)$ and $(a^2,b^2) \in B$ such that
$(a^2,b^2) \neq ACT_t(a^1,b^1)$ for all $t \in \reals$. The latter
condition is equivalent to the fact that there exists at least one
$\Delta_0 \in F_T$ such that $(a^2,b^2)_{\Delta_0}\neq
ACT^{\Delta_0}_t(a^1,b^1)_{\Delta_0} \in \TE_{\Delta_0}$ for all
real numbers $t$. For any $\Delta \in F_T$ define
$B_{\Delta}=\big\{(a,b)_{\Delta} \in \TED \,\, | \,\,\, (a,b) \in
B \big\}$ which is simply the projection of $B$ onto $\TED$.
Therefore $B_{\Delta}$ is an open convex subset of $\TED$. In
particular $(a^1,b^1)_{\Delta_0}$ and $(a^2,b^2)_{\Delta_0} \in
B_{\Delta_0}.$ Choose any $\lambda \in [0,1]$. 
For any $\Delta \in F_T$
\begin{equation} \label{Eqn_non_strict_UDelta_convexity_inequality}
\UDelta\Big((1-\lambda)(a^1,b^1) + \lambda (a^2,b^2)\Big) \leq
(1-\lambda) \UDelta(a^1,b^1) + \lambda \UDelta(a^2,b^2).
\end{equation}
Inequality (\ref{Eqn_non_strict_UDelta_convexity_inequality})
comes from the fact that if $(a^1,b^1)_{\Delta}$ and
$(a^2,b^2)_{\Delta}$ do not lie in the same orbit then the
inequality is strict by lemma \ref{Lem_convexity_of_local_U}.
Otherwise, if $ACT^{\Delta}_{t_0}(a^1,b^1)_{\Delta} =
(a^2,b^2)_{\Delta}$ for some $t_0 \in \reals$ then we have an
identity due to the $\reals-$invariance of $\UDelta$. However, for
the triangle $\Delta_0$ inequality
(\ref{Eqn_non_strict_UDelta_convexity_inequality}) is strict.
Therefore
\begin{align*}
&\Uglobal_{\Triang}\Big((1-\lambda)(a^1,b^1) + \lambda
(a^2,b^2)\Big)= \\
= &\Uglobal_{\Delta_0}\Big((1-\lambda)(a^1,b^1)
+ \lambda (a^2,b^2)\Big) + \sum_{\Delta \in
F_T\setminus\{\Delta_0\}}
\UDelta\Big((1-\lambda)(a^1,b^1) + \lambda (a^2,b^2)\Big)\\
\leq &\Uglobal_{\Delta_0}\Big((1-\lambda)(a^1,b^1) + \lambda
(a^2,b^2)\Big) + \sum_{\Delta \in F_T\setminus\{\Delta_0\}}
(1-\lambda)
\UDelta\big(a^1,b^1\big) + \lambda \UDelta\big(a^2,b^2\big)\\
< &(1-\lambda) \Uglobal_{\Delta_0}\big(a^1,b^1\big) + \lambda
\Uglobal_{\Delta_0}\big(a^2,b^2\big) +\\
&\phantom{ 1 - \lambda \Uglobal_{\Delta_0}\big(a^1,b^1\big) +
\lambda \Uglobal_{\Delta_0}\big(a^2,b^2\big) } + \sum_{\Delta \in
F_T\setminus\{\Delta_0\}} (1-\lambda) \UDelta\big(a^1,b^1) +
\lambda \UDelta\big(a^2, b^2\big)\\
= &(1-\lambda) \Uglobal_{\Triang}(a^1,b^1) + \lambda
\Uglobal_{\Triang}(a^2,b^2).
\end{align*}
Consequently, the function $\Uglobal_{\theta,\Theta}$ is also
strictly transversally locally convex on $\TE$
because it 
is a sum of a linear function with the function
$\Uglobal_{\Triang}$.
\end{proof}

\section{The angle extraction map} \label{Sec_angle_map}

We are ready to put into action all the constructions carried out
so far. Define the map
$$ \PhiT \, : \, \CPT \, \to \, \reals^{E_1 \cup E_{\pi}} \times
\reals^{V_1}$$  
which associates to a generalized hyper-ideal circle pattern its
vector $(\theta,\Theta) \in \reals^{E_1 \cup E_{\pi}} \times
\reals^{V_1}$ of angles between adjacent face circles and cone
angles at the vertex points. Clearly, we can take any other circle
pattern with the same combinatorics, which is isometric (and
scaled if applicable) to the initial one, and still obtain the
same angles as a result. Therefore, $\PhiT$ is well defined on the
isometry (or scaling) classes of generalized circle patterns from
$\CPT$. To understand $\PhiT$ better, we think that it is written
in $\TE$ coordinates, i.e. we think that $\CPT = \TE_0$, and
$$\PhiT \, : \, \TE_0 \, \to \, \reals^{E_1 \cup E_{\pi}} \times
\reals^{V_1}.$$ Recall that $\reals^{E_1 \cup E_{\pi}} \times
\reals^{V_1}$ is regarded as an affine subset of $\reals^{E_T}
\times \reals^{V}$ according to the paragraph right after theorem
\ref{Thm_description_of_polytopes}. Due to the global Gauss-Bonnet
theorem it happens so that $\PhiT(\TE_0) \subset \GB$. We also
would consider the extension $\tPhiT : \TE \to \GB$ defined as
$\tPhiT = \PhiT\circ PR$, which is by construction
$\reals-$invariant.

\begin{lem} \label{Lem}
The real analytic functions $\Uglobal_{\Triang}$ and
$\Uglobal_{\theta,\Theta}$ have partial derivatives
\begin{align}
\frac{\partial\Uglobal_{\Triang}}{\partial a_{ij}} &=
\alpha^{\Delta}_{ij} + \alpha^{\Delta'}_{ij} - \check{\theta}_{ij}
&\frac{\partial\Uglobal_{\Triang}}{\partial b_{k}} = \sum_{\Delta
\in F_k} \beta^{\Delta}_{k} -
\check{\Theta}_{k} \label{Formula_partial_derivative_Utriang}\\
\frac{\partial\Uglobal_{\theta,\Theta}}{\partial a_{ij}} &=
\alpha^{\Delta}_{ij} + \alpha^{\Delta'}_{ij} - {\theta}_{ij}
&\frac{\partial\Uglobal_{\theta,\Theta}}{\partial b_{k}} =
\sum_{\Delta \in F_k} \beta^{\Delta}_{k} - \Theta_{k}
\label{Formula_partial_derivatives_Utheta}
\end{align}
where $\Delta$ and $\Delta' \in F_T$ are the two faces having $ij
\in E_1 \cup E_{\pi}$ as a common edge, $k\in V_1$ and $(\theta,
\Theta)$ is an arbitrary vector from $\reals^{E_1\cup
E_{\pi}}\times\reals^{V_1}$. Furthermore,
\begin{align*}
&\tPhiT(a,b) = \big( \nabla \Uglobal_{\Triang} \big)_{(a,b)} +
(\check{\theta}, \check{\Theta})  \,\,\,\, \text{for all} \,\,
(a,b) \in \TE\\
&\PhiT = \tPhiT{|}_{\TE_0}.
\end{align*}
Moreover, for an arbitrary $(\theta,\Theta) \in \reals^{E \cup
E_{\pi}} \times \reals^{V_1}$ there exists $(a,b) \in \TE_0$ such
that $\PhiT(a,b) = (\theta,\Theta)$ if and only if $(a,b)$ is a
critical point of $\Uglobal_{\theta,\Theta}$ inside the domain
$\TE_0.$ The critical point is unique. Finally, $\PhiT$ is a real
analytic diffeomorphism between $\TE_0$ and its image
$\PhiT(\TE_0) \, \subset \, \reals^{E_1\cup
E_{\pi}}\times\reals^{V_1}$.
\end{lem}
\begin{proof}
According to the definition (\ref{Eqn_Global_Functional}) of
$\Uglobal_{\Triang}$, there are only two terms that explicitly
contain $a_{ij}$. These are $\UDelta(a,b)$ and
$\Uglobal_{\Delta'}(a,b)$, where $\Delta$ and $\Delta' \in F_T$
are the two faces having $ij \in E_1 \cup E_{\pi}$ as a common
edge. According to lemma \ref{Lem_derivatives_Ulocal},
\begin{align*}
\frac{\partial\Uglobal_{\Triang}}{\partial a_{ij}} &=
\frac{\partial\UDelta}{\partial a_{ij}} +
\frac{\partial\Uglobal_{\Delta'}}{\partial a_{ij}} =
\alpha^{\Delta}_{ij} - \check{\alpha}^{\Delta}_{ij} +
\alpha^{\Delta'}_{ij} - \check{\alpha}^{\Delta'}_{ij} =
\alpha^{\Delta}_{ij} + \alpha^{\Delta'}_{ij} - \check{\theta}_{ij}
\end{align*}
Similarly, the only terms $\UDelta(a,b)$ that explicitly depend on
$b_k$ are the ones for which $\Delta$ has $k$ as its vertex. Then
\begin{align*}
\frac{\partial\Uglobal_{\Triang}}{\partial b_{k}} &= \sum_{\Delta
\in F_k} \frac{\partial\UDelta}{\partial b_{k}} = \sum_{\Delta \in
F_k} \big(\beta^{\Delta}_{k} - \check{\beta}^{\Delta}_{k}\big) =
\sum_{\Delta \in F_k} \beta^{\Delta}_{k} - \check{\Theta}_{k}
\end{align*}
Consequently, the formulas
(\ref{Formula_partial_derivatives_Utheta}) for the partial
derivatives of $\Uglobal_{\theta,\Theta}$ follow easily from the
fact that
$$
\Uglobal_{\theta,\Theta}(a,b) = \Uglobal_{\Triang}(a,b) - \sum_{ij
\in E_1\cup E_{\pi}} \big(\theta_{ij} - \check{\theta}_{ij}\big)\,
a_{ij} -
 \sum_{k \in V_1} (\Theta_{k} - \check{\Theta}_{k}) \,  b_{k}
$$
according to expression (\ref{Eqn_Global_Functional_Theta}).
Formulas (\ref{Formula_partial_derivative_Utriang}) can be written
as
\begin{align*}
\alpha^{\Delta}_{ij} + \alpha^{\Delta'}_{ij} =
\frac{\partial\Uglobal_{\Triang}}{\partial a_{ij}}(a,b) +
\check{\theta}_{ij} \,\,\,\,\,\,\, &\sum_{\Delta \in F_k}
\beta^{\Delta}_{k} = \frac{\partial\Uglobal_{\Triang}}{\partial
b_{k}}(a,b)
 + \check{\Theta}_{k}
\end{align*}
As already discussed, each $(a,b) \in \TE$ gives rise to a
generalized hyper-ideal circle pattern, unique up to isometry (and
scaling). Theorem \ref{Prop_alpha_alpha_equals_theta} implies that
the number $\theta_{ij} = \alpha^{\Delta}_{ij} +
\alpha^{\Delta'}_{ij} \in (0,2\pi)$ for each $ij \in E_T$ is the
intersection angle between two adjacent face circles of the circle
pattern $(a,b)$ (compare with figure \ref{Fig2}a). Analogously,
the cone angle of the pattern $(a,b)$ at any vertex $k \in V_1$ is
$\Theta_k = \sum_{\Delta \in F_k} \beta^{\Delta}_{k} \, > \, 0$.
Consequently,
\begin{align*}
&(\theta,\Theta) = \tPhiT(a,b) = \big( \nabla \Uglobal_{\Triang}
\big)_{(a,b)} + (\check{\theta}, \check{\Theta})  \,\,\,\,
\text{for} \,\, (a,b) \in \TE \,\,\,\,\text{ and } \,\,\, \PhiT =
\tPhiT{|}_{\TE_0}.
\end{align*}
Furthermore, $(a,b) \in \TE_0$ is a critical point of
$\Uglobal_{\theta,\Theta}$ exactly when $\big(d
\Uglobal_{\theta,\Theta}\big)_{(a,b)} = 0,$ which means
$\big(\nabla\Uglobal_{\theta,\Theta}\big)_{(a,b)}=0$.  With the
latter fact in mind, formula
(\ref{Formula_partial_derivatives_Utheta}) reveals that $(a,b) \in
\TE_0$ is a critical point of $\Uglobal_{\theta,\Theta}$ if and
only if for each $ij \in E_1 \cup E_{\pi}$ the a priori assigned
number $\theta_{ij} \in (0, 2\pi)$ is equal to the intersection
angle $ \alpha^{\Delta}_{ij} + \alpha^{\Delta'}_{ij}$ of two
adjacent face circles of the pattern represented by $(a,b)$, as
well as the number $\Theta_k$ is the cone angle $\sum_{\Delta \in
F_k} \beta^{\Delta}_{k} \, > \, 0$ of $(a,b)$ at the cone point $k
\in V$. The function $\Uglobal_{\theta,\Theta}$ restricted to
$\TE_0$ is locally strictly convex, so $(a,b)$ is unique.
Consequently, the map $\PhiT$ is a bijection between $\TE_0$ and
its image $\PhiT(\TE_0)$, due to the fact that the pre-image of
each $(\theta,\Theta) \in \PhiT(\TE_0)$ is the unique critical
point of the functional $\Uglobal_{\theta,\Theta}$ on $\TE_0$. The
derivative of the expression $\tPhiT(a,b) = \big( \nabla
\Uglobal_{\Triang} \big)_{(a,b)} + (\check{\theta},
\check{\Theta})$ is the matrix $\big(D\tPhiT\big)_{(a,b)} =
\text{Hess}\big(\Uglobal_{\Triang}\big)_{(a,b)}$ which restricted
to the tangent space of $\TE_0$ at $(a,b)$ is positive definite by
lemma \ref{Lem_convexity_of_Uglobal_and_Utheta}, and thus
invertible. By the inverse mapping theorem, $\PhiT$ is a local
real analytic diffeomorphism and since it is one-to-one, it is
also a global real analytic diffeomorphism between $\TE_0$ and
$\PhiT(\TE_0)$.
\end{proof}

\section{The space of true hyper-ideal circle patterns} \label{Sec_space_patterns_C}

We focus our attention on the original cell complex
$\cellcomplex=(V,E,F)$ on the surface $S$. Denote by $\CPC$ the
space of either hyperbolic or Euclidean hyper-ideal circle
patterns on $S$ with combinatorics $\cellcomplex$, considered up
to isometries (and global scaling in the Euclidean case) which
preserve the induced by $\cellcomplex$ marking on $S$. Observe
that now we consider only circle patterns satisfying the local
Delaunay property. Our goal is to show that $\CPC$ is a real
analytic manifold.

\begin{prop} \label{Prop_space_of_patterns_C}
The space $\CPC$ is a real analytic manifold of dimension
$|V_1|+|E_1|$ in the hyperbolic case and $|V_1|+|E_1|-1$ in the
Euclidean case. It can be realized as a real analytic submanifold
of $\TE_0$. Consequently the map
$$\PhiC = \PhiT |_{\CPC}\, : \, \CPC \, \to \,
\reals^{E_{\pi} \cup E_1} \times \reals^{V_1}$$ is a real-analytic
diffeomorphism between $\CPC$ and $\PhiC(\CPC)$.
\end{prop}
\begin{proof}
For the sake of this proof let
$$\mathcal{N}_{\cellcomplex} = \Big\{ (\theta,\Theta)
\in \reals^{E_1\cup E_{\pi}}\times \reals^{V_1} \,\, \big{|} \,\,
\, \theta_{ij}=\pi \text{ for } ij \in E_{\pi} \Big\}.$$ One sees
that $\mathcal{N}_{\cellcomplex}$ is an affine subspace of
$\reals^{E_1\cup E_{\pi}}\times \reals^{V_1}$. In its own turn,
the latter is regarded as an affine subspace of
$\reals^{E_T}\times \reals^{V}$ by letting $\theta_{ij}=0$ for $ij
\in E_0$ and $\Theta_k = \sum_{ik \in E_k}  (\pi - \theta_{ik})$
for $k \in V_0$. Whenever $\PP=\hyperbolicplane$ let
$\mathcal{M}_{\cellcomplex}=\mathcal{N}_{\cellcomplex}$ and
whenever $\PP=\Euclideanplane$
$$\mathcal{M}_{\cellcomplex} = \Big\{ (\theta,\Theta)
\in \mathcal{N}_{\cellcomplex} \,\, \big{|} \,\,  \sum_{k \in V}
(2\pi - \Theta_k) = 2\pi\chi(S) \,
 \Big\}.$$ Just like before, $\mathcal{M}_{\cellcomplex}$
can be regarded as an affine subspace of $\reals^{E_1\cup
E_{\pi}}\times \reals^{V_1}$ and is contained in
$\mathcal{N}_{\cellcomplex}$. In terms of dimensions
$\dim{\mathcal{M}_{\cellcomplex}} =
\dim{\mathcal{N}_{\cellcomplex}} = |E_1| + |V_1|$ whenever
$\PP=\hyperbolicplane$ and $\dim{\mathcal{M}_{\cellcomplex}} =
\dim{\mathcal{N}_{\cellcomplex}} - 1 = |E_1| + |V_1| - 1$ whenever
$\PP=\Euclideanplane$. Furthermore, define
$\mathcal{M}_{\cellcomplex}^{D} = \big\{ (\theta,\Theta) \in
\mathcal{M}_{\cellcomplex} \,\, \big{|} \,\, \theta_{ij} \in
(0,\pi) \text{ for } ij \in E_1 \,  \big\}$, which is an open
polytopal subset of the affine space $\mathcal{M}_{\cellcomplex}$.
Let
$$\CPC = \PhiT^{-1}\Big(\PhiT(\TE_0) \cap \mathcal{M}_{\cellcomplex}^{D}\Big)$$
It is a real analytic submanifold of $\TE_0$, because $\PhiT^{-1}$
is a real analytic diffeomorphism between $\PhiT(\TE_0)$ and
$\TE_0$, and $\PhiT(\TE_0) \cap \mathcal{M}_{\cellcomplex}^{D}$ is
an open subset of the affine subspace $\mathcal{M}_{\cellcomplex}$
lying inside $\reals^{E_1\cup E_{\pi}}\times \reals^{V_1}$.

Now, let us have a hyper-ideal circle pattern on $S$ with
combinatorics $\cellcomplex$ (see figure \ref{Fig1}a). Then one
can subdivide its geodesic cell complex to obtain a geodesic
triangulation with combinatorics $\Triang$ (see figure
\ref{Fig3}b). The edge-lengths and the vertex radii associated to
this geodesic triangulation form $(l,r) \in \ER.$ In fact, one can
assume that after rescaling $(l,r) \in \ER_1$. By lemma
\ref{Lem_map_between_TE_and_ER}, one sees that $(a,b) =
\Psi^{-1}(l,r) \in \TE_0$. The angle data of the circle pattern,
regarded as a pattern with combinatorics $\Triang$, is
$(\theta,\Theta) = \PhiT(a,b) \in \PhiT(\TE_0)$. Since the
subdividing edges $E_{\pi} = E_T \setminus E$ of $\Triang$
 are redundant, $\theta_{ij}=\pi$ for $ij \in E_{\pi}$ due to proposition
\ref{Prop_special_angles_bw_face_circles}. For example, such
redundant edges on figure \ref{Fig2}a are $is$ and $us \in
E_{\pi}$. The latter fact combined with the fact that a
hyper-ideal circle pattern satisfies the local Delaunay property
yields $(\theta,\Theta) \in \PhiT(\TE_0) \cap
\mathcal{M}_{\cellcomplex}^{D}$. Hence $(a,b) \in \CPC$.

Conversely, let $(a,b) \in \CPC.$ Then, $(a,b)$ defines a
generalized hyper-ideal circle pattern with combinatorics
$\Triang$. By the nature of the map $\PhiT$ and the definition of
$\CPC$, we have that $(\theta,\Theta) = \PhiT(a,b) \in
\PhiT(\TE_0) \cap \mathcal{M}_{\cellcomplex}^{D}$. Therefore
$\theta_{ij}=\pi$ for each $ij \in E_{\pi}.$ Proposition
\ref{Prop_special_angles_bw_face_circles} implies that all
geodesic edges from $E_{\pi}$ are redundant edges of the
generalized circle pattern $(a,b)$, so in fact its combinatorics
is $\cellcomplex$ (for example, see figure \ref{Fig2}a).
Furthermore, the generalized pattern satisfies the local Delaunay
property so it is a hyper-ideal circle pattern.

We can conclude that $\CPC$ is indeed the space of either
Euclidean or hyperbolic hyper-ideal circle patterns on $S$ with
combinatorics $\cellcomplex$, considered up to isometries (and
global scaling in the Euclidean case) which preserve the induced
by $\cellcomplex$ marking on $S$.

Finally, both spaces $\CPC$ and $\PhiT(\CPC) = \PhiC(\CPC)$ are
nonempty because of lemma \ref{Lem_nonempty_space_of_patterns}.
Therefore, both $\CPC$ and $\PhiC(\CPC)$ are real analytic
manifolds of dimension $|E_1|+|V_1|$ when $\PP=\hyperbolicplane$
and $|E_1|+|V_1|-1$ when $\PP=\Euclideanplane$. \end{proof}

\section{Necessary conditions for existence of circle patterns}
\label{Sec_necessary_conditions}

Before we continue our exposition, we recapitulate and discuss
some notational assumptions. As usual, the surface $S$ has a fixed
cell decomposition $\cellcomplex = (V,E,F)$ which is further
subdivided into a triangulation $\Triang=(V,E_T,F_T)$ as explained
in section \ref{Sec_Generalized_space_of_patterns}. The set of
edges $E_T$ is partitioned into $E_T= E \cup E_{\pi} = E_1 \cup
E_0 \cup E_{\pi}$ and the set of vertices is partitioned into $V =
V_1 \cup V_0$. We assume that the space $\reals^{E_1} \times
\reals^{V_1}$ is an affine subspace of $\reals^{E_1\cup E_{\pi}}
\times \reals^{V_1}$ by taking $\theta_{ij} = \pi$ for all
auxiliary edges $ij \in E_{\pi}$. In its own turn,
$\reals^{E_1\cup E_{\pi}} \times \reals^{V_1}$ is viewed as an
affine subspace of the total space $\reals^{E_T}\times \reals^{V}$
by considering that $\theta_{ij} = 0$ for $ij \in E_0$ and
$\Theta_k = \sum_{ij \in E_k} (\pi - \theta_{ij})$ for $k\in V_0$.
Consequently both spaces $\mathcal{N}_{\cellcomplex}$ and
$\mathcal{M}_{\cellcomplex}$ can be regarded as affine subspaces
of $\reals^{E_T}\times \reals^{V}$. Furthermore, we can think that
the polytopes $\polytope^{e}$ and $\polytope^{h}$ defined in
theorem \ref{Thm_description_of_polytopes} also lie in
$\reals^{E_T}\times \reals^{V}$ as open polytopes of the affine
space $\mathcal{M}_{\cellcomplex} \, \subset \, \reals^{E_T}\times
\reals^{V}$. We adopt the common notation $\polytope$ for
$\polytope^{e}$ and $\polytope^{h}$ whenever a distinction between
the two is not necessary.

Let us summarize our central results up to now. So far we have
established that the space of hyper-ideal circle patterns with
combinatorics $\cellcomplex=(C,E,F)$ is the real analytic manifold
$\CPC$. Moreover, we have a natural map $\PhiC \, : \, \CPC \, \to
\, \mathcal{M}_{\cellcomplex}$ which assigns to a circle pattern
the angle data $(\theta,\Theta) \in \mathcal{M}_{\cellcomplex}$
consisting of the intersection angles $\theta$ between pairs of
adjacent face circles of the pattern as well as the cone angles
$\Theta$ at the vertices $V$. It has been established in
proposition \ref{Prop_space_of_patterns_C} that the map $\PhiC$ is
a real analytic diffeomorphism between $\CPC$ and the image
$\PhiC(\CPC) \subset \mathcal{M}_{\cellcomplex}$. Furthermore
$\dim \CPC = \dim \mathcal{M}_{\cellcomplex}$ which means that
$\PhiC(\CPC)$ is an open subset of $\mathcal{M}_{\cellcomplex}$.
To conclude the proof of theorem \ref{Thm_main} and theorem
\ref{Thm_description_of_polytopes}, we need to show that
$\PhiC(\CPC) = \polytope$. This will be done in two steps. First,
we will establish the inclusion $\PhiC(\CPC) \subseteq \polytope$,
i.e. the conditions of theorem \ref{Thm_description_of_polytopes}
are necessary. Second, we will show that the set $\PhiC(\CPC)$ is
relatively closed subset of the polytope $\polytope$. Since the
image $\PhiC(\CPC)$ is both open and relatively closed subset of
$\polytope$, it has to coincide with a connected component of
$\polytope$. But any convex polytope is connected, so $\polytope$
has only one connected component. Therefore $\PhiC(\CPC) =
\polytope$.

In this section, we establish the inclusion $\PhiC(\CPC) \subseteq
\polytope$. The main tool in the proof of this fact is the famous
Gauss-Bonnet formula, which is quite natural having in mind that
we work with angle data. The approach in \cite{Sch1} is very
similar to ours, but it contains some errors which we will point
out and address in the course of our proof.

Let the vector $(a,b) \in \CPC$ represent a given circle pattern
with combinatorics $\cellcomplex$ and let $(l,r) = \Psi(a,b) \in
\ER_1$ be the corresponding edge-lengths and vertex radii of the
pattern (see lemma \ref{Lem_map_between_TE_and_ER}). Let us denote
by $S_{l,r}$ the geometrizaton of $S$ via $(l,r)$, that is
$S_{l,r} $ is a geometric surface homeomorphic to $S$ together
with a geodesic cell complex with combinatorics $\cellcomplex$ on
which the circle pattern represented by $(l,r)$ is realized.
Observe that $S_{l,r}$ is naturally obtained by first subdividing
$\cellcomplex$ into a triangulation $\Triang$. This subdivision is
achieved by adding the auxiliary edges $E_{\pi}$. Then, one
assigns to each topological triangle from $\Triang$ the
edge-lengths $l$ and the vertex radii $r$ so that each
combinatorial triangle becomes a geometric decorated triangle.
Finally, by construction, the decorated triangles are compatibly
adjacent so they form a geometric surface $S_{l,r}$ with a
geodesic triangulation $\Triang_{l,r}$ combinatorially isomorphic
to $\Triang$, carrying the hyper-ideal circle pattern. However,
since $\theta_{ij}=\pi$ for all $ij \in E_{\pi}$, proposition
\ref{Prop_special_angles_bw_face_circles} guarantees that the face
circles of two decorated triangles that share a common auxiliary
edge from $E_{\pi}$ coincide, making this edge redundant (we can
erase it). Consequently, the combinatorics of the hyper-ideal
circle pattern is actually $\cellcomplex$. We denote by
$\cellcomplex_{l,r}$ the geodesic cell decomposition of $S_{l,r}$
obtained from $\Triang_{l,r}$ by erasing the auxiliary (redundant)
edges of type $E_{\pi}$ on $S_{l,r}$. This situation is shown on
figure \ref{Fig2}a, where the edges $is$ and $us$ are redundant
edges with $\theta_{is}=\theta_{us}=\pi$. The geometric complex
$\cellcomplex_{l,r}$ has the same combinatorics as $\cellcomplex$.

Recall that in section \ref{Sec_angle_data_polytope} we introduced
the dual complex $\cellcomplex^*=(V^*,E^*,F^*)$ (figure
\ref{Fig1}b) as well as the additional triangulation
$\hat{\Triang}=(\hat{V}, \hat{E},\hat{F})$ (figure \ref{Fig3}a)
both associated to the cell complex $\cellcomplex$ (figure
\ref{Fig1}a). The geometric realization $\cellcomplex_{l,r}^*$ of
$\cellcomplex^*$ is achieved by considering the geodesic complex
dual to $\cellcomplex$ whose vertices are the centers $\{ O_f \, |
\, f \in F \}$ of the face circles of the circle pattern on
$S_{l,r}$. This geometrization of $\cellcomplex^*$ is in fact the
$r-$weighted Voronoi diagram on $S_{l,r}$ for the finite set of
points $V$. Hence, the dual geodesic edges are segments lying on
the radical axis of pairs of vertex circles connected by an edge
of $\cellcomplex_{l,r}$. Moreover, the dual faces are the Voronoi
cells $\{W_{l,r}(i) \, | \, i \in V \}$, which means they are
convex geodesic polygons on $S_{l,r}$ with possibly a cone
singularity in their interior. This realization of
$\cellcomplex^*$ is only possible because of the local Delaunay
property of $\cellcomplex_{l,r}$. That is why the one-skeleton of
$\cellcomplex^*_{l,r}$ is an embedded geodesic graph on $S_{l,r}$.
Consequently, the triangulation $\hat{\Triang}$ can also be
realized in a natural way as a geodesic triangulation
$\hat{\Triang}_{l,r}$ on $S_{l,r}$. The vertices of
$\hat{\Triang}_{l,r}$ consist of the usual vertices of
$\cellcomplex_{l,r}$ together with the vertices of
$\cellcomplex^*_{l,r}$ on $S_{l,r}$. Then the edges of
$\hat{\Triang}_{l,r}$ are the geodesic dual edges (the edges of
$\cellcomplex^*_{l,r}$) as well as the geodesic segments
connecting the center of each face circle with all the vertices of
the face, lying inside that face. In section
\ref{Sec_angle_data_polytope} we called these \emph{corner edges}.

To fix some notation, any geodesic triangle from
$\hat{\Triang}_{l,r}$ can be denoted by $\triangle vO_fO_{f'}$
(e.g. figure \ref{Fig4}b point 1), where $f$ and $f'$ are two
decorated polygons from $\cellcomplex_{l,r}$ sharing a common edge
$uv$. The points $O_f$ and $O_{f'}$ are the centers of the face
circles of $f$ and $f'$ respectively. Furthermore, the two face
circles of $f$ and $f'$ intersect at the two points $P_u^{uv}$ and
$P_v^{uv}$ lying on the segment $uv$. Compare to corollary
\ref{Cor_two_adjacent_triangles} and definition
\ref{Def_angle_between_2_face_circles}. Let $\Omega \subset S$ be
an arbitrary admissible domain according to definition
\ref{Def_admissible_domain}. Then, since the topological
triangulation $\hat{\Triang}$ corresponds to the geodesic
triangulation $\hat{\Triang}_{l,r}$ constructed above, the
admissible domain $\Omega$ becomes a geometric domain
$\Omega_{l,r} \subset S_{l,r}$ with piecewise geodesic boundary
consisting entirely of geodesic edges of $\hat{\Triang}_{l,r}$. In
other words, $\Omega_{l,r}$ is the open interior of a union of
geodesic triangles of type $\triangle v O_f O_{f'}$ satisfying
definition \ref{Def_admissible_domain}. As a special case, the
geometrization of the open star $\text{OStar}(k)$ of a vertex $k
\in V$ is the open interior of the Voronoi cell $W_{l,r}(k)$.

To incorporate in our proof the statement of corollary
\ref{Cor_polytopes_general_admissible_domains} we are going to
work mostly with admissible domains of $(S,\cellcomplex)$ in this
section. Denote by $\widetilde{\polytope}$ the polytope from
theorem \ref{Thm_description_of_polytopes} constructed for the
collection of admissible domains, and by $\polytope$ the polytope
constructed from the collection of strict admissible domains
(corollary \ref{Cor_polytopes_general_admissible_domains}). If one
compares definition \ref{Def_admissible_domain} to definition
\ref{Def_strict_admissible_domain}, one sees that the collection
of all admissible domains of $(S,\cellcomplex)$ contains the
collection of all strict admissible domains of $(S,\cellcomplex)$.
Therefore, $\widetilde{\polytope} \subseteq \polytope$.
Furthermore, $\widetilde{\polytope}$ is an open subset of
$\polytope$ because the additional conditions that define
$\tilde{\polytope}$ are only strict inequalities.

\begin{lem} \label{Lem_necessary_conditions_general}
$\PhiC(\CPC) \subseteq \widetilde{\polytope} \subseteq \polytope$.
More precisely, $\PhiC(\CPC)$ is an open subset of both polytopes
$\widetilde{\polytope}$ and  $\polytope$.
\end{lem}

\begin{proof}
Let us fix an orientation on he boundary $\partial\Omega_{l,r}$.
Then the edges on it are also oriented. There could be two types
of (unoriented) edges on $\partial \Omega_{l,r}$: (i) dual edges
$O_fO_{f'} = uv^*$ of $\cellcomplex^*_{l,r}$ and (ii) corner edges
$v O_f$ of $\hat{\Triang}_{l,r}$. As usual we have assumed that
$f$ and $f'$ are two decorated polygons from $\cellcomplex_{l,r}$
that share a common edge $uv$ of $\cellcomplex_{l,r}$. In
particular, $v$ is a vertex of $f$. From now on we let $\epsilon =
0$ whenever $\PP = \Euclideanplane$ and $\epsilon = 1$ whenever
$\PP = \hyperbolicplane$.

\emph{Case 1 (dual edges).} At first, let us focus on a dual edge
$O_fO_{f'} = uv^*$ lying on $\partial\Omega_{l,r}$ (refer to
figure \ref{Fig4}b point 1). Following its orientation, we assume
that $O_f$ comes before $O_{f'}$. Then there is a geodesic
triangle $\triangle O_fO_{f'}P_{v}^{uv}$ lying inside the closure
of $\Omega_{l,r}$. Let angles $$\measuredangle O_{f'} O_f
P_v^{uv}= \varphi^{(1)}_f \,\,\,\,\, \text{and} \, \,\,\,\,
\measuredangle P_v^{uv} O_{f'} O_{f} = \varphi^{(2)}_{f'}.$$
Moreover, it is not difficult to see that $\measuredangle O_f
P_v^{uv} O_{f'} = \pi - \theta_{uv}$. All three angles
$\varphi^{(1)}_{f}, \varphi^{(2)}_{f'}$ and $\theta_{uv}$ belong
to the interval $[0, \pi)$. Furthermore, $\pi -
\big(\varphi^{(1)}_{f} + \varphi^{(2)}_{f'}\big) = \pi -
\theta_{uv} + \epsilon \text{Area}_{uv}$, where $\text{Area}_{uv}$
is the area of the triangle $\triangle O_fO_{f'}P_{v}^{uv}$. There
are two special cases. The first one is $P_v^{uv} \equiv P_u^{uv}
\in O_fO_{f'}$, possible exactly when $uv \in E_0$. Then
$\theta_{uv} = \varphi^{(1)}_{f} = \varphi^{(2)}_{f'} = 0$. The
second case is when $P_v^{uv} \equiv v,$ possible exactly when $v
\in V_0$, i.e. $r_v=0$.

\emph{Case 2 (corner edges).} As a next step, let us focus on a
corner edge $v O_{f'}$ lying on $\partial\Omega_{l,r}$. This means
that $v$ is a vertex of $\cellcomplex_{l,r}$ lying on
$\partial\Omega_{l,r}$ and that there are always two consecutive
corner edges on $\partial\Omega_{l,r}$ adjacent to $v$ (e.g.
figure \ref{Fig4}b points 2.1 and 2.2). Following the boundary's
orientation, these are the edges $O_f v$ and $v O_{f'}$, i.e.
$O_f$ is the first point, then comes $v$ and finally $O_{f'}$.
Denote by $W_{\Omega}(v)$ the closure of the intersection
$W_{l,r}(v) \cap \Omega_{l,r}$. Then $W_{\Omega}(v)$ is a geodesic
polygon on $S_{l,r}$ so we can think that it is isometrically
developed in the plane $\PP$. There are two situations we are
going to consider.


\emph{Situation 2.1} 
Assume that $\measuredangle O_{f'} v O_f < \pi,$ as an angle
measured inside $W_{\Omega}(v)$ (see figure \ref{Fig4}b point
2.1). Then $W_{\Omega}(v)$ is a convex polygon in $\PP$ with at
least four vertices. Let $O_{f_0}$ be any vertex different from
$v, O_f$ and $O_{f'}$. Focus on the convex quadrilateral
$vO_{f'}O_{f_0}O_f \subset W_{\Omega}(v)$. Let
$$\measuredangle v O_f O_{f_0} = \varphi^{(1)}_f, \,\,\, \measuredangle
O_{f_0} v O_f = \varphi^{(2)}_v, \,\,\, \measuredangle O_{f'} v
O_{f_0} = \varphi^{(1)}_v \,\,\,\, \text{and} \,\,\,\,
\measuredangle O_{f_0} O_{f'} v = \varphi^{(2)}_{f'}.$$ By the
convexity of $vO_{f'}O_{f_0}O_f$, there exists $\gamma_v \in [0,
\pi)$ such that $\measuredangle O_fO_{f_0}O_{f'} = \pi -
\gamma_v$. Moreover, $\varphi^{(1)}_{f} + \varphi^{(2)}_{v} +
\varphi^{(1)}_{v} + \varphi^{(2)}_{f'} +   \pi - \gamma_v = 2\pi -
\epsilon \text{Area}_v$, where $\text{Area}_v$ is the area of the
quadrilateral $vO_{f'}O_{f_0}O_f$. In addition to that, we set
$\delta_v = 0.$

\emph{Situation 2.2.} Assume that $\measuredangle O_{f'} v O_f
\geq \pi,$ as an angle measured inside $W_{\Omega}(v)$ (see figure
\ref{Fig4}b point 2.2). Let us start with the assumption that $v
\in V_1$, i.e. $r_v > 0$. Then the face circle centered at $O_f$
and the vertex circle centered at $v$ intersect at a unique point,
denoted by $T_{v,f}$, lying inside $W_{\Omega}(v)$ (the second one
is outside $W_{\Omega}(v)$, see figure \ref{Fig4}b point 2.2).
Analogously, the face circle centered at $O_{f'}$ and the vertex
circle centered at $v$ intersect at a unique point, denoted by
$T_{v,f'}$, lying inside $W_{\Omega}(v)$. Then $\measuredangle O_f
T_{v,f} v = \measuredangle v T_{v,f'} O_{f'} = \pi/2$. Let
$$\measuredangle v O_f T_{v,f} = \varphi^{(1)}_f, \,\,\,
\measuredangle T_{v,f} v O_f = \varphi^{(2)}_v, \,\,\,
\measuredangle O_{f'} v T_{v,f'} = \varphi^{(1)}_v \,\,\,
\text{and} \,\,\, \measuredangle T_{v,f'} O_{f'} v =
\varphi^{(2)}_{f'}.$$ Furthermore, $\varphi^{(1)}_f +
\varphi^{(2)}_v   + \epsilon\text{Area}_{v,f} = \varphi^{(1)}_{v}
+ \varphi^{(2)}_{f'} + \epsilon\text{Area}_{v,f'} = \pi/2$, where
$\text{Area}_{v,f}$ and $\text{Area}_{v,f'}$ are the areas of the
triangles $\triangle T_{v,f} O_f v$ and $\triangle T_{v,f'} v
O_{f'}$. In addition to that let $\text{Area}_v =
\text{Area}_{v,f} + \text{Area}_{v,f'}$. The fact that
$\measuredangle O_{f'} v O_f \geq \pi$ implies that the interiors
of the two triangles $\triangle T_{v,f} O_f v$ and $\triangle
T_{v,f'} v O_{f'}$ are disjoint, so $\measuredangle T_{v,f'} v
T_{v,f} = \delta_v \in [0,2\pi)$. When the vertex $v$ has radius
$r_v=0$, we can simply do as before, thinking that the two
triangles $\triangle T_{v,f} O_f v$ and $\triangle T_{v,f'} v
O_{f'}$ are degenerate so that $T_{v,f} \equiv v$ and
$T_{v,f'}\equiv v$. Then $\varphi^{(1)}_f = 0 =
\varphi^{(2)}_{f'}$ and $\varphi^{(2)}_v = \pi/2 =
\varphi^{(1)}_v$. Hence $\delta_v = \measuredangle O_{f'} v O_f -
\pi \in [0,\pi)$.

\emph{Remark.} In \cite{Sch1} a similar argument to the one used
in situation 2.2 has also been assumed to work in situation 2.1
(i.e. when $\measuredangle O_{f'}vO_f < \pi$). However, this is
incorrect, as for $r_v \neq 0$ the interiors of the two
right-angled triangles $\triangle T_{v,f} O_f v$ and $\triangle
T_{v,f'} v O_{f'}$ overlap making the argument in question
impossible to use for estimating the angle at $v$. That is why, in
the current paper, the approach outlined in situation 2.1 has been
introduced (compare 2.1 and 2.2 from figure \ref{Fig4}b).

Up to now we have been able to associate to each dual vertex $O_f$
lying on the boundary $\partial\Omega_{l,r}$ two angles
$\varphi_f^{(1)}$ and $\varphi_f^{(2)}$. These two angles are
constructed in such a way that their sum $\varphi_f^{(1)} +
\varphi_f^{(2)}$ subtracted from the angle at vertex $O_f,$
measured from inside $\Omega_{l,r}$, gives an angle $\delta_f \in
[0,2\pi)$ (see figure \ref{Fig4}b). Consequently, the angle at
each vertex $O_f \in
\partial\Omega_{l,r}$, measured from inside $\Omega_{l,r}$,
is equal to $ \varphi_f^{(1)} + \varphi_f^{(2)} + \delta_{f}$, and
the angle at each vertex $k \in \Omega_{l,r} \cap V$, measured
from inside $\Omega_{l,r}$, is equal to $ \varphi_k^{(1)} +
\varphi_k^{(2)} + \delta_{k}$ (see figure \ref{Fig4}b points 2.1
and 2.2).

Now let us apply  the Gauss-Bonnet formula to the domain
$\Omega_{l,r}$
\begin{align*}
2 \pi \chi(\Omega) =& 2 \pi \chi(\Omega_{l,r}) = \sum_{k \in
\Omega \cap V} \big(2\pi - \Theta_k \big) + \sum_{O_f \in
\partial\Omega_{l,r}} \Big(\pi - \big( \varphi_f^{(1)} + \varphi_f^{(2)}
+ \delta_{f} \big) \Big) \\
&+ \sum_{k \in \partial\Omega \cap V} \Big(\pi - \big(
\varphi_k^{(1)} + \varphi_k^{(2)} + \delta_{k} \big) \Big) -
\epsilon \text{Area}(\Omega_{l,r}).
\end{align*}
Instead of summing over the vertices of $\partial\Omega$, one can
obtain the same result by summing over the edges of
$\partial\Omega$
\begin{align*}
2 \pi \chi(\Omega) =& \sum_{k \in \Omega \cap V} \big(2\pi -
\Theta_k \big) + \sum_{ij^* \subset \partial \Omega} \Big(\pi -
\big( \varphi_f^{(1)} + \varphi_{f'}^{(2)}\big) \Big) -
\sum_{O_f \in \partial\Omega_{l,r}} \delta_f \\
&+ \sum_{k \in \partial\Omega \cap V} \Big(2\pi - \big(
\varphi^{(1)}_{f} + \varphi^{(2)}_{k} + \varphi^{(1)}_{k} +
\varphi^{(2)}_{f'} \big) \Big) - \epsilon
\text{Area}(\Omega_{l,r}),
\end{align*}
where in the second sum $f$ and $f'$ are the two faces that share
the edge $ij$, and in the forth summand $f$ and $f'$ are such that
$O_f k$ and $k O_{f'}$ are the two consecutive corner edges
adjacent to $k$ on $\partial \Omega_{l,r}$. According to case 1
above, $\pi - \big(\varphi^{(1)}_{f} + \varphi^{(2)}_{f'}\big) =
\pi - \theta_{ij} + \epsilon \text{Area}_{ij}$, so the
Gauss-Bonnet formula becomes
\begin{align*}
2 \pi \chi(\Omega) =& \sum_{k \in \Omega \cap V} \big(2\pi -
\Theta_k \big) + \sum_{ij^* \subset \partial \Omega} \big(\pi -
\theta_{ij}\big) + \sum_{ij^* \subset \partial \Omega} \epsilon
\text{Area}_{ij} -
\sum_{O_f \in \partial\Omega_{l,r}} \delta_f \\
&+ \sum_{k \in \partial\Omega \cap V} \Big(2\pi - \big(
\varphi^{(1)}_{f} + \varphi^{(2)}_{k} + \varphi^{(1)}_{k} +
\varphi^{(2)}_{f'} + \delta_k \big) \Big) - \epsilon
\text{Area}(\Omega_{l,r}),
\end{align*}
We would like to rewrite the summands of the fifth sum from the
formula above, which all fall into case 2. In situation 2.1 we
have $\delta_k=0$. Hence $2\pi - \big( \varphi^{(1)}_{f} +
\varphi^{(2)}_{k} + \varphi^{(1)}_{k} + \varphi^{(2)}_{f'} \big) =
\pi + \epsilon \text{Area}_k =  \pi - \gamma_k + \epsilon
\text{Area}_k,$ where $\gamma_k=0$. In situation 2.2 again
$\delta_k=0$ and thus $2\pi - \big( \varphi^{(1)}_{f} +
\varphi^{(2)}_{k} + \varphi^{(1)}_{k} + \varphi^{(2)}_{f'} \big) =
\pi - \gamma_k + \epsilon \text{Area}_k$, where $\gamma_k \in
[0,\pi)$. In situation 2.3 we have this time $\delta_k \in
[0,\pi)$. Hence $2\pi - \big( \varphi^{(1)}_{f} +
\varphi^{(2)}_{k} + \varphi^{(1)}_{k} + \varphi^{(2)}_{f'} +
\delta_k \big) = 2\pi - \big( \pi/2 - \epsilon \text{Area}_{v,f} +
\pi/2 - \epsilon \text{Area}_{v,f'} + \delta_k \big) = \pi -
\gamma_k + \epsilon \text{Area}_k$, where $\gamma_k = \delta_k$.
As a result, we obtain the equalities
\begin{align*}
2 \pi \chi(\Omega) =& \sum_{k \in \Omega \cap V} \big(2\pi -
\Theta_k \big) + \sum_{ij^* \subset \partial \Omega} \big(\pi -
\theta_{ij}\big) + \sum_{ij^* \subset \partial \Omega} \epsilon
\text{Area}_{ij} -
\sum_{O_f \in \partial\Omega_{l,r}} \delta_f \\
&\phantom{\sum_{k \in \Omega \cap V} \big(2\pi - \Theta_k \big)} +
\sum_{k \in \partial\Omega \cap V} \big(\pi - \gamma_k + \epsilon
\text{Area}_k \big)  - \epsilon
\text{Area}(\Omega_{l,r})\\
=& \sum_{k \in \Omega \cap V} \big(2\pi - \Theta_k \big) +
\sum_{ij^* \subset \partial \Omega} \big(\pi - \theta_{ij}\big) +
\sum_{ij^* \subset \partial \Omega} \epsilon \text{Area}_{ij} -
\sum_{O_f \in \partial\Omega_{l,r}} \delta_f \\
&\phantom{\sum_{k \in \Omega \cap V} \big(2\pi -\big)} + \pi
|\partial\Omega\cap V| - \sum_{k \in \partial\Omega \cap V}
 \gamma_k + \sum_{k \in \partial\Omega \cap V}  \epsilon \text{Area}_k  - \epsilon
\text{Area}(\Omega_{l,r}) \\
=& \sum_{k \in \Omega \cap V} \big(2\pi - \Theta_k \big) +
\sum_{ij^* \subset \partial \Omega} \big(\pi - \theta_{ij}\big) +
\pi |\partial\Omega\cap V|  -
\sum_{O_f \in \partial\Omega_{l,r}} \delta_f \\
&\phantom{\sum_{k \in \Omega \cap V} \big(\big)} - \sum_{k \in
\partial\Omega \cap V}
 \gamma_k - \epsilon \Big(
\text{Area}(\Omega_{l,r}) - \sum_{k \in \partial\Omega \cap V}
\text{Area}_k  - \sum_{ij^* \subset \partial \Omega}
\text{Area}_{ij} \Big)
\end{align*}
Observe that $\delta_f \geq 0$ for all $O_{f} \in \partial
\Omega_{l,r}$ and $\gamma_{k} \geq 0$ for all $k \in
\partial\Omega \cap V$. Moreover, the summands $\text{Area}_k$ and
$\text{Area}_{ij}$ are the areas of a collection of polygonal
subregions of the admissible domain $\Omega_{l,r}$ (actually a
bunch of triangles) with disjoint interiors. Therefore
$\text{Area}(\Omega_{l,r}) - \sum_{k \in \partial\Omega \cap V}
\text{Area}_k  - \sum_{ij^* \subset \partial \Omega}
\text{Area}_{ij} \geq 0$. Now observe that the presence of a
corner edge $vO_f$ on the boundary of $\Omega_{l,r}$ will always
give rise to at least one $\delta_f > 0$. Moreover, the presence
of a triangle from $\hat{\Triang}_{l,r}$ with no edges on the
boundary of $\Omega_{l,r}$ and whose interior is contained in
$\Omega_{l,r}$ also gives rise to $\delta_f > 0$. In particular,
such a triangle with no edges on the boundary exists in $\Omega$
if the set $\Omega \cap V$ has more than one element.
Consequently, the only time when all $\delta_f = \delta_k = 0$ on
the boundary $\partial\Omega_{l,r}$ is when $\Omega =
\text{OStar}(v)$ for some vertex $v \in V_0$. In this case
$\Omega_{l,r} = W_{l,r}(v)$ with $r_v = 0$ and $$2\pi \chi(\Omega)
= 2\pi = \big(2\pi - \Theta_v\big) + \sum_{iv^* \subset
\partial\text{OStar}(v)} \big(\pi - \theta_{iv}\big).$$ For all
the other cases of $\Omega,$ even when $\Omega = \text{OStar}(v)$
but $v \in V_1$, there will always be at least one $\delta_k > 0$.
Therefore in all these cases
\begin{align*}
2 \pi \chi(\Omega) < \sum_{k \in \Omega \cap V} \big(2\pi -
\Theta_k \big) + \sum_{ij^* \subset \partial \Omega} \big(\pi -
\theta_{ij}\big) + \pi |\partial\Omega\cap V|.
\end{align*}
Thus, we have derived conditions E2, H2, E4 and H4 from theorem
\ref{Thm_description_of_polytopes}, while E1 and H1 are true by
construction. Conditions E3 and H3 follow directly from the global
Gauss-Bonnet formula
$$2\pi\chi(S) = \sum_{k \in V} \big(2\pi -
\Theta_k \big) - \epsilon \text{Area}(S_{l,r}).$$
\end{proof}

\section{Sufficient conditions for existence of circle patterns}
\label{Sec_sufficient_conditions}

This section is devoted to the proof of the claim
$\Phi_{\cellcomplex}(\CPC) = \polytope$. Confirming the latter
fact also completes the proof of the two main theorems of this
article, namely theorem \ref{Thm_main} and
\ref{Thm_description_of_polytopes}, including corollary
\ref{Cor_polytopes_general_admissible_domains}. As we commented in
the previous section, it is enough to proof the following

\begin{lem} \label{Lem_closeness_of_image}
The image $\PhiC(\CPC) \subseteq \polytope$ is a relatively closed
subset of $\polytope$. Consequently, $\PhiC(\CPC)
 = \widetilde{\polytope} = \polytope$.
\end{lem}

\begin{proof}
Recall that the inclusion $\PhiC(\CPC) \subseteq \polytope$ was
established in lemma \ref{Lem_necessary_conditions_general}.
According to one definition of a relatively closed subset of an
open set in a metric space, $\PhiC(\CPC)$ is a relatively closed
subset of $\polytope$ exactly when for any sequence
$\big\{(\theta^{(n)},\Theta^{(n)})\big\}_{n=1}^{\infty}$ in
$\PhiC(\CPC)$ that converges in $\polytope$, the limit
$(\theta^{\star},\Theta^{\star})=\lim_{n \to \infty}
(\theta^{(n)},\Theta^{(n)})$ in fact belongs to $\PhiC(\CPC)$.
Since $\PhiC$ is a diffeomorphism, there exists a well-defined
sequence of circle patterns given by
$\big\{(a^{(n)},b^{(n)})\big\}_{n=1}^{\infty}$ such that
$(a^{(n)},b^{(n)}) = \PhiC^{-1}(\theta^{(n)},\Theta^{(n)}) \, \in
\, \CPC$ for all $n \in \naturals$. Our goal is to prove that
$\big\{(a^{(n)},b^{(n)})\big\}_{n=1}^{\infty}$ has a limit
$(a^{\star},b^{\star})$ that belongs to $\CPC$. Consequently, that
would mean
\begin{align*}
\PhiC(a^{\star},b^{\star}) &= \PhiC\big( \lim_{n \to \infty}
(a^{(n)}, b^{(n)})\big) = \lim_{n \to \infty}  \PhiC\big(a^{(n)},
b^{(n)}\big)\\
&= \lim_{n \to \infty} (\theta^{(n)},\Theta^{(n)}) =
(\theta^{\star},\Theta^{\star}) \,\, \in \,\,\, \PhiC(\CPC).
\end{align*}

We begin with fixing some terminology and notations. To simplify
the exposition, most of the time, we will use the term
\emph{convergence} in the sense of \emph{convergence after
selecting a subsequence}. The notation $\text{cl}(U)$ will mean
the closure of a set $U$. For convenience we switch to edge-length
and vertex radii coordinates $(l^{(n)},r^{(n)}) =
\Psi(a^{(n)},b^{(n)})$. As usual $ij \in E$ is an arbitrary edge
of $\cellcomplex$ while $f$ and $f' \in F$ are the two faces of
$\cellcomplex$ that share $ij$ as a common edge. We denote by
$O_f$ and $O_{f'} \in V^*$ the corresponding pair of dual vertices
from $\cellcomplex^*$. 
In general, we will use the same labelling notations for the
objects from both the topological cell complexes $\cellcomplex,
\cellcomplex^*, \Triang$
 and $\hat{\Triang}$ and their geometric counterparts
 $\cellcomplex_{l^{(n)},r^{(n)}}, \cellcomplex^*_{l^{(n)},r^{(n)}},
 \Triang_{l^{(n)},r^{(n)}}$
 and $\hat{\Triang}_{l^{(n)},r^{(n)}}$. 
Recall the definition of the pair of points $P^{ij}_i(n)$ and
$P^{ij}_j(n)$ lying on the geodesic edge $ij$ of
$\cellcomplex_{l^{(n)},r^{(n)}}$, 
which were defined in corollary \ref{Cor_two_adjacent_triangles}
and in definition \ref{Def_angle_between_2_face_circles}. Let
Denote by $R_f^{(n)}$ the radius of the face circle $c^{(n)}_f$ of
the decorated polygon $f$ from the hyper-ideal circle pattern
represented by $(l^{(n)},r^{(n)})$. Denote by $\lambda^{(n)}_{if}$
the length of the geodesic edge $iO_{f}$ from
$\hat{\Triang}_{l^{(n)},r^{(n)}}$ and by $h_{ij}^{(n)}$ the length
of the geodesic edge $O_fO_{f'}$ of
$\cellcomplex^{*}_{l^{(n)},r^{(n)}}$. It follows directly from the
geometry of the two adjacent decorated polygons $f$ and $f'$ that
$\lambda^{(n)}_{if} = \lambda\big(R^{(n)}_f, r^{(n)}_i\big)$ and
$h^{(n)}_{ff'} = h\big(R^{(n)}_f, R^{(n)}_{f'},
\theta^{(n)}_{ij}\big)$, where
\begin{align}
\lambda\big(R^{(n)}_f, r^{(n)}_i\big) &=
\sqrt{\big(R^{(n)}_f\big)^2 - \big(r^{(n)}_i\big)^2} \,\,\,\,
\text{
when } \,\, \PP=\Euclideanplane  \label{Formula_vertex_edge_Eucl} \\
\lambda\big(R^{(n)}_f, r^{(n)}_i\big) &=
\cosh^{-1}\Big(\cosh{R^{(n)}_f}\cosh{r^{(n)}_i}\Big)
 \,\,\,\, \text{ when } \,\,
 \PP=\hyperbolicplane \label{Formula_vertex_edge_hyp}\\
h\big(R^{(n)}_f, R^{(n)}_{f'}, \theta^{(n)}_{ij}\big) &=
\sqrt{\big(R^{(n)}_f\big)^2 + \big(R^{(n)}_{f'}\big)^2 + 2
R^{(n)}_f R^{(n)}_{f'} \cos{\theta^{(n)}_{ij}} } \,\,\,\, \text{
when } \,\, \PP=\Euclideanplane \label{Formula_dual_edge_Eucl}\\
h\big(R^{(n)}_f, R^{(n)}_{f'}, \theta^{(n)}_{ij}\big) &=
\cosh^{-1} \Big(\cosh{R^{(n)}_{f}} \cosh{R^{(n)}_{f'}} +
\sinh{R^{(n)}_{f}} \sinh{R^{(n)}_{f'}}
\cos{\theta^{(n)}_{ij}}\Big) \nonumber\\
& \phantom{Hhhh = \cosh^{-1} \cosh{R^{(n)}_{f}}
\cosh{R^{(n)}_{f'}} - \sinh{R^{(n)}_{f}}} \,\, \text{ when } \,\,
\PP=\hyperbolicplane \label{Formula_dual_edge_hyp}
\end{align}
Notice that in the Euclidean case the sequence
$\big\{(l^{(n)},r^{(n)}) \big\}_{n=1}^{\infty}$ is bounded,
because it belong to the bounded set $\ER_1$ (see section
\ref{Sec_Generalized_space_of_patterns}). Despite the
unboundedness of $\ER_1=\ER$ in the hyperbolic case, the sequence
$\big\{ (l^{(n)},r^{(n)}) \big\}_{n=1}^{\infty}$ is still bounded.
The latter fact however heavily relies on the assumption that the
limit $(\theta^{\star},\Theta^{\star})$ lies in the polytope
 $\polytope$, so all $\theta_{ij}^{\star} \in [0,\pi)$ and all
 $\Theta_k^{\star} > 0$, for $ij \in E$ and $k \in V$. A short outline of how one can confirm
 boundedness goes as follows. The unboundedness of $(l^{(n)},r^{(n)})$ implies that at least one single-component
 sequence  $l^{(n)}_{ij} > 0$, for $ij \in E$, is
unbounded. Hence there will be a subsequence of
$\{l_{ij}^{(n)}\}_{n=1}^{\infty}$ that diverges to $+\infty$,
making one vertex $i \in V$ on $S_{l^{(n)},r^{(n)}}$ degenerate to
a cusp, i.e. roughly speaking $i$ will approach the ideal boundary
of the hyperbolic plane (here is where one applies the fact that
$\theta_{ij}^{\star} \in [0,\pi)$). But that forces the
corresponding sequence of cone angles
$\{\Theta_{ij}^{(n)}\}_{n=1}^{\infty}$ to converge to zero, which
means that $\Theta^{\star}_k=0$. The latter is a contradiction
with the assumption that $\Theta^{\star}_k > 0$ by the definition
of $\polytope$. Thus, one concludes that there exists a positive
real number $L > 0$ with the property that both $r_k^{(n)}$ and
$l_{ij}^{(n)} \in [0, L)$ for all $k \in V,$ all $ij \in E_T$ and
all $n \in \naturals$. Therefore, one can safely assume that
$\big\{ (l^{(n)},r^{(n)})\big\}_{n=1}^{\infty}$ converges to
$(l^{\star},r^{\star}) \in \text{cl}(\ER_1)$.

\begin{claim} \label{Claim_boundedness_of_Rf}
For any $n \in \naturals$ and $f \in F, \,\, R^{(n)}_f \, \in \,
(0,L)$.
\end{claim}

\noindent \emph{Proof of claim \ref{Claim_boundedness_of_Rf}}.
Recall that the centers $\big\{ O_f \, | \, f \in F\big\}$ of the
face circles $\big\{ c^{(n)}_f \, | \, f \in F\big\}$ of the
hyper-ideal circle pattern on $S_{l^{(n)},r^{(n)}}$ are the
vertices of the dual geodesic complex
$\cellcomplex^*_{l^{(n)},r^{(n)}}$, which also happens to be the
$r-$weighted Voronoi diagram on $S_{l^{(n)},r^{(n)}}$. Hence each
$O_f$ lies on $S_{l^{(n)},r^{(n)}} \setminus V$. Fix one arbitrary
$O_f$. Then there exists a decorated triangle $\Delta=ijk$ from
the geodesic subtriangulation $\Triang_{l^{(n)},r^{(n)}}$ of
$\cellcomplex_{l^{(n)},r^{(n)}}$ for which $O_f \in
\text{cl}(\Delta)$. Recall that $O_f$ is the center of the
face-circle $c_{f}^{(n)}$ with radius $R_f^{(n)}$. Note that
$c^{(n)}_f$ is not necessarily the face circle of $\Delta$! The
edge-lengths of $\Delta$ are $l^{(n)}_{ij}, l^{(n)}_{jk}$ and
$l^{(n)}_{ki}$. For convenience develop $\Delta$ together with
$c^{(n)}_f$ and $O_f$ on the plane $\PP$. By the Delaunay property
of $c^{(n)}_f$, none of the three vertices $i, j$ and $k$ lies in
the interior of $c^{(n)}_f$. Therefore $R_f^{(n)} \leq
\lambda^{(n)}_{if}$ since the vertex $i$ is not inside
$c^{(n)}_f$. Furthermore, for the geodesic triangle $\Delta$ the
inequality $\lambda^{(n)}_{if} \leq \max{\big\{l^{(n)}_{ij},
l^{(n)}_{jk}, l^{(n)}_{ki}\big\}} < L$ holds, because $O_f \in
\text{cl}(\Delta)$. Consequently $R_f^{(n)} < L$.
\hfill\(\triangle\)

\smallskip

Now let us consider the augmented sequence $\big\{(l^{(n)},
r^{(n)}, R^{(n)})\big\}_{n=1}^{\infty}$. It has been established
that $(l^{(n)}, r^{(n)}, R^{(n)}) \in [0, L)^{E \cup V \cup F}$
for all $n \in \naturals$, so this sequence is convergent and its
limit is $(l^{\star}, r^{\star}, R^{\star}) \in \text{cl}(\ER_1)
\times [0, L]^F$.

\begin{claim} \label{Claim_Rf_not_zero}
There exists $\e_0 > 0$ such that for any $n \in \naturals$, any
$k \in V_1$ and any $f \in F$, both $r^{(n)}_k \in (\e_0, L)$ and
$R^{(n)}_f \in (\e_0, L)$.
\end{claim}

\noindent \emph{Proof of claim \ref{Claim_Rf_not_zero}}. It has
been already established in claim \ref{Claim_boundedness_of_Rf}
that $0 < r^{(n)}_k < L$ and $0 < R^{(n)}_f < L$ for all $k \in
V_1$, all $f \in F$ and all $n \in \naturals$. Let us focus our
attention on the triangulation $\hat{\Triang} = (\hat{V}, \hat{E},
\hat{F})$ and its geodesic realization
$\hat{\Triang}_{l^{(n)},r^{(n)}}$. For any $\hat{v} \in \hat{V}
\setminus V_0$ there exists a genuine circle $c^{(n)}_{\hat{v}}$
centered at $\hat{v}$. If $\hat{v}=k \in V_1$ then $c^{(n)}_k$ is
a vertex circle of radius $r^{(n)}_k > 0$ and if $\hat{v}=O_f \in
V^*$ then $c^{(n)}_f$ is a face circle of radius $R^{(n)}_f > 0$.
Denote by $\hat{V}_0$ the set of all vertices of $\hat{\Triang}$
for which the limit $r^{\star}_k = 0$ if the vertex is from $V,$
and the limit $R^{\star}_f=0$ if the vertex is from $V^*$. Notice
that $V_0 \subseteq \hat{V}_0$.

\smallskip

We claim that $\hat{V}_0 \neq \hat{V}$.

\smallskip

Indeed, for any edge $ij \in E_1$ and any $n \in \naturals$ the
triangle inequality for $\triangle ijO_f$ leads to $0 <
l^{(n)}_{ij} \leq \lambda_{if}^{(n)} + \lambda_{jf}^{(n)} =
\lambda(R^{(n)}_f,r^{(n)}_i) + \lambda(R^{(n)}_f,r^{(n)}_j)$ where
$\lambda$ is the real analytic function defined either by formula
(\ref{Formula_vertex_edge_Eucl}) or formula
(\ref{Formula_vertex_edge_hyp}), while $r^{(n)}_i \geq0, \,
r^{(n)}_j \geq 0$ and $R^{(n)} > 0$. Furthermore, for any $ij \in
E_0$ the inequality simplifies to $0 < l^{(n)}_{ij} = r^{(n)}_i +
r^{(n)}_j$, where $r^{(n)}_i > 0, \, r^{(n)}_j > 0$. Therefore, if
we assume that $\lim _{n \to \infty} R^{(n)}_f = \lim_{n \to
\infty} r^{(n)}_i = 0$ for all $f \in F$ and $i \in V$, then $\lim
l^{(n)}_{ij} = 0$ for all $ij \in E$ (and thus for all $ij \in
E_T$). First, in the Euclidean case $\sum_{ij \in E_T}
l^{(n)}_{ij} = 1$ for every $n \in \naturals$. Hence the limit of
the latter sum is $\sum_{ij \in E_T} l^{\star}_{ij} = 1$ which
clearly contradicts the earlier conclusion $\lim l^{(n)}_{ij} =
l^{\star}_{ij}= 0$ for all $ij \in E_T$. Second, in the hyperbolic
case the global Gauss-Bonnet formula for the area of
$S_{l^{(n)},r^{(n)}}$ is
$$\text{Area}(S_{l^{(n)},r^{(n)}}) = \sum_{k \in V} \big(2\pi -
\Theta_k^{(n)} \big) -  2\pi\chi(S).$$ Therefore, by the
assumption that $(\theta^{\star}, \Theta^{\star}) \in \polytope^h$
the limit of the positive sequence of surface areas
$\text{Area}(S_{l^{(n)},r^{(n)}})$, $\, n \in \naturals$ exists
and is also positive (non-zero), which clearly cannot happen if
$l_{ij}^{\star}=0$ for all $ij \in E_T$.

Define the open domain $N(\hat{V}_0) = \cup \big\{ \,
\text{OStar}(\hat{v}) \, \, | \,\, \hat{v} \in \hat{V}_0\big\}$.
Take a connected component of $N(\hat{V}_0)$ and call it $\Omega$.
By construction $\partial\Omega$ is composed of edges from
$\hat{\Triang}$. Also, $\partial\Omega \cap V_0 = \varnothing$.

Assume $V_0 \neq \hat{V}_0$. Then without loss of generality
$\Omega \cap V_1 \neq \varnothing$ or $\Omega \cap V^* \neq
\varnothing$. Our goal is to show that this assumption contradicts
the assumption that $(\theta^{\star}, \Theta^{\star}) \in
\polytope$. We can easily convince ourselves that if the radius of
a face circle of a decorated polygon $f$ converges to zero, the
radii of all of its vertex circles, except for possibly two, also
converge to zero. Therefore, $\Omega \cap V^* \neq \varnothing$
implies $\Omega \cap V \neq \varnothing$. Also, $\Omega \neq
\varnothing$ as well as $\Omega \neq S$ since $N(\hat{V}_0) \neq
S$, due to $\hat{V}_0
\neq \hat{V}$ as proved before. 
All these facts show that the open set $\Omega$ is a strict
admissible domain in the sense of definition
\ref{Def_strict_admissible_domain}.


As usual, we denote by $\Omega_{l^{(n)},r^{(n)}}$ the geometric
realization of $\Omega$ on the surface $S_{l^{(n)},r^{(n)}}$ via
the geodesic representation $\hat{\Triang}_{l^{(n)},r^{(n)}}$ of
the topological triangulation $\hat{\Triang}$. We remind the
reader that every face of $\hat{\Triang}_{l^{(n)},r^{(n)}}$ is a
geodesic triangle of type $\triangle iO_fO_{f'}$ with angle
$\measuredangle O_{f} i O_{f'} = \pi - \theta^{(n)}_{ij}$ and
edge-lengths $\lambda^{(n)}_{if},\, \lambda^{(n)}_{if'}, \,
h^{(n)}_{ff'}$. Generically speaking $\triangle iO_fO_{f'}$
corresponds to exactly one point $P^{ij}_i(n)$ (compare with
figure \ref{Fig4}b point 1). Based on that, we distinguish three
types of such geometric triangles.

\smallskip
\noindent \emph{Type 1.} $\, i \in V_1$ and $ij \in E_1$. Then
$P^{ij}_i(n)$ is strictly in the interior of $\triangle
iO_fO_{f'}$ (as shown on figure \ref{Fig4}b point 1) so
$$l_{\PP}\big(iP^{ij}_i(n)\big) < r^{(n)}_i, \,\,\,\,\,\,\,
\measuredangle O_{f'}P^{ij}_i(n)O_f = \pi - \theta_{ij}^{(n)} \,
\in (0, \pi), \,\,\,\,\,\,\, R^{(n)}_f < \lambda^{(n)}_{if},
\,\,\,\,\,\,\, R^{(n)}_{f'} < \lambda^{(n)}_{if'},$$
$$h^{(n)}_{ff'} =h\big(R^{(n)}_f,R^{(n)}_{f'},\theta^{(n)}_{ij}\big);$$

\smallskip
\noindent \emph{Type 2.} $\, i \in V_0$ and $ij \in E_1$. Then
$P^{ij}_i(n) \equiv i$ so
$$l_{\PP}\big(iP^{ij}_i(n)\big)= 0 = r^{(n)}_i, \,\,\,\,\,\,
\measuredangle O_{f'} i O_f = \pi - \theta_{ij}^{(n)} \, \in (0,
\pi), \,\,\,\,\,\,\, R^{(n)}_f = \lambda^{(n)}_{if},
\,\,\,\,\,\,\, R^{(n)}_{f'} = \lambda^{(n)}_{if'},$$
$$h^{(n)}_{ff'} =h\big(R^{(n)}_f,R^{(n)}_{f'},\theta^{(n)}_{ij}\big);$$

\smallskip
\noindent \emph{Type 3.} $\, i, j \in V_1$ and $ij \in E_0$. Then
$P^{ij}_i(n) \equiv H_{ij} = O_fO_{f'} \cap ij$ so
$$l_{\PP}\big(iH_{ij}\big)= r^{(n)}_i, \,\,\,\,\,
P^{ij}_i(n) \in O_{f'}O_f, \,\,\,\,\,\,\, R^{(n)}_f <
\lambda^{(n)}_{if}, \,\,\,\,\,\, R^{(n)}_{f'} <
\lambda^{(n)}_{if'}, \,\,\,\,\,\, h^{(n)}_{ff'} = R^{(n)}_f +
R^{(n)}_{f'}.$$

\smallskip
Our main objective is to explore how the Gauss-Bonnet formula for
$\Omega_{l^{(n)},r^{(n)}}$ behaves when $n \to \infty$. Observe
that $\Omega_{l^{(n)},r^{(n)}}$ is made of triangles $\triangle
iO_fO_{f'}$ that have at least one vertex contained in $\hat{V}_0$
(compare with the depiction of $\Omega$ on figure \ref{Fig3}a).
Depending on how many and which vertices of $\triangle iO_fO_{f'}$
belong to $\hat{V}_0$, there are several cases for $n \to \infty$.

\smallskip
\noindent \emph{Case 1.} Both vertices $O_f$ and $O_{f'} \notin
\hat{V}_0$ but $i \in \hat{V}_0$. Then $\triangle iO_fO_{f'}$
could be of type 1, 2 or 3 and its edge $O_fO_{f'}$ always lies on
the boundary of $\Omega_{l^{(n)},r^{(n)}}$. It is straightforward
to see that when $n \to \infty$ a triangle of type 1 or 2 converge
to a triangle of type 2 ( $ P^{ij}_i(\star) \equiv i$ ), with
angle $\measuredangle O_fiO_{f'} = \pi - \theta^{\star}_{ij} $ and
edge-lengths $R^{\star}_f = \lambda^{\star}_{if}, \,\,\,
R^{\star}_{f'} = \lambda^{\star}_{if'}, \,\,\, h^{\star}_{ff'} =
h\big(R^{\star}_{f}, R^{\star}_{f'}, \theta^{\star}_{ij}\big)$. A
triangle of type 3 converges to a geodesic segment $O_fO_{f'}$ of
length $h^{\star}_{ff'} = R^{\star}_f + R^{\star}_{f'} =
\lambda^{\star}_{if} + \lambda^{\star}_{if'}$ containing the
vertex $i$.

\smallskip
\noindent \emph{Case 2.} Both $i$ and $O_f \notin \hat{V}_0$ but
$O_{f'} \in \hat{V}_0$. Then its edge $iO_f$ always lies on
$\partial\Omega_{l^{(n)},r^{(n)}}$. Furthermore, $\triangle
iO_fO_{f'}$ could be of type 1 or 3 because $r^{\star}_i > 0$ due
to $i \notin \hat{V}_0$ and $R^{\star}_{f'}=0$. Either way, it is
not difficult to observe that when $n \to \infty$ the triangle
$\Delta iO_fO_{f'}$ converges to a right angled triangle with
angles $\measuredangle O_fiO_{f'} = \pi - \theta^{\star}_{ij}$ and
$\measuredangle iO_{f'}O_f = \pi/2$, and edge-lengths
$\lambda^{\star}_{if}, \,\,\, r^{\star}_{i} =
\lambda^{\star}_{if'}, \,\,\, h^{\star}_{ff'} = R^{\star}_{f}$.
Moreover, $P^{ij}_i(\star) \equiv H_{ij}$.

\smallskip
\noindent \emph{Case 3.} Both $i$ and $O_{f'} \in \hat{V}_0$ but
$O_f \notin \hat{V}_0$. Then $\triangle iO_fO_{f'}$ could be of
type 1, 2 or 3. For all three types when $n \to \infty$ the
triangle $\triangle iO_fO_{f'}$ converges to a geodesic segment of
length $R^{\star}_f = \lambda^{\star}_{if} = h^{\star}_{ff'}$
 where $i \equiv O_{f'} \equiv H_{ij} \equiv P^{ij}_i(\star)$.

\smallskip
\noindent \emph{Case 4.} Both $O_f$ and $O_{f'} \in \hat{V}_0$ but
$i \notin \hat{V}_0$. Then $\triangle iO_fO_{f'}$ could be of type
1 or 3. For both types, one can conclude that when $n \to \infty$
the triangle $\Delta iO_fO_{f'}$ converges to a geodesic segment
of length $r^{\star}_i = \lambda^{\star}_{if} =
\lambda^{\star}_{if'}$ where $O_f \equiv O_{f'} \equiv H_{ij}
\equiv P^{ij}_i(\star)$ and $h^{\star}_{ff'} = 0$.

\smallskip
\noindent \emph{Case 5.} All three vertices $i, O_f$ and $O_{f'}
\in \hat{V}_0$. Then $\triangle iO_fO_{f'}$ could be of type 1, 2
or 3. It is straightforward to confirm that when $n \to \infty$
the triangle $\Delta iO_fO_{f'}$ collapses to a point and all
edge-lengths become zero.

\smallskip
Apply the Gauss-Bonnet formula to the geometric domain
$\Omega_{l^{(n)},r^{(n)}}$. By the formulas derived in the proof
of lemma \ref{Lem_necessary_conditions_general}
\begin{align*}
2 \pi \chi(\Omega)  = &\sum_{k \in \Omega \cap V} \big(2\pi -
\Theta_k^{(n)} \big) + \sum_{ij^* \subset
\partial \Omega} \big(\pi - \theta_{ij}^{(n)}\big)  -
\sum_{O_f \in \partial\Omega} \delta_f^{(n)} \\
&+ \sum_{k \in \partial\Omega \cap V} \Big(2\pi - \big(
\varphi^{(1)}_{f}(n) + \varphi^{(2)}_{k}(n) + \varphi^{(1)}_{k}(n)
+
\varphi^{(2)}_{f'}(n) + \delta_k^{(n)} \big) \Big) \\
&+ \sum_{ij^* \subset
\partial \Omega} \epsilon \text{Area}_{ij}^{(n)} - \epsilon
\text{Area}(\Omega_{l^{(n)},r^{(n)}}).
\end{align*}

Take $n \to  \infty$. Then by assumption $\lim_{n \to \infty}
\theta^{(n)}_{ij} = \theta^{\star}_{ij}$ and $\lim_{n \to \infty}
\Theta^{(n)}_{k} = \Theta^{\star}_{k}$. The combination of cases 1
and 3 yields $\lim_{n \to \infty} \delta_f^{(n)} = 0$ for all $O_f
\in \partial \Omega \cap V^*$. Furthermore, the combination of
cases 2 and 4 reveals that for large enough $n \in \naturals$,
every vertex $k \in \partial\Omega \cap V_1$ satisfies situation
2.2 from the proof of lemma \ref{Lem_necessary_conditions_general}
and hence $\delta_k^{(n)} \equiv 0$ for large enough $n$. Then by
cases 2 and 4 $\lim_{n \to \infty} \big(\varphi^{(1)}_{f}(n) +
\varphi^{(2)}_{k}(n)\big) = \lim_{n \to \infty}
\big(\varphi^{(2)}_{f'}(n) + \varphi^{(1)}_{k}(n)\big) = \pi/2$.
Finally, all five cases lead to the conclusion that $$\lim_{n \to
\infty} \sum_{ij^* \subset \partial\Omega} \text{Area}_{ij}^{(n)}
= \lim _{n \to \infty}
\text{Area}\big(\Omega_{l^{(n)},r^{(n)}}\big).$$ Consequently, by
taking $n$ to infinity, we obtain the identity
\begin{align*}
2 \pi \chi(\Omega)  = \sum_{k \in \Omega \cap V} \big(2\pi -
\Theta_k^{\star} \big) + \sum_{ij^* \subset
\partial \Omega} \big(\pi - \theta_{ij}^{\star}\big) + \pi |\partial\Omega \cap V_1|
\end{align*}
However, it is assumed that $(\theta^{\star}, \Theta^{\star}) \in
\polytope$, which implies that for the constructed strict
admissible domain $\Omega$
\begin{align*}
2 \pi \chi(\Omega)  < \sum_{k \in \Omega \cap V} \big(2\pi -
\Theta_k^{\star} \big) + \sum_{ij^* \subset
\partial \Omega} \big(\pi - \theta_{ij}^{\star}\big) + \pi |\partial\Omega \cap V_1|
\end{align*}
This provides the sought out contradiction. \hfill\(\triangle\)

\begin{claim} \label{Claim_pattern_nondegeneracy}
The following statements hold:

\smallskip \noindent \emph{1.} For any edge $ij \in E_T$ from the
triangulation $\Triang$, its edge-length satisfies the inequality
$l^{\star}_{ij}
> 0$.

\smallskip
\noindent \emph{2.} For any triangle $\Delta = ijk \in F_T$ from
the triangulation $\Triang$, the three edge-lengths satisfy the
triangle inequalities $l^{\star}_{uv} < l^{\star}_{vw} +
l^{\star}_{wu}$ where $u \neq v \neq w \in \{i,j,k\}$.

\smallskip
\noindent \emph{3.} For any edge $ij \in E_T \setminus E_0$ from
the triangulation $\Triang$, the edge-length and the vertex radii
satisfy the inequality $l^{\star}_{ij} > r^{\star}_{i} +
r^{\star}_{j}$, where one or both radii are allowed to be zero.
\end{claim}

\noindent \emph{Proof of claim \ref{Claim_pattern_nondegeneracy}.}
Fix an edge $ij \in E_T$. At first, assume $i$ or $j \in V_1$.
Without loss of generality let $i \in V_1$. Then $l_{ij}^{(n)}
\geq r^{(n)}_i + r^{(n)}_j \geq r^{(n)}_i > 0.$ Claim
\ref{Claim_Rf_not_zero} implies that $\lim_{n \to \infty}
r^{(n)}_i = r^{\star}_i
> 0.$ Hence $\lim_{n \to \infty} l^{(n)}_{ij} = l_{ij}^{\star} \geq r^{\star}_i > 0$.
Next, assume that both $i$ and $j \in V_0.$ Consider the
quadrilateral $iO_{f'}jO_f$ which has one pair of edges $iO_f$ and
$jO_f$ of equal length $R^{(n)}_f$, another pair of edges
$iO_{f'}$ and $jO_{f'}$ of length $R^{(n)}_{f'}$ and a pair of
equal angles $\measuredangle O_{f'}iO_f = \measuredangle
O_{f'}jO_f = \pi - \theta^{(n)}_{ij}.$ By claim
\ref{Claim_Rf_not_zero}, the quadrilateral's limit when $n \to
\infty$ is again a quadrilateral with a pair of edges $iO_f$ and
$jO_f$ of equal length $R^{\star}_f>0$, a second pair of edges
$iO_{f'}$ and $jO_{f'}$ of length $R^{\star}_{f'}>0$ and a pair of
equal angles $\measuredangle O_{f'}iO_f = \measuredangle
O_{f'}jO_f = \pi - \theta^{\star}_{ij}<\pi.$ Then $l^{\star}_{ij}$
is the length of the diagonal $ij$ of the limit quadrilateral, so
$l^{\star}_{ij} > 0$.


Take an arbitrary polygonal face $f \in F$ and consider any
triangle $\Delta \in F_T$ such that $\Delta = ijk \subset f$. Then
for all $n \in \naturals$ the triangle inequalities $l^{(n)}_{uv}
< l^{(n)}_{vw} + l^{(n)}_{wu}$ hold, with $u \neq v \neq w \in
\{i,j,k\}$. Therefore $l^{\star}_{uv} \leq l^{\star}_{vw} +
l^{\star}_{wu}$ when $n \to \infty$. Assume that for some triangle
$\Delta = ijk$ contained in a face $f\in F$ one of the three
inequalities becomes $l^{\star}_{ij} = l^{\star}_{jk} +
l^{\star}_{ki}$ when $n \to \infty$. By the preceding paragraph,
$l_{ij}^{\star}
> 0, \, l_{jk}^{\star} > 0, \, l_{ki}^{\star} > 0$. Therefore,
the angle $\measuredangle ikj$ is forced to converge to $\pi$
while the other two angles become zero. Consequently, the sequence
of geodesic triangles $\triangle ijk$ with edge-lengths
$\big(l^{(n)}_{ij}, l_{jk}^{(n)}, l_{ki}^{(n)}\big)$ converges to
a geodesic segment $ij$ of length $l_{ij}^{\star}$ containing the
point $k$ in its interior. But the latter conclusion means that
the sequence of face circle radii $R^{(n)}_f$ has to diverge to
$+\infty$, which contradicts claim \ref{Claim_boundedness_of_Rf}.
Hence, there is a well defined geometric and geodesically
triangulated limit surface $S_{l^{\star},r^{\star}}$.

For any edge $ij \in E_T \setminus E_0$ and any $n \in \naturals$
the inequality $l^{(n)}_{ij}
> r^{(n)}_i + r^{(n)}_j$ holds. Therefore $l^{\star}_{ij}
\geq r^{\star}_i + r^{\star}_j$ when $n \to \infty$. Assume that
$l^{\star}_{ij} = r^{\star}_i + r^{\star}_j$ for some $ij \in
E_T\setminus E_0$. Clearly, both vertex radii cannot be zero
because $l_{ij}^{\star} > 0,$ so either exactly one vertex radius
is zero or both radii are nonzero. Either way, in both cases, the
two face circles (which exist because $S_{l^{\star},r^{\star}}$ is
well defined) on both sides of $ij$ have to be tangent, which
immediately implies that the corresponding angle
$\theta^{\star}_{ij} = 0$, which is not the case.
\hfill\(\triangle\)

\smallskip
Claim \ref{Claim_pattern_nondegeneracy} implies that the limit
$(l^{\star},r^{\star})$ belongs to $\ER_1$ and thus represents a
generalized circle pattern with combinatorics $\Triang$ on a well
defined geometric surface $S_{l^{\star},r^{\star}}$. As usual,
$\Triang_{l^{\star},r^{\star}}$ is the geodesic realization of
$\Triang$ on $S_{l^{\star},r^{\star}}$. Therefore
$$(\theta^{\star},\Theta^{\star}) = \lim_{n\to \infty}
(\theta^{(n)},\Theta^{(n)}) =  \lim_{n\to \infty} \PhiT \circ
\Psi^{-1} (l^{(n)},r^{(n)}) = \PhiT \circ \Psi^{-1}
(l^{\star},r^{\star}).$$ However, $\theta_{ij}^{(n)} = \pi$ for
each $ij \in E_{\pi}$ and all $n \in \naturals$. Hence
$\theta_{ij}^{\star} = \pi$ for all $ij \in E_{\pi}$. On top of
that, by assumption, $\theta^{\star}_{ij} \in (0,\pi)$ for $ij \in
E_1$ and $\theta^{\star}_{ij} = 0$ for $ij \in E_0$. Thus, the
circle pattern represented by $(l^{\star},r^{\star})$ has actual
combinatorics $\cellcomplex$ and is Delaunay. Therefore,
$(a^{\star},b^{\star}) = \Psi^{-1}(l^{\star},r^{\star})$ gives
rise to a hyper-ideal circle pattern on $S$ with combinatorics
$\cellcomplex$, so it belongs to the space $\CPC$ and is the limit
of the sequence $\big\{(a^{(n)},b^{(n)})\big\}_{n=1}^{\infty}$.
Finally, we have concluded the proof of lemma
\ref{Lem_closeness_of_image} and consequently, the proofs of both
theorems \ref{Thm_main} and \ref{Thm_description_of_polytopes}.
\end{proof}


\end{document}